\documentclass{amsart}

\usepackage[margin=2.5cm]{geometry}
\usepackage{mathtools}
\usepackage{amssymb, MnSymbol}
\usepackage{mathrsfs}
\usepackage{color}
\usepackage{hyperref}

\theoremstyle{plain}
\newtheorem{theorem}{Theorem}[section]
\newtheorem{proposition}[theorem]{Proposition}
\newtheorem{lemma}[theorem]{Lemma}
\newtheorem{corollary}[theorem]{Corollary}
\newtheorem*{theoremA}{Theorem A}
\newtheorem*{theoremB}{Theorem B}
\newtheorem*{theoremC}{Theorem C}
\newtheorem*{theoremD}{Theorem D}

\theoremstyle{definition}
\newtheorem{definition}[theorem]{Definition}
\newtheorem{example}[theorem]{Example}

\theoremstyle{remark}
\newtheorem{remark}[theorem]{Remark}

\numberwithin{equation}{section}
\renewcommand{\labelenumi}{\textit{(\roman{enumi})}}

\newcommand{\introd}[1]{\emph{#1}}%\index{#1}}

\let\Oldepsilon\epsilon
\let\Oldvarepsilon\varepsilon
  \renewcommand{\varepsilon}{\Oldepsilon}
  \renewcommand{\epsilon}{\Oldvarepsilon}
\let\Oldphi\phi
\let\Oldvarphi\varphi
  \renewcommand{\varphi}{\Oldphi}
  \renewcommand{\phi}{\Oldvarphi}
\newcommand{\Tau}{\mathrm{T}}
\newcommand{\Nu}{\mathrm{N}} %Index space for nets

\newcommand{\N}{\mathbb{N}}  % Numeri naturali
\newcommand{\Z}{\mathbb{Z}}  % Numeri interi
\newcommand{\C}{\mathbb{C}}  % Numeri complessi
\newcommand{\R}{\mathbb{R}}  % Numeri reali

\renewcommand{\Re}{\operatorname{Re}}
\renewcommand{\Im}{\operatorname{Im}}

\newcommand{\Trans}{\operatorname{Trans}}  % Span
\newcommand{\Span}{\operatorname{Span}}  % Span
\newcommand{\Ran}{\operatorname{Ran}}  % Ker
\newcommand{\Dom}{\operatorname{Dom}}  % Ker
\newcommand{\Ker}{\operatorname{Ker}}  % Ker
\newcommand{\diam}{\operatorname{diam}}
\newcommand{\Sad}{\operatorname{ad_{s}}}
\newcommand{\SIad}{\operatorname{ad_{si}}}
\newcommand{\SRad}{\operatorname{ad_{sr}}}
\newcommand{\Nqad}{\operatorname{qad_{n}}}
\newcommand{\qad}{\operatorname{qad}}
\newcommand{\SRqad}{\operatorname{qad_{sr}}}
\newcommand{\Nad}{\operatorname{ad_{n}}}
\newcommand{\nil}{\operatorname{nil}}
\newcommand{\rad}{\operatorname{rad}}
\newcommand{\rrad}{\operatorname{rrad}}
\newcommand{\trace}{\operatorname{trace}}
\newcommand{\Iso}{\operatorname{Iso}}
\newcommand{\ad}{\operatorname{ad}}
\newcommand{\Ad}{\operatorname{Ad}}
\newcommand{\Aut}{\operatorname{Aut}}
\newcommand{\Der}{\operatorname{der}}

\newcommand{\pinfty}{+\infty}  %oppure {\infty}
\newcommand{\expe}{\mathrm{e}}

\newcommand{\barB}{\breve B}             %ball with \leq r
\newcommand{\afterbar}{\overline{\phantom{x}}}
\newcommand{\Jyvaskyla}{Jyv\"a\-skyl\"a}

\begin{document}

\title{From homogeneous metric spaces to Lie groups}

\author[Cowling]{Michael G. Cowling}
\address{School of Mathematics and Statistics\\ University of New South Wales\\ UNSW Sydney 2052\\ Australia}
\email{m.cowling@unsw.edu.au}
\thanks{MGC acknowledges support received from the Australia Research Council (DP140100531 and DP170103025).}

\author[Kivioja]{Ville Kivioja}
\address{Department of Mathematics and Statistics\\ University of \Jyvaskyla\\ \Jyvaskyla FI-40014 Finland}
\email{ville.k.kivioja@jyu.fi}
\thanks{VK was partly supported by the Emil Aaltonen Foundation.}

\author[Le Donne]{Enrico Le Donne}
\address{Department of Mathematics and Statistics\\ University of \Jyvaskyla\\ \Jyvaskyla FI-40014 Finland}
\curraddr{Department of Mathematics\\ University of Fribourg\\
Fribourg CH-1700 Switzerland}
\email{enrico.ledonne@unifr.ch}
\thanks{ELD was partially supported by the Academy of Finland (grants 288501 ``Geometry of subRiemannian groups'' and 322898 ``Sub-Riemannian Geometry via Metric-geometry and Lie-group Theory''),
and by the European Research Council (ERC Starting Grant 713998 GeoMeG ``Geometry of Metric Groups'').}

\author[Nicolussi Golo]{Sebastiano Nicolussi Golo}
\address{Department of Mathematics and Statistics\\ University of \Jyvaskyla\\ \Jyvaskyla FI-40014 Finland}
\email{sebastiano.s.nicolussi-golo@jyu.fi}
\thanks{SNG was supported by
the Academy of Finland (grant 322898 ``Sub-Riemannian Geometry via Metric-geometry and Lie-group Theory'');
the University of Padova STARS Project ``Sub-Riemannian Geometry and Geometric Measure Theory Issues: Old and New'';
the INdAM--GNAMPA Project 2019 ``Rectifiability in Carnot groups'';
the EPSRC Grant "Sub-Elliptic Harmonic Analysis" (EP/P002447/1);
and the Marie Curie Actions-Initial Training Network ``Metric Analysis For Emergent Technologies (MAnET)'' (n.~607643).}

\author[Ottazzi]{Alessandro Ottazzi}
\address{School of Mathematics and Statistics\\ University of New South Wales\\ UNSW Sydney 2052\\ Australia}
\email{a.ottazzi@unsw.edu.au}
\thanks{}

%53C30, % Homogeneous manifolds
%22F30, % Homogeneous spaces
%20F69, %  (2000-now) Asymptotic properties of groups
%22E25% -- Nilpotent and solvable Lie groups


% ----------------------------------------------------------------
\begin{abstract}
We study homogeneous metric spaces, by which we mean connected, locally compact metric spaces whose isometry group acts transitively.

After a review of a number of classical results, we use the Gleason--Iwasawa--Montgomery--Yamabe--Zippin structure theory to show that for all positive \( \epsilon \), each such space is \( (1,\epsilon) \)-quasi-isometric to a connected metric Lie group.

Next, we develop the structure theory of Lie groups to show that every homogeneous metric manifold is homeomorphically roughly isometric to a quotient space of a connected amenable Lie group, and roughly isometric to a simply connected solvable metric Lie group.

Third, we investigate solvable metric Lie groups in more detail, and expound on and extend work of Gordon and Wilson \cite{Gordon-Wilson-fine, Gordon-Wilson-solvmanifolds} and Jablonski \cite{Jablonski} on these, showing, for instance, that connected solvable Lie groups may be made isometric if and only if they have the same real-shadow.

Finally, we show that homogeneous metric spaces that admit a metric dilation are all metric Lie groups with an automorphic dilation.
\end{abstract}

\maketitle
% ----------------------------------------------------------------

\section{Introduction}
In this paper we present some links between Lie theory and metric geometry.
We consider \emph{homogeneous metric spaces}, that is, metric spaces whose isometry groups act transitively, subject to a number of standing assumptions:
\begin{enumerate}\renewcommand{\labelenumi}{(\alph{enumi})}
  \item homogeneous metric spaces are \emph{connected and locally compact}, unless explicitly stated otherwise;
  \item a metric means a distance function unless it is preceded by  \emph{infinitesimal}; and
  \item metrics are \emph{admissible}, that is, compatible with the topology of the under\-lying space.
\end{enumerate}
However, we do \emph{not} assume that they are riemannian, or geodesic, or quasigeodesic, or even proper.
If the metric space is also a topological manifold, and the metric topology and manifold topology coincide, then we write \emph{metric manifold}.
We consider locally compact groups and Lie groups equipped with admissible left-invariant metrics, which we call \emph{metric groups} and \emph{metric Lie groups}.

\tableofcontents

\subsection{Background}\label{ssec:overview}

Geometry and topology on Lie groups and their quotients have a very long history, which we cannot even begin to survey here; rather, we refer the reader to Helgason \cite{Helgason-DGLGSS}, Kobayashi and Nomizu \cite{Kobayashi-Nomizu-1, Kobayashi-Nomizu-2} and Samelson \cite{Samelson}.
Nevertheless, there are a few milestones that are specially relevant for this paper, namely Milnor \cite{Milnor-note}, Wolf \cite{Wolf-Growth}, Alekseevski\u\i\ \cite{Alekseevskii}, Wilson \cite{Wilson-homogeneous} and Gordon and Wilson \cite{Gordon-Wilson-fine, Gordon-Wilson-solvmanifolds}; in these papers and the texts cited previously, Lie groups and their quotients are considered as models for \emph{riemannian} manifolds.

There are very good reasons to consider Lie groups and their quotients with more general metrics.
These appear naturally in studying rigidity of symmetric spaces (see Mostow \cite{Mostow-Ridigity} and Pansu \cite{Pansu-Metriques}), regularity of subelliptic operators (see Folland and Stein \cite{Folland-Stein-dbar} and Rothschild and Stein \cite{Rothschild-Stein-hypoelliptic}), and asymptotic properties of nilpotent groups (see Gromov \cite{Gromov-poly-growth, Gromov_CCfromwithin} and Pansu \cite{Pansu-croissance}).
Negatively curved homogeneous riemannian manifolds, classified by Heintze \cite{Heintze}, have parabolic visual boundaries that are self-similar Lie groups with metrics that are not always riemannian.
The restriction to a connected closed subgroup of a riemannian metric need not be riemannian, or even geodesic.
For more information on these developments, see Montgomery \cite{Montgomery-Tour}, Cornulier and de la Harpe
\cite{Cornulier-Harpe-Geometry}, and Dungey, ter Elst and Robinson \cite{Dungey-Elst-Robinson}.

The prototypical examples of homogeneous metric spaces are connected locally compact groups with left-invariant metrics.
Solvable and nilpotent Lie groups, including the stratified groups of Folland \cite{Folland} or Carnot groups of Pansu \cite{Pansu-these}, are particularly nice examples.
Starting with these, one may obtain new examples by considering \( \ell^p \) products, passing to subgroups and quotients, and composing the metric with concave functions, as in the ``snowflake'' construction.

\subsection{Main results and contents}\label{ssec:main-results}

Chapter \ref{sec:prel} reviews the basic facts about homogeneous metric spaces and their isometry groups.
In more detail, we consider the realisation of homogeneous metric spaces as coset spaces of almost connected locally compact isometry groups, we describe various constructions to produce new metric spaces from old, and we discuss polynomial growth and doubling properties.
Because we allow metrics that are not proper or quasigeodesic, we observe some paradoxical phenomena, such as metric groups that are of polynomial growth as groups but not as metric spaces.
The introduction to Chapter \ref{sec:prel} provides more information.

Chapter \ref{sec:Lie-theory-metric-space} focusses on the use of Lie theory.
In dealing with general rather than riemannian metrics on Lie groups, what happens at the Lie algebra level may not determine what happens at the group level, and so the global approach is to be preferred.
That being said, however, the theory is similar in the riemannian and in the general cases.

Our first theorem follows from the Gleason--Iwasawa--Montgomery--Yamabe--Zippin structure theory of almost connected locally compact groups.

\begin{theoremA}
Let \( M \) be a homogeneous metric space.
Then \( M \) is
\begin{enumerate}
\item
\( (1, \epsilon) \)-quasi-isometric to a connected metric Lie group \( G_\epsilon \), for all positive \( \epsilon \), and
\item
{roughly isometric to a contractible metric manifold.}

\end{enumerate}
\end{theoremA}

We prove an extended version of Theorem A as Theorem \ref{thm:main-1}.

Our contributions here are the observations that quasi-isometry may be sharpened to rough isometry and the additive constant in (a) may be made arbitrarily small.
For fundamental groups of compact riemannian manifolds, part (b) was shown by \v{S}varc \cite{Svarc} and rediscovered by Milnor \cite{Milnor-note}.
More recently it has been extended, with quasi-isometry rather than rough isometry, to the case of {quasi}geodesic metrics and to spaces of polynomial growth: see \cite[Theorem 4.C.5.]{Cornulier-Harpe-Geometry} and \cite[Proposition 1.3]{Breuillard-Geometry}.

One of our aims is to study the following relation between topological groups.
Given two topological groups \( G \) and \( H \), we say that \emph{\( G \) may be made isometric to \( H \)} if there exist admissible left-invariant metrics \( d_G \) and \( d_H \) such that the metric spaces \( (G, d_G)  \) and \( (H, d_H) \) are isometric.
Moreover, if \( G \) is already a metric group, then we may impose the extra condition that the new metric is roughly isometric to the initial one; in this case, the Gromov--Hausdorff distance of the new metric space from the original one is finite.

Our next theorem, which relies heavily on the Levi and Iwasawa decompositions, shows that every homogeneous metric manifold may be made isometric to a compact quotient of a direct product of a solvable and a compact Lie group.

\begin{theoremB}
Let \( (M,d) \) be a homogeneous metric manifold.
Then there is a metric \( d' \) on \( M \) such that the identity mapping from \( (M,d) \) to \( (M, d') \) is a homeo\-morphic rough isometry, and {there is a transitive closed connected amenable subgroup \( A \) of \( \Iso(M,d') \); hence \( M \) is homeomorphic to \( A/K \), where \( K \) is a compact subgroup of \( A \).}
\end{theoremB}

We prove an extended version of Theorem B as Theorem \ref{thm:main-2}.

We believe that Theorem B is new, though it may have been known but not published.
Much is known about the isometry of riemannian symmetric spaces and riemannian solvmanifolds; but we are not aware of a complete treatment of the general case.
Gordon and Wilson \cite{Gordon-Wilson-fine, Gordon-Wilson-solvmanifolds} certainly came close to this, and promised a solution to the general case at the end of \cite{Gordon-Wilson-solvmanifolds}, but as far as we know this proposed paper did not eventuate.

In various special cases, we obtain simpler and more explicit results; see Corollaries \ref{cor:main-thm-2-Riemannian}, \ref{cor04111226}, and \ref{cor:allsemisimple_are_NAxK}.
Corollary \ref{cor:main-thm-2-Riemannian} is of particular interest: there we consider riemannian homogeneous spaces and riemannian metrics.
In this case, the result of Theorem B holds with rough isometry replaced by bi-Lipschitz equivalence.
Bi-Lipschitz equivalence is stronger locally, but weaker globally, than rough isometry, and our Theorem B provides more information about the large scale behaviour of homogeneous spaces than the strictly riemannian version.
This is further evidence that consideration of more general metrics can unlock information that is not accessible in the riemannian framework.

In Chapter \ref{sec:solvable}, we examine solvable metric Lie groups.
We need more background, which we discuss in more detail later.
Auslander and Green \cite{Auslander-Green} discovered that a connected simply connected solvable Lie group \( G \) of polynomial growth could be embedded in a connected solvable Lie group \( H \) (the \emph{hull} of \( G \)), in such a way that
\[
H = G \rtimes T \qquad\text{and}\qquad H = N \rtimes T ,
\]
where \( T \) is a torus (a compact connected abelian Lie group) in \( H \), and \( N \) is the nilradical (the largest connected normal nilpotent subgroup) of \( H \).
Then \( G \) is homeomorphic to \( N \), since both may be identified with \( H/T \), and \( G \) and \( N \) enjoy various similarities (see \cite{Auslander-Green, Alexopoulos}); \( N \) is known as the nilshadow of \( G \).
Gordon and Wilson~\cite{Gordon-Wilson-fine, Gordon-Wilson-solvmanifolds} considered this from a Lie algebraic point of view, and described \( G \) and \( N \) as modifications of each other; they considered general solvable Lie groups.
Recently, Cornulier \cite{Cornulier-dimension}, and very recently, Jablonski \cite{Jablonski} showed that every connected, simply connected solvable Lie group is homeomorphic to a split-solvable Lie group, which we call its real-shadow, in the same way as a connected, simply connected solvable group of polynomial growth is homeomorphic to its nilshadow.

We give a complete and coherent treatment of this recent development.
We then proceed to describe when solvable Lie groups may be made isometric.
Here is our third main theorem.

\begin{theoremC}
Let \( G_0 \) be a connected split-solvable Lie group, \( T \) be a maximal torus in \( \Aut(G_0) \), and \( d_0 \) be a \( T \)-invariant metric on \( G_0 \).
Let \( G_1 \) be a connected solvable Lie group.
Then the following are equivalent:
\begin{enumerate}
  \item \( G_1 \) may be made isometric to \( G_0 \);
  \item \( G_1 \) may be made isometric to \( (G_0,d_0) \);
  \item \( G_0 \) is the real-shadow of \( G_1 \); and
  \item \( G_1 \) may be embedded in \( H \coloneqq  G_0 \rtimes T \) in such a way that every element of \( h \) has a unique factorisation \( gt \), where \( g \in G_1 \) and \( t \in T \).
\end{enumerate}
\end{theoremC}

We prove an extended version of Theorem C as Theorem \ref{thm:main-3}.

While the results here are mostly known, our proofs are often different to and sometimes simpler than those of previous authors, and we believe that the reader will find it useful to have a clear account of this development.

Theorem C has various corollaries and extensions, some of which are due to Gordon and Wilson \cite{Gordon-Wilson-solvmanifolds} (for riemannian metrics) and Breuillard \cite{Breuillard-Geometry} (for the polynomial growth case).
First, the  metric \( d_0 \) on a connected split-solvable Lie group considered in Theorem C may be taken to be riemannian.
Next, if \( G_1 \) and \( G_2 \) are connected solvable Lie groups, then they may be made isometric if and only if they have the same real-shadow \( G_0 \), and in this case they may both be made isometric to \( (G_0, d_0) \).
In the special case in which \( G_0 \) is of polynomial growth, then \( G_0 \) is necessarily nilpotent, and so we obtain a characterisation of groups which may be made isometric to nilpotent Lie groups.

The classification of nilpotent groups up to quasi-isometry is an important unsolved problem.
Our result shows that if a connected Lie group admits one metric for which it is isometric to a nilpotent metric Lie group \( (N_1, d_1) \) and another for which it is isometric to another nilpotent metric Lie group \( (N_2, d_2) \), then necessarily \( N_1 \) and \( N_2 \) are isomorphic.

For more details and other results, see the discussion following the proof of Theorem C in Chapter \ref{sec:solvable}.

Finally, in Chapter \ref{sec:dil}, we discuss homogeneous metric spaces that admit metric dilations.
A map \( \delta:X\to Y \) between metric spaces is called a \emph{metric dilation} if \( \delta \) is bijective and \( d(\delta(x), \delta(x') ) =\lambda d(x, x') \)  for all  \( x, x'\in X \), for some \( \lambda\in (1, \pinfty) \), and a
\emph{metrically self-similar group} is a metric group \( (G, d) \) that admits a map \( \delta: G\to G \) that is both a metric dilation and an automorphism.
The stratified groups of Folland and Stein \cite{Folland-Stein-Hardy} with the Hebisch--Sikora metric \cite{Hebisch-Sikora}, the Carnot groups of Pansu \cite{Pansu-Metriques} and finite dimensional normed vector spaces are examples of metrically self-similar groups; so are the parabolic visual boundaries of the negatively curved connected homogeneous riemannian spaces described by Heintze \cite{Heintze}.
Our fourth main theorem described homogeneous metric spaces with dilations.

\begin{theoremD}
If a homogeneous metric space admits a metric dilation, then it is isometric to a metrically self-similar Lie group.
Moreover, all metric dilations of a metrically self-similar Lie group are automorphisms.
\end{theoremD}

Theorem D appears later as Theorem \ref{thm:main-4}.
It generalises a result of \cite{LeDonne-characterization}, where it is shown that a space is a sub-Finsler Carnot group if and only if the conditions in Theorem~D hold and the metric is geodesic.

As a consequence of \cite[Proposition 2.2]{Siebert} and \cite{Kivioja-LeDonne}, if a metric space \( M \) is isometric to a metrically self-similar Lie group \( (G, d') \), then \( G \) is a gradable, connected simply connected nilpotent Lie group isomorphic to the nilradical of \( \Iso(M) \).
However, \( M \) may also be isometric to a Lie group that is not nilpotent.
As discussed after Theorem~C, there are metric groups that are not nilpotent but which are isometric to metrically self-similar metric Lie groups; it follows from Theorem~D that if \( M \) is a metric Lie group and \(  \delta \) is a metric dilation, then \( \delta \) is an automorphism if and only if \( M \) is nilpotent.

While many of the results will be familiar to the experts, we included proofs if we could not find an explicit proof in the literature or if we could give an easier one.
We have not attempted to provide a full bibliography of all the areas that we touch on, but rather refer mainly to those papers that we use.
At the end of Chapters \ref{sec:prel} to \ref{sec:dil}, the reader will find some discussion of who did what and when, and of related results.
The reader may wish to consult some other works in this area, in particular, the books of Cornulier and de la Harpe \cite{Cornulier-Harpe-Geometry} for more information.
Recent papers, such as \cite{Fassler-LeDonne}, refer to other relevant recent works.

\subsection{Notation and conventions}\label{ssec:conventions}

We remind the reader of our convention that \emph{homogeneous metric spaces are connected and locally compact, unless explicitly stated otherwise}.
Metric manifolds, metric groups and metric Lie groups are examples of these.
Some of our results may be proved in greater generality, but this assumption will save space.

A set that is a \emph{neighbourhood} need not be open.
Locally compact groups are always locally compact Hausdorff topological groups.

The expression \emph{the isometry group} means the full isometry group, while \emph{an isometry group} means a closed subgroup of the full isometry group.

\emph{Constants} are always nonnegative real numbers that may vary from one occurrence to the next.
These are often denoted by \( C \) or \( \epsilon \), and we do not specify that these letters denote constants when they occur.

We denote by \( e_G \), or more simply \( e \), the \emph{identity element} of a group \( G \); the identity of \( G_1 \) may be denoted by \( e_1 \).

\subsection{Thanks}\label{ssec:thanks}
The research and writing of this work took many years, at least in part because the authors live in different continents.
During this period, most of the Europe-based authors visited Australia, and the Australia-based authors visited Europe, with support from various research grants, mentioned elsewhere, but also from the Universities of Jyv\"askyl\"a and of New South Wales.
In addition, the affiliations of the authors changed.
Apart from the institutions mentioned on the first page, Sebastiano Nicolussi Golo also worked at the Universities of Trento and Padova in Italy and Birmingham in the UK.
We thank the referee of a very preliminary version of this work for very many helpful comments that led to substantial improvements.

\section{Preliminaries}\label{sec:prel}
In this chapter, we recall some more or less familiar facts.
After introducing some notation in Section \ref{ssec:notation}, we discuss isometry groups of metric spaces, and in particular homogeneous metric spaces in Section \ref{ssec:isohom}, and examine the relations with coset spaces in Section \ref{ssec:metric-coset-spaces}.
Next we consider changes of metrics in Section \ref{ssec:new-metric}.
In Section \ref{ssec:simply-transitive}, we consider when there are simply transitive isometry groups, and finally, we discuss invariant measures, polynomial growth, and the doubling property in Section \ref{ssec:polgr}.
While these are very closely related in the case of \emph{proper quasigeodesic} metrics (see \cite{Cornulier-sublinear}), this is not the case for more general metrics, as we are going to see.

\subsection{Notation}\label{ssec:notation}

When \( (M, d) \) is a metric space, we sometimes write just \( M \), leaving the metric \( d \) implicit.
We denote by \( B(x, r) \) or \( B_d(x, r) \) the open ball \( \{ y \in M : d(x,y) < r\} \), and by \( \barB(x, r) \) or \( \barB_d(x, r) \) the closed ball \( \{ y \in M : d(x,y) \leq r\} \), which need not be the closure of the open ball \( B(x,r) \); set closure is denoted with a bar.
The metric space is said to be \emph{proper} if closed bounded sets are compact, or equivalently, if all balls \( \barB_d(x, r) \) are compact, and is said to be \emph{geodesic} if every pair of points may be joined by a curve whose (rectifiable) length is equal to the distance between the points.
Berestovski\u{\i} \cite{Berestovskii} showed that a homogeneous metric manifold is geodesic if and only if it is equipped with an invariant infinitesimal sub-Finsler metric.

A function \( f:(M_1, d_1)\to (M_2, d_2) \) is an \emph{\( (L, C) \)-quasi-isometry} \index{quasi-isometry} if
\[
L^{-1} d_1(x, y) - C \le d_2(f(x), f(y)) \le L d_1(x, y) + C
\]
for all \( x, y\in M_1 \), and for every \( z\in M_2 \) there is \( x\in M_1 \) such that \( d_2(f(x), z)\le C \).
If such a function exists between two metric spaces, then we say that they are \emph{\( (L, C) \)-quasi-isometric}, or more simply \emph{quasi-isometric}.

There is a zoo of equivalences of metric spaces that we might consider.
Quasi-isometry (for some choice of the constants \( L \) and \( C \), possibly depending on the function) is an equivalence relation.
If \( C = 0 \), then \( f \) is called \emph{bi-Lipschitz}; \index{bi-Lipschitz} bi-Lipschitz gives us another equivalence relation, which, in contrast to quasi-isometry, implies homeomorphism.
A third equivalence relation is \emph{rough isometry}, \index{rough isometry} which is defined to be \( (1,C) \)-quasi-isometry for a suitable choice of \( C \), which may depend on \( f \); we sometimes call \( C \) the implicit constant of a rough isometry.
This is finer than general quasi-isometry and more restrictive at large scales than bi-Lipschitz.
Yet another equivalence relation that we consider is homeomorphic rough isometry.
A fifth relation that we consider applies to topological rather than metric groups: we say that \( G_1 \) and \( G_2 \) \emph{may be made isometric} provided that there exist admissible left-invariant metrics \( d_1 \) and \( d_2 \) such that \( (G_1,d_1) \) and \( (G_2, d_2) \) are isometric.

\subsection{Homogeneous metric spaces}\label{ssec:isohom}
We define an \introd{isometry} of a metric space \( (M, d) \) to be a \emph{surjective} map \( f \) on \( M \) such that
\begin{equation}
\label{eq:metric_preserving}
d(f(x), f(y))=d(x, y) \qquad\forall x, y\in M.
\end{equation}
We denote by \( \Iso(M, d) \) the set of all isometries of \( (M, d) \); given the surjectivity, it is evident that \( \Iso(M,d) \) is a group under  composition.
We recall that a metric space \( (M, d) \) is said to be \introd{homogeneous} if its isometry group acts transitively, and our convention that a homogeneous metric space \( (M,d) \) is connected and locally compact, but not necessarily proper, unless explicitly stated otherwise.

Changing the metric on a space (without changing its topology) may change its isometry group.
For instance, we may equip \( \R^2 \) with any one of the bi-Lipschitz equivalent translation-invariant metrics
\[
d( (x_1, y_1), (x_2, y_2) ) = \bigl( |x_1 - x_2|^p + a |y_1 - y_2|^p \bigr) ^{1/p},
\]
where \( 1 \leq p < \pinfty \) and \( 0 < a < \pinfty \).
When \( p = 2 \), the iso\-metry group includes rotations, but otherwise it does not.
And when \( p = 2 \), the rotation group depends on the parameter \( a \).
However, in this example, each of the iso\-metry groups act by bi-Lipschitz transformations with respect to all the other metrics.

We prove that \( \Iso(M, d) \) is a metrisable, locally compact and \( \sigma \)-compact topological group that acts with compact stabilisers (Theorem \ref{thm:isom-group-props}),  and whose identity component acts transitively (Corollary~\ref{cor:H_0-transitive}).
In Theorem \ref{thm:general-isometry-group}, we also prove a more quantitative and precise statement about the
metrisability, namely that for every \( \epsilon \in \R^+ \), the group \( \Iso(M, d) \) may be metrised such that it is \( (1,\epsilon) \)-quasi-isometric to~\( (M,d) \).

\begin{proposition}\label{prop:3-topols}
Let \( (M, d) \) be a metric space, not necessarily connected or locally compact. Then the compact-open topology and the topologies of uniform convergence on compacta and of pointwise convergence agree on \( \Iso(M, d) \), and the group \(  \Iso(M, d) \), endowed with any of these topologies, is a topological group.
\end{proposition}

\begin{proof}
For the fact that these topologies agree on \( \Iso(M, d) \), see  \cite[Lemmas 5.B.1 and 5.B.2]{Cornulier-Harpe-Geometry}.
That this structure makes the isometry group a topological group is well known; van Dantzig and van der Waerden \cite{Dantzig-Waerden} show this in the case where \( M \) is connected, locally compact and separable, and a proof of the general case may be found in \cite[Lemma 5.B.3]{Cornulier-Harpe-Geometry}.
\end{proof}

We now equip \( \Iso(M,d) \) with any of the topologies above.

We are not assuming that our metric spaces are proper, but we still need some substitute for a proper metric, and this construction (and some other useful facts) will be the subject of the next two lemmas.
Much of this is ``folklore'', but we do not know a reference
and so we include proofs.
We first choose \( \ell \in \R^+ \) small enough that \( \barB(p, 2\ell) \) is compact for one and hence every \( p \) in \( M \) by homogeneity.
Then there exists a positive integer \( L \) for which the compact set
\( \barB(p, 2\ell) \) may be covered by \( L \) open balls of radius \( \ell \), for one and hence all \( p \) in \( M \) by homogeneity.

Given a point \( o \in M \), we define sets \( V_n(o, \ell) \) inductively:
first, \( V_0(o, \ell) \coloneqq  \{o\} \), then
\begin{equation}\label{eq:def-Vn0ell}
V_{n}(o, \ell) \coloneqq \bigcup_{p\in V_{n-1}(o, \ell)} \barB(p, \ell)
\end{equation}
when \( n \in \Z^+ \).
Further, we define \( U_o \coloneqq  \{ g \in \Iso(M,d) : g(o) \in \barB(o, \ell)\} \).

\begin{lemma}\label{lem:prelim}
Let \( G \) be the isometry group of a homogeneous metric space \( (M, d) \), and \( o \) be any point of \( M \).
Then
\begin{enumerate}
\item \( V_{n}(o, \ell) \) may be covered by at most \( L^{n} \) open balls \( B(p, \ell) \) for all \( n \in \Z^+ \);
\item \( M = \bigcup_{n \in \N} V_{n}(o, \ell) \), whence \( (M,d) \)
    is \( \sigma \)-compact and second countable;
\item a subset \( A \) of \( M \) is precompact if and only if \(A
    \subseteq V_{n}(o, \ell) \) for some  \( n \in \N \);
\item the \( n \)-fold product \( U_o^n \) is equal to \( \{ g \in G : g(o) \in V_{n}(o, \ell)\} \) for all \( n \in \N \);
\item \( U_o \) is compact in \( G \), whence \( U_o^n \) is compact in \( G \) and so \( V_{n}(o, \ell) \) is compact in \( M \) for all \( n \in \N \)
\end{enumerate}
\end{lemma}

\begin{proof}
First, if \( x \in \bigcup_{q \in \barB(p, \ell) } \barB(q, \ell) \), then
\[
d(x,p) \leq d(x,q) + d(q,p) \leq 2\ell.
\]
Hence \( \bigcup_{q \in \barB(p, \ell) } \barB(q, \ell) \) may be covered by \( L \) balls of radius \( \ell \), by our choice of \( L \).
Now (i) may be proved by induction.

From (i), we see that \( V_n(o,\ell) \) is precompact.
Now \( \bigcup _{n \in \N} V_{n}(o, \ell) \) is both open and closed in \( M \) and hence coincides with \( M \).
It follows that  \( M \) is \( \sigma \)-compact and hence second countable, which completes the proof of (ii).

To prove (iii), note that
\( \{ \bigcup_{p \in V_n(o,\ell) } B(p, \ell) : n \in \N\} \)
is an increasing open cover of \( M \), and hence if \( A \) is a precompact
subset of \( M \), then for some \( n \),
\[
A \subseteq \bar A \subseteq \bigcup_{p \in V_n(o,\ell) } B(p, \ell) \subseteq V_{n+1} (o, \ell).
\]
Conversely, if \( A \subseteq V_{n+1} (o, \ell) \) then \( A \) is precompact.

For (iv), we must show that
\begin{equation}\label{eq:U-V-reln}
U_o^n = \{g \in G : g(o)\in V_n(o, \ell)\} .
\end{equation}
If \( n=1 \), then \eqref{eq:U-V-reln} holds by definition.
Assume that \eqref{eq:U-V-reln} holds when \( n=k \).
On the one hand, if \( f\in U_o^{k+1} \), then \( f=gh \) where \( g\in U_o^k \) and \(  h\in U_o \), so
\[
f(o)\in g(\barB(o, \ell)) = \barB(g(o), \ell)\subseteq V_{k+1}(o, \ell).
\]
On the other hand, suppose that \( f(o)\in V_{k+1}(o, \ell) \).
There exists \( q \in V_{k}(o,\ell) \) such that \( f(o) \in \barB(q, \ell) \) by definition, and by transitivity and the inductive hypothesis, there exists \( g \in U_o^k \) such that \( q = g(o) \).
Now \( g^{-1}f(o)\in \barB(o, \ell) \), that is, \( g^{-1}f\in U_o \), since \( g^{-1} (\barB(g(o), \ell)) = \barB(o, \ell) \), and we may conclude that \( f\in U_o^{k+1} \).
By induction, \eqref{eq:U-V-reln} holds for all \( n \).

For (v), the Arzel\`a--Ascoli theorem shows that \( U_o \) is precompact in the compact-open topology.
Moreover, if \( (f_{n})_{n \in \N} \) is a sequence of elements of \( U_o \) that converges to \( f \in G \), then \( f_{n}(o) \) converges to \( f(o) \in M \) and \( d( f_{n}(o), o) \leq \ell \) for all \( n \), whence \( d( f(o), 0) \leq \ell \) and \( f \in U_o \).  Thus \( U_o \) is compact.

Since \( G \) is a topological group, \( U_o^n \) is compact for each \( n \in \N \), and so \( V_n(o,\ell) \) is compact from (iv) and the continuity of the map \( g \mapsto g(o) \) from \( G \) to \( M \).
\end{proof}

We now construct two proper metrics on \( M \); the first has the advantage that it is closely related to the sets \( V_n(o,\ell) \) and the second that it is admissible.
We define the \introd{Busemann gauge} \( \rho_{[\ell]} \) on \( M \) by
\begin{equation}\label{eq:Busemann-gauge}
\rho_{[\ell]}(o, p) \coloneqq  \ell \min\{ n \in \N : p \in V_{n}(o,\ell) \}
\end{equation}
and the \introd{derived semi-intrinsic metric} \( d_{[\ell]} \) by
\begin{equation}\label{eq:derived-metric}
d_{[\ell]}(p,q)
\coloneqq  \inf \Bigl\{ \sum_{j=1}^k d(x_j, x_{j-1} ) : x_0, \dots, x_k \in M, x_0 = p, x_k = q, d(x_j,x_{j-1}) \leq \ell  \Bigr\} .
\end{equation}
We note that \( \rho_{[\ell]} \) takes discrete values.
Observe that, in the case where the metric space is \( \R \) and the metric is given by \( d(x,y) = |x-y|^{\theta} \), where \( \theta \in (0,1) \) and \( \ell =1 \), the derived semi-intrinsic metric is given by \( d_{[\ell]}(x,y) = \lfloor |x-y| \rfloor + ( |x-y| - \lfloor |x-y| \rfloor )^\theta \), and is somewhat bizarre;  here \( \lfloor x \rfloor \) denotes the integer part of \( x \).

\begin{lemma}\label{lem:Busemann-gauge}
The Busemann gauge \( \rho_{[\ell]} \) and the derived semi-intrinsic metric \( d_{[\ell]} \) are both metrics on the set \( M \).
In addition,
\[
d(p,q) \leq d_{[\ell]} (p,q) \leq \rho_{[\ell]}(p,q) \leq 2 d_{[\ell]} (p,q) + \ell
\qquad\forall p, q \in M.
\]
Hence \( d_{[\ell]} \) is proper as \( \rho_{[\ell]} \) is proper.
Further, if \( d(x,y) \leq \ell \), then \( d_{[\ell]}(x,y) = d(x,y) \) for all \( x, y \in M \), whence \( d_{[\ell]} \) is admissible.
\end{lemma}

\begin{proof}
It is easy to see that both \( \rho_{[\ell]} \) and \( d_{[\ell]} \) are metrics.

Take \( p, q \in M \).
On the one hand, if \( q \in V_n(p,\ell) \), then by definition there are points \( x_j \in M \), where \( 0 \leq j \leq n \), such that \( x_0 = p \), \( x_n =x \) and \( x_j \in \barB(x_{j-1}, \ell) \).
It follows immediately that \( d_{[\ell]}(p,q) \leq n\ell \).

Moreover, for any positive \( \epsilon \), there are points \( x_0, \dots, x_k \) such that \( x_0 = p \), \( x_k = q \) and \( \sum_{j=1}^k d(x_{j}, x_{j-1}) \leq d_{[\ell]}(p,q) + \epsilon \).
Observe that we may omit points \( x_j \) if \( d(x_{j+1}, x_{j}) + d(x_{j}, x_{j-1}) \leq  \ell \), for in this case
\[
d(x_{j+1}, x_{j-1}) \leq d(x_{j+1}, x_{j}) + d(x_{j}, x_{j-1}) \leq \ell.
\]
We omit such points recursively until this is no longer possible.

At this point, we may not only assume that \( \sum_{j=1}^k d(x_{j}, x_{j-1}) \leq d_{[\ell]}(p,q) + \epsilon \), but also that \( d(x_{j+1}, x_{j}) + d(x_{j}, x_{j-1}) >  \ell \).
It follows that
\[
d_{[\ell]}(p,q) + \epsilon \geq \sum_{j=1}^k d(x_{j}, x_{j-1}) > \lfloor k/2\rfloor\ell,
\]
and this implies that
\[
\rho_{[\ell]}(p,q) \leq k \ell \leq \ell + 2 d_{[\ell]}(p,q).
\]
The rest of the proof is evident.
\end{proof}

It is easy to see that, if \( d \) is a geodesic metric, then \( d_{[\ell]} \) coincides with \( d \).
Moreover, if we start with arbitrary admissible metrics \( d_1 \) and \( d_2 \) with a common transitive isometry group, and construct the Busemann gauges \( \rho_{1,[\ell]} \) and \( \rho_{2,[\ell]} \) or the derived semi-intrinsic metrics \( d_{1,[\ell]} \) and \( d_{2,[\ell]} \) (still with the assumption that the balls \( B_{d_1}(p, 2\ell_1) \) and \( B_{d_2}(p, 2\ell_2) \) are relatively compact) then \( \rho_{1,[\ell]} \) and  \( \rho_{2,[\ell]} \) are quasi-isometric.
Hence by Lemma \ref{lem:Busemann-gauge} all the metrics  \( \rho_{1,[\ell]} \), \( \rho_{2,[\ell]} \), \( d_{1,[\ell]} \) and \( d_{2,[\ell]} \) are quasi-isometric.
It is also straightforward to see that the derived semi-intrinsic metrics \( d_{[\ell_1]} \) and \( d_{[\ell_2]} \) are quasi-isometric (again, provided that the balls \( B_{d_1}(p, 2\ell_1) \) and \( B_{d_2}(p, 2\ell_2) \) are relatively compact).

We now introduce an important class of metrics.

\begin{definition}
A metric on a homogeneous metric space \( (M,d) \) is called \emph{{proper quasigeodesic}} if the identity map is a quasi-isometry from \( (M,d) \) to \( (M, \rho_{[\ell]}) \), where \( \rho_{[\ell]} \) is the Busemann gauge defined in \eqref{eq:Busemann-gauge}.
\end{definition}

This definition is not standard, but coincides with the usual versions.
Two distinct {proper quasigeodesic} metrics on \( M \) are quasi-isometric.

\subsection{Metric spaces and coset spaces}\label{ssec:metric-coset-spaces}

We begin by clarifying notation.
A group \( H \) \emph{acts} on a set \( M \) if there is a homomorphism \( \alpha \) from \( H \) to \( \Trans(M) \), the group of all invertible transformations of \( M \).
If the action is \emph{effective}, that is, if \( \alpha(h) p = p \) for all \( p \in M \) only if \( h = e \), then \( H \) may be identified with a subgroup of \( \Trans(M) \).

\begin{remark}\label{rem:no-compact-normal-subgroups}
If a group \( H \) acts transitively on a set \( M \), then all the stabilisers of points in \( M \) are conjugate.
Hence a normal subgroup of \( H \) that is contained in one stabiliser is contained in all stabilisers, that is, it fixes all points.
Thus if \( H \) acts effectively and transitively on a set, then no nontrivial compact normal subgroups of \( H \) are contained in a stabiliser.
In general, if \( H \) acts transitively but not effectively, and \( K \) is the stabiliser of a point, then \( N \coloneqq  \bigcap_{h \in H} hKh^{-1} \) is a normal subgroup of \( H \) that may be factored out to obtain a effective action of \( H/N \), since \( H/K \) may be identified with \( (H/N)/(K/N) \).

We write \( Z(H) \) for the centre of a group \( H \); then what we have just shown implies in particular that if \( H \) acts effectively on a set, and \( K \) is the stabiliser of a point, then \( K \cap Z(H) =\{e\} \).
\end{remark}

An action \( \alpha \) of a group \( H \) on a metric space \( (M,d) \) is \emph{isometric} or \emph{by isometries} if  \( \alpha(H) \subseteq \Iso(M,d) \).

\begin{theorem}\label{thm:isom-group-props}
Let \( (M, d) \) be a homogeneous metric space, \( o \) be a point of \( M \), \( \rho_{[\ell]} \) be the Busemann gauge of \eqref{eq:Busemann-gauge}, and \( H \) be the isometry group of \( (M,d) \).
Then
\begin{enumerate}
\item \( H \) is locally compact, \( \sigma \)-compact and second countable;
\item the stabiliser \( K \) of \( o \) is compact;
\item \( H \) is metrisable, and for each \( \epsilon \in \R^+ \), the \introd{Busemann metric} \( d_H \) on \( H \), given by
\[
d_H(g, h) \coloneqq \sup\{d(g(q), h(q))\expe^{-{\rho_{[\ell]}(o, q)}/{\epsilon}} : q\in M\},
\]
is an admissible left-invariant metric on \( H \);
\item  the map \( \pi: g\mapsto g(o) \) from \( (H, d_H) \) to \( (M, d) \) is 1-Lipschitz and \( (1, 2 \epsilon / \expe ) \)-quasi-isometric; more precisely,
\begin{equation*}
d_H(g,h) - 2\epsilon / \expe \le d(g(o),h(o)) \le d_H(g,h)
\qquad\forall g, h \in H.
\end{equation*}
\item \( d_H \) is right-\( K \)-invariant, that is, \( d_H(gk, hk) = d_H(g,h) \) for all \( g, h \in H \) and all \( k \in K \), and \( \diam_H(K) \leq 2\epsilon/\mathrm{e} \).
\end{enumerate}
\end{theorem}

\begin{proof}
The local compactness of \( H \) was shown by van Dantzig and van der Waerden \cite{Dantzig-Waerden}.

By Lemma \ref{lem:prelim} (v), (ii) and (iv) and Proposition \ref{prop:3-topols}, the set \( U_o \) and hence the sets \( U_o^n \) are compact in \( H \) when \( n \in \N \), and \( H = \bigcup_{n \in \N} U_o^n \), whence \( H \) is \( \sigma \)-compact.
The second countability of \( H \) follows from that of \( M \).

Next, van Dantzig and van der Waerden proved (ii), which also follows from the fact that the stabiliser of \( o \) is a closed subset of the compact set \( U_o \) of Lemma \ref{lem:prelim}.

Clearly \( d_H \) is left-invariant; we need to show that it is admissible.
Let \( ( g_n )_{n \in \N } \) be a sequence in \( H \).
On the one hand, if \( g_n \to g \) in \( (H, d_H) \), then
\[
d(g_n(p), g(p))
\le \expe^{{\rho_{[\ell]}(o,p)}/{\epsilon}} d_H(g_n, g),
\]
for all \( p\in M \), and hence \( g_n \) converges to \( g \) pointwise, and so in \( H \).

On the other hand, if \( g_n \to g \) in \( H \), then the convergence is uniform on compacta, by Proposition~\ref{prop:3-topols}.
Fix \( \eta \in (0,1) \), and take \( R \in \R^+ \) such that \(  t\expe^{-{t}/{\epsilon}}<\eta \) whenever \( t>R \).
Define \( A \) to be the closure of \( \{p\in M: \rho_{[\ell]}(o, p)\le R\} \).
Then \( A \) contains \( o \) and is compact in \( M \) by definition and part (v) of Lemma \ref{lem:prelim}.
Hence there is \( n_0\in \N  \) such that \( d(g_n(p), g(p))\le \eta \) for all \( p\in A \) and all \( n \ge n_0 \).
Therefore
\[
d(g_n(p), g(p))\expe^{-{\rho_{[\ell]}(o,p)}/{\epsilon}} \le \eta ,
\]
if \( n \ge n_0 \) and \( p \in A \), while if \( n \ge n_0 \) and \( p \notin A \), then
\begin{align*}
&d(g_n(p), g(p))\expe^{-{\rho_{[\ell]}(o,p)}/{\epsilon}} \\
&\quad\le \bigl( d(g_n(p), g_n(o)) + d(g_n(o), g(o)) + d(g(o), g (p))\bigr)\expe^{-{\rho_{[\ell]}(o,p)}/{\epsilon}} \\
&\quad\le ( 2 d(o, p) + \eta ) \expe^{-{\rho_{[\ell]}(o,p)}/{\epsilon}} \\
&\quad\le 3 \eta .
\end{align*}
We conclude that \( d_H(g_n, g) \le 3\eta \) for all \( n\ge n_0 \).
As \( \eta \) may be chosen to be arbitrarily small, \( g_n \to g \) in \( (H, d_H) \).

By definition, \( d(\pi(g), \pi( h)) = d(g(o), h(o))\le d_H(g, h) \) for all \( g,h \in H \), so \( \pi \) is 1-Lipschitz.
Moreover, \( \pi \) is surjective by the homogeneity assumption, and
\begin{equation}\label{eq:quasi-isometric-projection}
\begin{aligned}
 d_H(g, h)
&\le \sup\{(d(g(p), g(o))+d(g(o), h(o))+d(h(o), h(p)))\expe^{-{\rho_{[\ell]}(o,p)}/{\epsilon}}: p\in M\} \\
&\le d(g(o), h(o)) \sup\{\expe^{-{\rho_{[\ell]}(o,p)}/{\epsilon}}: p\in M\}
+ 2\sup\{d(o, p)\expe^{-{\rho_{[\ell]}(o,p)}/{\epsilon}}: p\in M\}\\
&\le d(\pi(g), \pi(h)) + 2  \epsilon / \expe
\end{aligned}
\end{equation}
for all \( g, h\in H \), whence \( \pi \) is a \( (1, 2\epsilon / \expe) \)-quasi-isometry.

Finally, for \( g, h \in H \) and \( k \in K \),
\[
\begin{aligned}
d_H(gk, hk)
&= \sup\{d(gk(q), hk(q))\expe^{-{\rho_{[\ell]}(o, q)}/{\epsilon}} : q\in M\} \\
&= \sup\{d(g(k(q)), h(k(q)))\expe^{-{\rho_{[\ell]}(o, k(q))}/{\epsilon}} : q\in M\}
= d_H(g,h),
\end{aligned}
\]
as required.
Further, from \eqref{eq:quasi-isometric-projection},
\[
\diam(K)
= \sup\{ d_H(k,e) : k \in K \}
\leq 2\epsilon/\mathrm{e} + \sup\{ d(k(o),e(o)) : k \in K\}
= 2\epsilon/\mathrm{e} ,
\]
and the proof is complete.
\end{proof}

Observe that we could define the Busemann metric in the statement of the theorem using \( d_{[\ell]} \) rather than \( \rho_{[\ell]} \), and the proof above would work with minor modifications.
Observe also that \( \Iso(M,d) \subseteq \Iso(M,d_{[\ell]}) \), where \( d \) is a derived semi-intrinsic metric as defined just before Lemma \ref{lem:Busemann-gauge}.

We now consider closed subgroups of the isometry group in more detail.

\begin{theorem}\label{thm:general-isometry-group}
Let \( (M, d) \) be a homogeneous metric space, \( G \) be a closed subgroup of \( \Iso(M, d) \), and \( S \) be the stabiliser in \( G \) of a point \( o \) in \( M \).
Then
\begin{enumerate}
\item
\( G \) is locally compact and \( S \) is compact;
\item
if \( G \) acts transitively on \( M \), then the map \( gS\mapsto g(o) \) is a homeomorphism from \(  G/S \) to \( M \);
\item
if \( \barB(o, \ell)\subseteq G(o) \) for some choice of \( \ell\in\R^+ \) and \(  o\in M \), then \( G \) acts transitively on \( M \);
\item if \( G \) is open in \( \Iso(M,d) \), then it acts transitively on \( M \);
\item if \( G \) acts transitively on \( M \), then for each \( \epsilon \in \R^+ \) and \( o \in M \), we may equip \( G \) with an admissible left-invariant metric in such a way that the map \( g \mapsto g(o) \) is \( 1 \)-Lipschitz and a \( (1, \epsilon) \)-quasi-isometry;
\item if \( G \) acts transitively on \( M \), then for each \( n \in \N \) and \( o \in M \),
\[
\{ g \in G : g(o) \in \barB(o, \ell)\}^n = \{ g \in G : g(o) \in V_{n}(o, \ell)\} .
\]
\end{enumerate}
\end{theorem}

\begin{proof}
Part (i) is standard: closed subspaces of locally compact or compact spaces are locally compact or compact.

Part (ii) follows from \cite[Theorem 3.2, page~121]{Helgason-DGLGSS}.

For part (iii), the orbit \( G(o) \) is nonempty, open and closed.
As \( M \) is connected, by our standing assumption, \( M = G(o) \).

For part (iv), it follows from part (ii) that the map \( g\mapsto g(o) \) from \( G \) to \( M \) is open.
Consequently \( G(o) \) is open and \( G \) acts transitively by part (iii).

The proof of part (v) is similar to the proof of part (iii) in Theorem~\ref{thm:isom-group-props}, and the proof of part (vi) is similar to the proof of part (iv) in Lemma~\ref{lem:prelim}.
\end{proof}

\begin{corollary}\label{cor:H_0-transitive}
Let \( (M, d) \) be a homogeneous metric space.
The connected component \( H \) of the identity in \( \Iso(M, d) \) is locally compact and acts transitively on \( M \), and the quotient \( \Iso(M,d)/H \) is compact.
\end{corollary}

\begin{proof}
The subgroup \( H \) is closed in \( \Iso(M,d) \), and hence is locally compact.
It is also normal, and the totally disconnected locally compact group \( \Iso(M, d)/H \) has a neighbourhood base \( \Nu \) of the identity consisting of open and closed subgroups, ordered by reverse inclusion; see \cite[Proposition 4.13]{Stroppel}.
For each \( \nu\in \Nu \), let \( H_\nu \) be the preimage of \( \nu \) in \(  \Iso(M, d) \).
Then \( ( H_\nu )_{\nu\in \Nu} \) is a net of open and closed subgroups of \(  \Iso(M, d) \) such that \( H =\bigcap_{\nu\in \Nu} H_\nu \), and \( H_\nu \) acts transitively on \( M \) for every \( \nu\in \Nu \) by Theorem~\ref{thm:general-isometry-group}.

Take \( o, p\in M \).
For each \( \nu\in \Nu \), there is \( g_\nu\in H_\nu \) such that \(  g_\nu(o)=p \).
By the Arzel\`a--Ascoli theorem, \( \{ g \in \Iso(M,d) : g(o) = p\} \) is compact; since each \( g_\nu \) lies in this set, we may assume that \( g_{\nu} \) converges to \( g\in \Iso(M,d) \) by passing to a subnet if necessary.
For each \( \nu \in \Nu \),  \( g_{\nu'} \in H_\nu \) when \( \nu' \ge \nu \), and hence \( g \in H_\nu \).
In conclusion, \( g \in\bigcap_{\nu\in \Nu}H_\nu = H \) and \( g(o) = p \).

Let \( K \) be the stabiliser in \( \Iso(M,d) \) of the point \( o \) in \( M \); then \( K \) is compact. Since \( H \) acts transitively, for every \( g \in \Iso(M,d) \), there exists \( h \in H \) such that \( h^{-1}g(o) = o \), that is, \( h^{-1}g \in K \).  It follows that \( \Iso(M,d) \subseteq HK \).
\end{proof}

The next definition summarises and extends the structure that we have seen in the last theorems.

\begin{definition}\label{def:metric-projection}
A \introd{homogeneous metric projection} is a pair of homogeneous metric spaces \( (M_1,d_1) \) and \( (M_2,d_2) \), with a group \( H \) acting isometrically, continuously and transitively on both \( M_1 \) and \( M_2 \), and an \( H \)-equivariant projection \( \pi: M_1 \to M_2 \) such that
\[
d_2(x_2, y_2) = \inf\{ d_1(x_1,y_1) : \pi x_1 = x_2, \pi y_1 = y_2 \}
\qquad \forall x_2,y_2  \in M_2.
\]
The set \( \{ x_1 \in M_1 : \pi x_1 = x_2\} \) is called the \emph{fibre above \( x_2 \)} in \( M_2 \).
\end{definition}

Because \( H \) acts continuously on both \( M_1 \) and \( M_2 \), the stabilisers \( K_1 \) of a point \( x \) in \( M_1 \) and \( K_2 \) of \( \pi x \) in \( M_2 \) are closed, and it is clear that \( K_1 \subseteq K_2 \).
There is then a natural identification of the fibre above \( x \) with the quotient space \( K_2/K_1 \), and all the fibres are isometric to each other because \( H \) acts transitively.
As noted in the remark above, the subgroup of \( H \) of elements that act trivially on \( M_1 \) (and \emph{a fortiori} on \( M_2 \)) is a closed normal subgroup that may be factored out.

With \( K_1 \) and \( K_2 \) as above, if the set \( K_2/K_1 \) is compact, then the diameter of each fibre is bounded; hence there exists a constant \( C \) such that
\[
d_1(x,y) - C  \le d_2(\pi x, \pi y) \le d_1(x,y)
\qquad \forall x,y  \in M_1,
\]
that is, \( \pi \) is \( 1 \)-Lipschitz and a rough isometry.
The constant \( C \) is called the implicit constant of the projection \( \pi \) and may be identified with the diameter of \( K_2/K_1 \).

Let \( \pi \) be the projection from a locally compact group \( H \) onto a quotient space \( H/K \).
We recall that a \emph{section} \( \sigma \) for \( H/K \) in \( H \) is a mapping such that \( \pi \circ \sigma \) is the identity map on \( H/K \).
It is well-known that sections exist: they may be taken to be Borel or even Baire (see, for instance, \cite {Kehlet}).
It is evident that if \( \pi \) is a homogeneous metric projection from \( (M_1, d_1) \) onto \( (M_2, d_2) \) and \( H \) is a common transitive isometry group, then \( M_2 \) may be identified with \( H/K_2 \), where \( K_2 \) is a compact subgroup of \( H \), and a section from \( M_2 \) to \( H \) composed with the projection from \( H \) to \( M_1 \) is a section from \( M_2 \) to \( M_1 \).
If
\begin{equation*}
d_1(x,y) - C  \le d_2(\pi x, \pi y) \le d_1(x,y)
\qquad\forall x, y \in M_1,
\end{equation*}
and if \( \sigma \) is a section for \( M_2 \) in \( M_1 \), then
\[
d_2(p, q)  \leq d_1(\sigma (p), \sigma(q)) \leq d_2(p, q) +  C
\qquad\forall p, q \in M_2.
\]

We conclude this section with two remarks.

\begin{remark}\label{rem:fourth-topology}
Let \( (M,d) \) be a homogeneous metric space, and let \( H \) be a subgroup of \( \Iso(M,d) \) that acts transitively on \( M \).
Equip \( \Iso(M,d) \) with the topology of Proposition \ref{prop:3-topols}, \( H \) with the relative topology, and \( M \) with the topology induced by \( d \).
Take an arbitrary point \( o \) of \( M \).

Then the relative topology on \( H \) is also the only topology on \( H \) such that the mapping \( \pi: h \mapsto ho \) is continuous and open.
Indeed, the sets \( \{ g \in H: d(hx,x) < \epsilon \} \), where \( x \in M \) and \( \epsilon \in \R^+ \) form a subbase for any topology on \( H \) such that \( \pi \) is continuous and open, and also for the topology of pointwise convergence.

This implies that if \( U \subset H \) and \( U = UK \), then \( U \) is open in \( H \) if and only if \( Uo \) is open in \( M \).
It follows that if we change the metric on \( M \) to a new metric that induces a different topology and is such that \( H \) is still an isometry group, then the topology of \( H \) as an isometry group with the new metric must also change.
\end{remark}

\begin{remark}\label{rem:metric-groups-are-closed}
Let \( (G,d) \) be a metric group, that is, \( G \) is a connected locally compact group, with an admissible metric \( d \).
The group \( G \), acting on itself by left translations, may be viewed as a subgroup of \( \Iso(G,d) \); this subgroup is closed.
Indeed, take \( g_n \in G \) such that \( g_n \to h \) in \( \Iso(G,d) \); we need to show that \( h \in G \).
Let \( g = he \) in \( G \).
Now \( g_n = g_ne \to he = g \) in \( G \).
Consequently, \( g_ng' \to gg' \) for all \( g' \in G \); since the topology of \( \Iso(G,d) \) is that of pointwise convergence, \( g_n \to g \) in \( \Iso(G,d) \).
\end{remark}

\subsection{Modifying metrics}\label{ssec:new-metric}

In dealing with homogeneous metric spaces, a useful technique is the use of pseudometrics on groups; we show how to use these to modify metrics.

A \introd{pseudometric} is a function that satisfies all the conditions required of a metric, except perhaps the condition that \( d(x,y) = 0 \) implies that \( x = y \).
Let \( \dot d \) be a left-invariant pseudometric on a topological group \( G \).
We define the \emph{kernel} of \( \dot d \) to be the set \( \{x \in G : d(x,e) = 0\} \), and say that \( \dot d \) on \( G \) is \emph{continuous} if \( \dot d(x_n, y) \to \dot d(x,y) \) for all \( y \in G \) whenever \( x_n \to x \) in \( G \), \emph{semiproper} if \( \{ x \in G : \dot d(x,e) = 0\} \) is compact, and \emph{proper} if \( \{ x \in G : \dot d(x,e) < C\} \) is relatively compact for all \( C \in \R^+ \).

Given a pseudometric space \( (M, \dot d) \), we define the \emph{ball} \( B_{\dot d}(x, r) \) to be the set \( \{ y \in M : \dot d(x,y) < r \} \); then \( B_{\dot d}(x, r) \) is open if \( \dot d \) is continuous.
Further, given pseudometric spaces \( (M_1, \dot d_1) \) and \( (M_2, \dot d_2) \), we say that a bijection \( f: M_1 \to M_2 \) is an \emph{isometry} if \( \dot d_2(f x_1, f y_1) = \dot d_1(x_1, y_1) \) for all \( x_1, y_1 \in M_1 \).

\begin{lemma}\label{lem:pseudometrics}
Suppose that \( (M,d) \) is a homogeneous metric space, that \( G \) is a transitive closed subgroup of \( \Iso(M,d) \), and that \( K \) is the stabiliser in \( G \) of a point \( o \) in \( M \).
Then \( \dot d: G \times G \to [0, \pinfty) \), defined by
\[
\dot d(x,y) \coloneqq  d(xo, yo)
\qquad\forall x, y \in G,
\]
is a continuous left-invariant pseudometric on \( G \), and
\begin{enumerate}
  \item \( \bigcap_{x \in G} xKx^{-1} = \{e\} \);
  \item \( \dot d(x,e) = 0 \) if and only if \( x \in K \);
  \item \( \dot d(x,y) = \dot d(xk, yk') \) for all \( x, y \in G \) and \( k,k' \in K \);
  \item the topology induced by \( d \) on \( G/K \) coincides with the quotient topology on \( G/K \).
\end{enumerate}

Conversely, if \( \dot d \) is a continuous left-invariant pseudometric on a connected metrisable topological group \( G \), then \( K \coloneqq \{x \in G: \dot d(x,e) = 0\} \) is a closed subgroup of \( G \), {and \( \{x \in G: \dot d(x,y) = 0\} = yK \);} further,  (iii) holds.
Define the function \( d: G/K \times G/K \to [0, \pinfty) \) by
\begin{equation}\label{eq:pseudometric-to-metric}
d(xK, yK) \coloneqq  \dot d(x,y)
\qquad\forall x, y \in G;
\end{equation}
then \( d \) is a metric on the set \( G/K \), and \( G \) acts continuously and transitively by isometries on \( (G/K,d) \).
Further, the subgroup \( N \coloneqq  \bigcap_{x \in G} xKx^{-1} \) is closed and normal  in \( G \), and acts trivially on \( G/K \), so that \( G/N \) may be identified with a transitive subgroup of \( \Iso(G/K, d) \).
Finally, suppose that the topology induced by \( d \) on \( G/K \) coincides with the quotient topology on \( G/K \).
Then
\begin{enumerate}\setcounter{enumi}{4}
  \item the Busemann metric \( d_\epsilon \) on \( G/N \), given by
\[
d_\epsilon(g, h)
\coloneqq \sup\{d(g(q), h(q))\expe^{-{\rho_{[\ell]}(o, q)}/{\epsilon}} : q\in G/N\},
\]
is admissible on \( G/N \); and
\item
the subgroup \( K/N \) of \( G/N \) is compact.
\end{enumerate}
\end{lemma}

\begin{proof}
Take \( x,y,z \in G \).
Then \( \dot d(x,y) \geq 0 \) and \( \dot d(x,y) = \dot d(y,x) \) by definition; further,
\[
\dot d(x,z) = d(xo,zo) \leq d(xo,yo) + d(yo,zo) = \dot d(x,y) + \dot d(y,z),
\]
and
\[
\dot d(x,y) = d(xo,yo) = d(zxo, zyo) =\dot d(zx, zy).
\]
Hence \( \dot d \) is a left-invariant pseudometric on \( G \).

The compactness of \( K \) and items (i) and (iv) are proved in Section \ref{ssec:metric-coset-spaces}; items (ii) and (iii) follow immediately from the definitions.

Conversely, if \( \dot d \) is a continuous left-invariant pseudometric on a topological group \( G \), and \( K = \{x \in G: \dot d(x,e) = 0\} \), then
\[
\dot d(x^{-1}y,e) = \dot d(y, x) \leq \dot d(y, e) + \dot d(e, x) =0,
\]
for all \( x, y \in K \) whence \( K \) is a subgroup of \( G \), which is closed since \( \dot d \) is continuous.
{
Observe that
\[
\dot{d}(x,y) = 0 \iff \dot{d}(y^{-1}x,e) = 0 \iff y^{-1}x \in K \iff x \in yK.
\]}
Moreover,
\[
\dot d(xk, yk')
\leq \dot d(xk, x) + \dot d(x, y) + \dot d(y, yk')
= \dot d(x,y)
\]
and
\[
\dot d(x, y)
\leq \dot d(x, xk) + \dot d(xk, yk') + \dot d(yk', y)
= \dot d(xk, yk'),
\]
so (iii) holds.
It follows immediately that \( \dot d \) induces a metric \( d \) on \( G/K \), by the formula
\[
d(xK, yK) = \dot{d}(x,y)
\qquad\forall x,y \in G,
\]
and \( G \) acts transitively and continuously by isometries on \( (G/K,d) \).
It is evident that \( N \) is closed and normal, and is precisely the subgroup of \( G \) that stabilises every point of \( G/K \), hence \( G/N \) acts effectively, transitively and isometrically on \( G/K \), which we may identify with \( (G/N)/(K/N) \) by a standard isomorphism theorem.

Now we suppose that the topology induced by \( d \) on \( G/K \) coincides with the quotient topology on \( G/K \), that is, that \( d \) is admissible, and prove (v) and (vi).
We may and shall suppose that \( N \) is trivial, otherwise we just divide it out.
By Remark \ref{rem:fourth-topology}, the topology on \( G \) coincides with the relative topology as a subgroup of \( \Iso(G/K,d) \), and Theorem \ref{thm:general-isometry-group} implies (v) and (vi).
\end{proof}

The reader may wish to check that, in the first part of the preceding lemma, if \( d \) is proper on \( G/K \), then \( \dot d \) is proper on \( G \),
while in the second part, \( \dot d \) is semiproper if and only if \( d \) is proper.

\begin{definition}\label{def:admissible-pseudometric}
A left-invariant continuous pseudometric \( \dot{d} \) on a topological group \( G \) with kernel \( K \) is said to be \emph{admissible} if the topology of the induced metric on \( G/K \) coincides with the quotient topology on \( G/K \).
Equivalently, the sets \( B_{\dot{d}}(x,r)K \), where \( x \in G \) and \( r \in \R^+ \) form a base for the topology of \( G/K \), or the sets \( B_{\dot{d}}(x,r) \), where \( x \in G \) and \( r \in \R^+ \) form a base for the subtopology of \( G \) of all right-\( K \)-invariant sets of the topology.
\end{definition}

By the proof of the previous lemma and the continuity of \( \dot{d} \), the sets \( B_{\dot{d}}(x,r) \) satisfy \( B_{\dot{d}}(x,r) = B_{\dot{d}}(x,r)K \)  and are open in \( G \).
Hence the key to showing admissibility is to show that if \( U \) is an open neighbourhood of \( x \) in \( G \) and \( U = UK \), then \( B_{\dot{d}}(x,r) \subseteq U \) when \( r \) is small enough.

\begin{corollary}\label{cor:convergence-on-compacta}
If \( \dot{d} \) is a left-invariant continuous admissible pseudometric on \( G \), and \( x_n \to x \) in \( G \) as \( n \to \pinfty \), then
\( \sup_{y \in K_c} \dot{d}(x_ny,xy) \to 0 \) for all compact subsets \( K_c \) of \( G \).
\end{corollary}

\begin{proof}
Let \( K \) be the kernel of \( \dot{d} \), and \( d \) be the corresponding metric on \( G/K \).
Convergence of a sequence in \( G \) implies pointwise convergence and hence locally uniform convergence of the corresponding sequence of elements of \( \Iso(G/K,d) \), by Proposition \ref{prop:3-topols}.
\end{proof}

We show now that if \( G \) is a locally compact group and \( d_G \) is an admissible left-invariant metric on \( G \) that is also right-\( K \)-invariant, where \( K \) is a closed bounded subgroup of \( G \), then the quotient space \( G/K \) may be equipped with a metric in a natural way.

\begin{lemma}\label{lem:quotient-metric}
Let \( K_0 \) and \( K \) be compact subgroups of a locally compact group \( G \) such that \( K_0 \subseteq K \).
Suppose that \( \dot{d} \) is a left-invariant right-\( K \)-invariant continuous admissible pseudometric on \( G \) with kernel \( K_0 \), and take \( C \coloneqq  \sup\{ \dot{d}(x,y) : x,y \in K \} \) (which is finite).
Then
\[
\ddot{d} (x, y) \coloneqq  \min \{ \dot{d} (xk, yk') : k, k' \in K \}
\qquad\forall x, y  \in G
\]
defines a left-invariant continuous admissible pseudometric on \( G \) with kernel \( K \), and
\[
\dot{d}(x,y) - C \leq \ddot{d}(x, y) \leq \dot{d}(x,y)
\qquad\forall x,y  \in G.
\]
\end{lemma}

\begin{proof}
Since \( \dot{d} \) is continuous and right-\( K \)-invariant and \( K \) is compact, we may write
\begin{equation}\label{eq:equivalent ddot d}
\ddot{d} (x, y)
= \min \{ \dot{d}(xk, y) : k \in K \}
= \min \{ \dot{d}(x, yk') : k' \in K \}.
\end{equation}

Clearly \( \ddot{d} \) is left-invariant and \( \ddot{d}(x, y) \geq 0 \) and \( \ddot{d}(x, y) = \ddot{d}(y, x) \) for all \( x, y \in G \).
Further,
\[
\dot{d}(xk, zk') \leq \dot{d}(xk, y) + \dot{d}(y, zk') ,
\]
and taking minima shows that \( \ddot{d}(x, z) \leq \ddot{d}(x, y) + \ddot{d}(y, z) \) for all \( x, y, z \in G \).
Suppose that \( \ddot{d} (x, y) = 0 \); then there exists \( k \in K \) such that \( \dot{d}(x, y k) = 0 \).
Hence \( x \in yK_0 \) and \( xK = yK \).

We now show that the pseudometric \( \ddot{d} \) is admissible.
By the remark following Definition \ref{def:admissible-pseudometric}, it suffices to consider \( x \in G \) and an open neighbourhood \( U \) of \( x \) in \( G \) such that \( U = UK \), and show that some \( B_{\ddot{d}}(x,r) \subseteq U \).
Clearly \( U = UK_0 \), and since \( \dot{d} \) is admissible, there exists \( r \in \R^+ \) such that \( x \in B_{\dot{d}}(x,r) \subseteq U \).
From \eqref{eq:equivalent ddot d},
\[
B_{\ddot{d}}(x,r) = \bigcup_{k \in K} B_{\dot{d}}(xk,r) = B_{\dot{d}}(xk,r) K \subseteq UK = U,
\]
so \( \ddot{d} \) is admissible.
\end{proof}

\begin{corollary}\label{cor:quotient-metric}
Let \( K_0 \) and \( K \) be compact subgroups of a locally compact group \( G \) such that \( K_0 \subseteq K \).
If \( d_{0} \) is a \( G \)-invariant admissible metric on \( G/K_0 \) such that
\[
d_{0}(xkK_0, ykK_0) = d_{0}(xK_0, yK_0)
\qquad\forall x, y  \in G \quad\forall k \in K,
\]
then \( d \), defined by
\[
d_{}(xK, yK)
= \min\{ d_{0}(xkK_0,yk'K_0) : k,k' \in K \}
\qquad\forall x,y  \in G,
\]
is a \( G \)-invariant admissible metric on \( G/K \), and the projection \( \pi: G/K_0 \to G/K \) is a \( G \)-equivariant rough isometry; more precisely,
\[
d_{0}(xK_0,yK_0) - C
\leq d_{}(xK, yK)
\leq  d_{0}(xK_0,yK_0)
\]
for all \( x,y  \in G \).
\end{corollary}

\begin{proof}
This follows from the preceding lemma, translated into the language of metrics using Lemma \ref{lem:pseudometrics}.
Indeed, the metric \( d_{0} \) induces a pseudometric \( \dot{d} \) on \( G \) which satisfies the conditions required in the previous lemma; the previous lemma constructs another pseudometric \( \ddot{d} \) on \( G \); finally \( d_{} \) is the metric on \( G/K \) induced by \( \ddot{d} \).
\end{proof}

A locally compact topological group \( G \) is said to be \introd{metrisable} if there is a metric \( d_G \) on \( G \) that induces the topology of \( G \); it is known that \( d_G \) may be taken to be left-invariant (see \cite[Theorem 8.3]{Hewitt-Ross}), and we shall always do so.
Conversely, it is easy to check that if \( d_G \) is a left-invariant metric on \( G \), then \( G \) with the topology induced by \( d_G \) is a topological group (that is, multiplication and inversion are continuous) if and only if \( d_G \) satisfies the condition \( d_G(x_n, x) \to 0 \) as \( n \to \pinfty \) implies that
\( d_G(x_n z,xz) \to 0 \) as \( n \to \pinfty \) for all \( z \in G \).

Lemma \ref{lem:quotient-metric} suggests the question whether, given a pseudometric group \( (G,d) \) and a closed \( d \)-bounded subgroup \( K \) of \( G \), it is possible to adjust \( d \) on \( G \) to obtain a pseudometric that is both left-invariant and right-\( K \)-invariant.
This is the point of the next lemma.
We say that a closed subgroup \( K \) of \( G \) is compact modulo a closed central subgroup \( Z \) of \( G \) provided that \( K/ (K \cap Z) \) is compact.

\begin{lemma}\label{lem:G-Z-K-invariant-metric}
Let \( Z \) be a closed central subgroup of a locally compact group \( G \), and let \( \dot{d} \) be a left-invariant continuous admissible pseudometric on \( G \).
Suppose that \( K \) is a subgroup of \( G \) that is compact modulo \( Z \), and set
\begin{equation}\label{eq:def-d-sub-K}
C \coloneqq  \sup_{k \in K} \inf_{z \in Z} \dot{d}(kz,e) .
\end{equation}
Then \( C \) is finite.
Further, \( \dot{d}_K \), defined by
\[
\dot{d}_K(g,h) \coloneqq  \sup_{k \in K} \dot{d}(gk, hk)
\qquad\forall g,h \in G,
\]
is a left-invariant, right-\( K \)-invariant, continuous, admissible pseudometric on \( G \), and
\begin{equation}\label{eq:quasi-pseudo-isometry}
\dot{d}(g,h) \leq \dot{d}_K(g,h) \leq \dot{d}(g,h) + 2 C
\qquad\forall g, h \in G.
\end{equation}
\end{lemma}

\begin{proof}
In light of the existence of suitable sections for quotients of locally compact groups (see, for instance, \cite{Kehlet}), there is a compact subset \( K_c \) of \( K \) such that \( K \subseteq K_c Z \).
Then
\[
     \sup_{k\in K} \dot{d}(gk, hk)
\leq \sup_{k\in K_c}\sup_{z\in Z} \dot{d}(gkz, hkz)
=    \sup_{k\in K_c} \dot{d}(gk, hk)
\leq \sup_{k\in K} \dot{d}(gk, hk),
\]
and so
\begin{equation}\label{eq:dK-redefined}
\dot{d}_K(g,h) = \sup_{k\in K_c} \dot{d}(gk, hk)
\qquad\forall g, h \in G.
\end{equation}
Similarly,
\[
C
= \sup_{k \in K} \inf_{z \in Z} \dot{d}(kz,e)
= \sup_{k \in K_c} \inf_{z \in Z} \dot{d}(kz,e)
\leq \sup_{k \in K_c} \dot{d}(k,e)
< \pinfty.
\]

By definition, given \( k \in K \) and \( z \in Z \),
\[
\dot{d}(gk, hk)
= \dot{d}(gkz, hkz)
\leq \dot{d}(gkz, g) + \dot{d}(g,h) + \dot{d}(h, hkz)
\leq \dot{d}(g,h) + 2\dot{d}(kz,e)
\]
for all \( g, h \in G \); we obtain \eqref{eq:quasi-pseudo-isometry} for \( \dot{d}_K \) by optimising in \( z \).
In particular, we see that \( \dot{d}_K \) is finite.
We may easily check that \( \dot{d}_K \) is a pseudometric on \( G \).
It remains to show that \( \dot{d}_K \) is admissible and continuous.

The continuity of \( \dot{d}_K \) follows immediately from \eqref{eq:dK-redefined} and Corollary \ref{cor:convergence-on-compacta}.

To check admissibility, we suppose that \( x \in G \) and \( V \) is an open neighbourhood of \( x \) in \( G \), and take \( U = VK \).
We need to show that \( B_{\dot{d}_K}(x,r) \subseteq U \) when \( r \) is small enough.
But \( B_{\dot{d}_K}(x,r) \subseteq B_{\dot{d}}(x,r) \) and the admissibility of \( \dot{d} \) implies that \( B_{\dot{d}}(x,r) \subseteq U \) when \( r \) is small enough.
\end{proof}

The next result follows immediately from Lemmas \ref{lem:pseudometrics} and \ref{lem:G-Z-K-invariant-metric}.

\begin{corollary}\label{cor:left-G-right-K-invariant metric}
Suppose that \( K_o \) is a compact subgroup of a locally compact group \( G \), and \( K \) is a subgroup of \( G \) that contains \( K_o \) and is compact modulo the centre of \( G \).
If \( d \) is a \( G \)-invariant metric on \( G/K_o \), then there is a metric \( d' \) on \( G/K_o \) such that the identity mapping on \( G/K_o \) is a rough isometry from \( (G/K_o,d) \) to \( (G/K_o,d') \) and \( d' \) is left-invariant and right-\( K \)-invariant, in the sense that
\[
d'(gg'kK_o, gg''kK_o) = d'(g'K_o, g''K_o)
\qquad\forall g,g',g'' \in G \quad\forall k \in K.
\]
\end{corollary}

We have seen that, starting from a homogeneous metric space \( (M,d) \), we may construct various transitive isometry groups \( H \), which are metrisable locally compact groups, and realise \( M \) as \( H/K \), where \( K \) is the stabiliser of a point \( o \) in \( M \).
Conversely, given a quotient space \( H/K \) of a metrisable locally compact group, it is natural to ask whether \( H/K \) may be given the structure of a metric space on which \( H \) acts isometrically.
The following corollary answers this question.

\begin{corollary}\label{cor:H/K-is-metrisable}
Given a compact subgroup \( K \) of a connected metrisable locally compact group \( H \) such that \( \bigcap_{h \in H} hKh^{-1} = \{e\} \), there exists an admissible metric \( d \) on \( H/K \) such that \( H \) may be identified with a closed subgroup of \( \Iso(H/K,d) \).
\end{corollary}

\begin{proof}
First, if \( H \) is metrisable, then, as noted above, there is a left-invariant admissible metric \( d_1 \) on \( H \).
We modify \( d_1 \) if necessary so that it is right-\( K \)-invariant, by defining \( d_2 \) by
\[
d_2(x,y) \coloneqq  \max \{ d_1(xk, yk) : k \in K \}
\qquad\forall x,y \in H.
\]
Lemma \ref{lem:G-Z-K-invariant-metric} shows that \( d_2 \) is a metric.
By Lemma \ref{lem:quotient-metric}, \( d \), defined by
\begin{equation}\label{eq:def-metric}
d(xK,yK) \coloneqq  \inf\{ d_2(xk, yk') : k, k' \in K\}
\qquad\forall xK, yK \in H/K,
\end{equation}
is an admissible metric on \( H/K \), and \( H \) acts isometrically on \( (H/K,d) \).
The condition on the conjugates of \( K \) ensures that the action is effective.

Take a sequence \( (h_n)_{n\in\N} \) in \( H \) and \( j \in \Iso(H/K,d) \) such that \( h_n h K \to j h K \) for all \( h \in H \); to see that \( H \) is closed in \( \Iso(H/K,d) \), we must show that \( j \in H \).

First, take \( g \in H \) such that \( jeK = gK \); then \( g^{-1}h_n h K \to g^{-1}j h K \) for all \( h \in H \), and \( g^{-1}jeK = eK \).
Next, since \( g^{-1}h_n eK \to eK \), \eqref{eq:def-metric} implies that we may choose \( g_n \in H \) such that \( g^{-1} h_n eK = g_n eK \) and \( g_n \to e \) in \( H \).
Write \( h_n' \) for \( g_n^{-1} g^{-1} h_n \) and \( j' \) for \( g^{-1} j \).
Thus, \( h'_n hK \to j' hK \) for all \( h \in H \) and \( h'_n eK = eK \) for all \( n \in \N \), that is, \( h'_n \in K \).
Since \( K \) is compact, we may, after passing to a subsequence if necessary, suppose that \( h'_n \to k \) in \( K \).
It then follows that \( khK = \lim_n h'_n hK = j'hK \) for all \( h \in H \), and since \( \Iso(H/K,d) \) acts effectively, \( j' = k \).
Thus \( j \in H \), and \( H \) is closed in \( \Iso(H/K,d) \).
\end{proof}

Now we discuss covering maps of homogeneous metric spaces.
If \( M^\sharp \) and \( M \) are connected topological spaces, then a continuous surjection \( \pi: M^{\sharp} \to M \) is said to be a \emph{covering map} provided that, for all sufficiently small neighbourhoods \( U \) in \( M \), there are disjoint neighbourhoods \( V_z \) in \( M^{\sharp} \), where \( z \in Z \), such that \( \pi^{-1}(U) = \bigsqcup_{z \in Z} V_z \) and the restriction of \( \pi \) to \( V_z \) is a homeomorphism onto \( U \).

In the case of connected topological groups, which we write \( H^\sharp \) and \( H \), we take \( \pi \) to be a homomorphism, with kernel \( Z \).
In this case, \( Z \) is discrete and normal in \( H^{\sharp} \),
which implies that \( Z \) is central, since \( \{ x \in G : xzx^{-1} = z\} \) is both open and closed in \( G \) for each \( z \in \Ker \pi \).
For such \( \pi \), for all sufficiently small neighbourhoods \( U \) in \( H \), there is a neighbourhood \( V \) in \( H^{\sharp} \) such that the restriction of \( \pi \) to \( V \) is a homeomorphism onto \( U \) and  \( \pi^{-1}(U) = \bigsqcup_{z \in \Ker \pi} zV \).

When we deal with homogeneous metric spaces, universal covering spaces need not exist; consider, for example, an infinite product of circles.

\begin{lemma}\label{lem:covering-space}
Suppose that \( \pi: G^{\sharp} \to G \) is a covering map of connected locally compact topological groups, \( K^\sharp \) and \( K \) are closed subgroups of \( G^{\sharp} \) and \( G \), and \( K^\sharp \) is an open subgroup of \( \pi^{-1}K \).
Then the canonical projection \( \pi^\sharp : G^\sharp / K^\sharp \to G / K \)  is a covering map.
Suppose that \( d \) is a \( G \)-invariant metric on \( G/K \).
Then for all \( \epsilon \in \R^+ \), there exists a \( G^\sharp \)-invariant metric \( d^{\sharp} \) on \( G^{\sharp} / K^\sharp \) such that
\[
d^{\sharp}(x,y) - \epsilon
\leq d(\pi x,\pi y)
\leq d^{\sharp}(x,y)
\qquad\forall x, y \in G^{\sharp} / K^\sharp.
\]

If \( K_1 \) is a connected subgroup of \( G \) that contains \( K \) and \( d \) is right-\( K_1 \)-invariant, then \( d^\sharp \) may be taken to be right-\( \pi^{-1} K_1 \)-invariant.
\end{lemma}

\begin{proof}
The mapping \( \pi^\sharp \) is the composition of two mappings: the canonical projection from \( G^\sharp / K^\sharp \) to \( G^\sharp / \pi^{-1} K \) and the canonical isomorphism of \( G^\sharp / \pi^{-1} K \) with \( G/K \), which is a homeomorphism.
It is obvious that we can use the latter map to transfer the metric from \( G/K \) to \( G^\sharp / \pi^{-1} K \) so that the homeomorphic isomorphism is also an isometry, so it suffices to deal with the canonical projection.
To simplify the notation, we replace \( G^\sharp \), \( K^\sharp \), \( \pi^\sharp \) and \( \pi^{-1} K \) by \( G \), \( K \), \( \pi \)  and \( K^\flat \).
Thus \( K \) is an open subgroup of \( K^\flat \), which is a closed subgroup of \( G \), and we consider the projection \( \pi: G/ K \to G/K^\flat  \); we need to prove that \( \pi \) is a covering map and show how to lift a metric on \( G/K^\flat \) to \( G/K \).

From the hypotheses, we may find points \( z_j \in K^\flat \) such that \( K^\flat = \bigsqcup_j z_j K \).
Moreover, there is an open set \( U \) in \( G \) such that \( U = U^{-1} \) and \( U^2 \cap K^\flat = K \).
Then the sets \( U z_j K \) are open in \( G \) and disjoint, and the mapping \( u z_j K \mapsto u K^\flat \) is a homeomorphism from \( U z_j K \) to \( UK^\flat \), and then by the \( G \)-equivariance of \( \pi \), the restriction of \( \pi \) to a set \( gU z_j K \), where \( g \in G \), is a homeomorphism to \( gUK^\flat \).
It follows that \( \pi \) is a covering map.

Next, a metric \( d \) on \( G/K^\flat \) gives rise to a pseudometric \( \dot d \) on \( G \) with kernel \( K^\flat \).
We may define a (not necessarily proper) metric \( d_1 \) on \( G/K \) by choosing \( \epsilon \) small enough that \( B(eK^\flat, \epsilon) \subseteq UK^\flat \), and then setting
\[
d_1(xK, yK) \coloneqq
\begin{cases}
  \min\{ d(\pi x, \pi y) , \epsilon)
           & \text{if \( x, y \in gU K \) for some \( g \in G \)} \\
  \epsilon & \text{otherwise}.
\end{cases}
\]
We leave to the reader the task of checking that a suitable linear combination \( d_{G/K} \) of \( \dot d \) and \( d_1 \) has the required properties.
\end{proof}

\begin{lemma}\label{lem:unwinding-product-groups}

Let \( H \) be a locally compact group with closed subgroups \( S_1 \) and \( S_2 \) such that \( H = S_1 \cdot S_2 \), and let \( H_\times = S_1 \times S_2 \).
Let \( \omega: H_\times \to H \) be the mapping \( (s_1,t) \mapsto s_1t^{-1} \).
Then \( \omega \) is a homeomorphism.
Further, if \( \dot{d} \) is a left-invariant and right-\( S_2 \)-invariant continuous admissible pseudometric on \( H \), then \( \dot{d}_\times \), given by
\[
\dot{d}_\times((s_1,s_2), (s_1', s_2')) = \dot{d}(s_1s_2^{-1}, s_1's_2'^{-1})
\qquad\forall s_1,s_1' \in S_1 \quad\forall s_2,s_2' \in S_2,
\]
is a left-invariant continuous admissible pseudometric on \( S_1 \times S_2 \).
\end{lemma}

\begin{proof}

Since \( s_2 \mapsto s_2^{-1} \) is a homeomorphism of \( S_2 \) and \( \psi: (s_1,s_2) \to s_1s_2 \) is a homeomorphism, \( \omega \) is a homeomorphism from \( H_\times \) to \( H \).
Since \( \dot{d} \) is a left-\( S_1 \)-invariant and right-\( S_2 \)-invariant pseudometric, \( \dot{d}_\times \) is a left-\( (S_1 \times S_2) \)-invariant pseudometric.
Since \( \dot{d} \) is continuous, so is \( \dot{d}_\times \).

Let \( K \) be the kernel of \( \dot d \) and \( K_\times \) be the kernel of \( \dot d_\times \).
From Lemma \ref{lem:pseudometrics},
\begin{align*}
p_\times K_\times
&= \{ q_\times \in H_\times : \dot d_\times(p_\times,q_\times) =0 \} \\
\noalign{\noindent{and}}
p K
&= \{ q \in H : \dot d(p,q) =0 \} .
\end{align*}
for all \( p_\times \in H_\times \)  and all \( p \in H \).
The definition of \( \dot{d}_\times \) then implies that
\begin{equation*}
q \in pK_\times \iff
\dot d_\times(p,q) =0 \iff
\dot d(\omega(p),\omega(q)) =0 \iff
\omega(q) \in \omega(p) K.
\end{equation*}
It follows that \( \omega \) induces a homeomorphism from \( H_\times/K_\times \) to \( H/K \),
which is an isometry by construction.
The admissibility of \( \dot{d} \) and that of \( \dot{d}_\times \) are therefore equivalent.
\end{proof}

We note conversely that if the map \( \omega: S_1 \times S_2 \to H \), given by \( \omega(s_1,s_2) = s_1s_2^{-1} \) is an isometry from the pseudometric group \( S_1 \times S_2 \) to the pseudometric group \( H \), then the pseudometric on \( H \) must be right-\( S_2 \)-invariant.

\subsection{Simply transitive isometry groups}\label{ssec:simply-transitive}

Here we are interested in the question whether a homogeneous metric space admits a simply transitive isometry group.

\begin{theorem}\label{thm:isometry-isomorphism}
Let \( (M,d) \) be a homogeneous metric space, \( H \) denote \( \Iso(M,d) \) and \( K \) denote the stabiliser of a base point \( o \) in \( M \); let \( G \) be a group.
Then the following are equivalent:
\begin{enumerate}
\item there is a simply transitive action of \( G \) on \( M \) by isometries;
\item there is a left-invariant metric \( d_G \) on \( G \) such that \( (G, d_G) \) is isometric to \( (M, d) \);
\item there is a monomorphism \( \alpha : G \to H \) such that \( \alpha(G) \cap K = \{ e_H \} \) and \( H = \alpha(G) K \).
\end{enumerate}
In addition, if (i), (ii) and (iii) hold, and \( G \) is a topological group, then the following are equivalent:
\begin{enumerate}\setcounter{enumi}{3}
  \item the metric \( d_G \) of (ii) is admissible;
  \item \( \alpha \) is a homeomorphism from \( G \) to \( \alpha(G) \), equipped with the relative topology as a subset of \( H \).
\end{enumerate}
Finally if (i) to (v) all hold, then \( \alpha(G) \) is closed in \( H \).
\end{theorem}

\begin{proof}
Suppose that (i) holds, and denote the action by \( \alpha \).
We define the left-invariant pull-back metric \( d_G \) on \( G \) by
\[
d_G (g, g') \coloneqq  d( \alpha(g) o, \alpha(g') o)
\qquad\forall g, g' \in G;
\]
then the map \( g \mapsto \alpha(g) o \) is an isometry from \( (G,d_G) \) to \( (M,d) \), so (ii) holds.

Assume that (ii) holds, and that \( F \colon (G, d_G) \to (M, d) \) is an isometry.
By composing with a translation of \( G \) if necessary, we may suppose that \( F(e) =o \).
For \( g \in G \), define the mapping \( \alpha(g): M \to M \) by the formula
\[
\alpha(g)(p) \coloneqq  F(g F^{-1}(p))
\qquad\forall p \in M.
\]
It is straightforward to check that (iii) holds.

Now assume that (iii) holds.
Then \( \alpha(G) \) is transitive since every element of  \( H \) may be written as \( \alpha(g) k \) where \( g \in G \) and \( k \in K \), and \( \alpha(G) \) is simply transitive since \( \alpha(G) \cap K = \{e_H\} \).
So \( G \) acts simply transitively by isometries on \( (M,d) \), and (i) is proved.

Now assume that (i), (ii) and (iii) hold, and that \( G \) is a topological group.
Consider, for \( g \) and a net of elements \( g_\nu \) in \( G \), the following statements:
\begin{enumerate}\renewcommand{\labelenumi}{(\alph{enumi})}
  \item \( g_\nu \to g \) in \( G \) as \( \nu \to \infty \);
  \item \( g_\nu g' \to gg' \) in \( G \) as \( \nu \to \infty \) for all \( g' \in G \);
  \item \( d_G(g_\nu g' , gg') \to 0 \) as \( \nu \to \infty \) for all \( g' \in G \);
  \item \( d(\alpha(g_\nu g')(o), \alpha(gg')(o)) \to 0 \) as \( \nu \to \infty \) for all \( g' \in G \);
  \item \( \alpha(g_\nu)(p) \to \alpha(g)(p) \) in \( M \) as \( \nu \to \infty \) for all \( p \in M \);
  \item \( \alpha(g_\nu) \to \alpha(g) \) in \( H \).
\end{enumerate}
Since \( G \) is a topological group, (a) and (b) are equivalent, while (c) and (d) are equivalent by definition, (d) and (e) are equivalent by writing \( p = g'(o) \), and (e) and (f) are equivalent by definition of the topology on \( H \).
Further, (b) and (c) are equivalent if and only if \( d_G \) is admissible.

If \( d_G \) is admissible, then (a) and (f) are equivalent, so \( \alpha \) is a homeomorphism of \( G \) onto its image in \( H \).
Conversely, if the topology of \( \alpha(G) \) induced by that of \( G \) coincides with that induced by \( H \), then (a) and (f) are equivalent, and so \( d_G \) is admissible.

We now suppose that if (i) to (v) all hold, and show that \( \alpha(G) \) is closed in \( H \).
We take a net \( (g_\nu) \) in \( G \) such that \( \alpha(g_\nu) \to h \) in \( H \), and need to prove that \( h \in \alpha(G) \).
Now \( h = \alpha(g)k \), where \( g \in G \) and \( k \in K \); by replacing \( g_\nu \) by \( g^{-1}g_\nu \) if necessary, we may assume that \( \alpha(g_\nu) \to k \) in \( H \), and must prove that \( k = e \).
Now
\[
d_G(g_\nu , e_G) = d( \alpha(g_\nu)o,o) \to d(ko,o) = 0,
\]
so \( g_\nu \to e_G \), as required.
\end{proof}

The theorem above shows that, if we are looking for metric groups that are isometric to a given homogeneous space, and whose topology is related to that of the homogeneous space, it will suffice to look for closed subgroups of the isometry group.
Actually, since our homogeneous spaces are assumed to be connected, it will suffice to look for closed subgroups of the connected component of the identity in the isometry group.
The conditions in the theorem will appear quite often, and so it is useful to have some additional notation.

\begin{definition}\label{def:various-products}
If \( G \) and \( K \) are subgroups of a group \( H \), then \( GK \) denotes the subset \( \{ gk: g \in G, k \in K\} \) of \( H \).

We write \( H = G \cdot K \) to indicate that \( G \) and \( K \) are closed subgroups of a locally compact group \( H \), such that the mapping \( (g,k) \mapsto gk \) from the set
\( G \times K \) with the product topology to \( H \) is a homeomorphism.

If \( H = G \cdot K \)  and moreover \( G \) is normal in \( H \),
then we write \( H = G \rtimes K \) and call \( H \) the semidirect product of \( G \) and \( K \).
\end{definition}

\begin{remark}\label{rem:more-on-G-dot-K}
First, if \( H = G \cdot K \), then \( G \) is homeomorphic to \( H/K \).
Further, if \( H \) is connected, so are \( G \) and \( K \).

Next, the subgroup \( K \) is not required to be compact in Definition \ref{def:various-products}.
However, if \( K \) is compact, then the condition that the mapping is a homeomorphism in the definition of the expression \( H = G \cdot K \) is satisfied provided only that the mapping is a bijection.
Indeed, if \( (g_\nu) \) and \( (k_{\nu'}) \) are nets such that \( g_\nu \to g \) in \( G \) and \( k_{\nu'} \to k \) in \( K \), then
\( g_\nu k_{\nu'} \to gk \) in \( H \) since multiplication is continuous.
Conversely if \( G \) is closed and \( K \) is compact, and \( g_\nu k_{\nu} \to h \) in \( H \), then, by passing to a subnet, we may assume that \( k_{\nu} \to k \) in \( K \), and then \( g_\nu \to hk^{-1} \) in \( H \) and so in \( G \) since \( G \) is closed; if the net \( k_{\nu} \) had two limits, then we could factorise \( h \) as a product \( gk \) in two distinct ways, which contradicts bijectivity.
\end{remark}

\begin{remark}\label{rem:metric-groups-give-products}
If \( (G,d) \) is a metric group, and \( K \) is the stabiliser in \( \Iso(G,d) \) of \( e \) in \( G \), then \( G \) and \( K \) are both closed in \( \Iso(G,d) \), by Remark \ref{rem:metric-groups-are-closed} and Theorem \ref{thm:isom-group-props}, and have trivial intersection, so we may write \( \Iso(G,d) = G \cdot K \).
Hence if \( H \) is a subgroup of \( \Iso(G,d) \) that contains \( G \), then we may write \( H = G \cdot K_0 \), where \( K_0 \) is the stabiliser in \( H \) of \( e \) in \( G \).
\end{remark}

The next lemma is about groups that nearly act simply transitively.

\begin{lemma}\label{lem:H=alphaGK}
Suppose that \( \alpha \) is a continuous monomorphism of a connected locally compact group \( G \) into a connected metrisable locally compact group \( H \), and that \( K \) is a compact subgroup of \( H \).
Let \( \omega: G \times K \to H \) be the continuous mapping
\( (g,k) \mapsto \alpha(g)k^{-1} \).
Suppose also that there are neighbourhoods \( U_0 \) of \( e_G \) in \( G \) and \( V_0 \) of \( e_K \) in \( K \) such that, if \( e_G \in U \subseteq U_0 \) and \( e_K \in V \subseteq V_0 \), then the restricted mapping \( \omega|_{U \times V} \) is a bijection onto a neighbourhood of \( e_H \) in \( H \).
Then
\begin{enumerate}
\item \( H=\alpha(G)K \),
\item there is an open set \( U_1 \) in \( G \) containing \( e_G \) such that the restriction \( \omega|_{U_1 \times K} \) is a homeomorphism onto its image, with the relative topology;
\item \( \alpha^{-1}(K) \) is discrete in \( G \) and \( G/\alpha^{-1}(K) \) is homeomorphic to \( H/K \);
\item \( \alpha(G)\cap K \) is finite if and only if \( \alpha(G) \) is closed in \( H \); and
\item if \( \alpha(G) \cap K = \{e_H\} \), then \( H = \alpha(G) \cdot K \).
\end{enumerate}
\end{lemma}

\begin{proof}
To prove (i), we equip the connected space \( H/K \) with an \( H \)-invariant metric, by using Corollary \ref{cor:H/K-is-metrisable}, so that \( G \) acts isometrically on \( H/K \).
By assumption, \( \omega(G \times K) \) contains a neighbourhood of \( e_H \), so the image of the base point \( K \) in \( H/K \) under \( \alpha(G) \) contains a neighbourhood of the base point, whence \( G \) acts transitively on \( H/K \) by part (ii) of Theorem \ref{thm:general-isometry-group}, and \( H = \alpha(G)K \).
Hence (i) holds.

Now we prove (ii).
By compactness, there exist finitely many points \( k_1 \), \dots, \( k_I \) in \( K \) such that \( K = \bigcup_i k_i V_0 \).
Suppose that \( i \) in \( \{1, \dots, I\} \).
If \( \alpha(U_0) \cap k_i V_0 \neq \emptyset \), then there exist \( u_i \in U_0 \) and \( v_i \in V_0 \) such that \( \alpha(u_i) = k_i v_i \).
Now if \( u \in U_0 \cap \alpha^{-1}(K) \), then there exist \( j \) in \( \{1, \dots, I\} \) and \( v \in V_0 \) such that \( \alpha(u) = k_j v \).
We deduce that
\[
\alpha(u)v^{-1} = k_j = \alpha(u_j) v_j^{-1},
\]
whence \( u = u_j \).
Thus
\[
U_0 \cap \alpha^{-1}(K)
= \{ u_j:
\alpha(U_0) \cap k_j V_0 \neq \emptyset, \ k_j = \alpha(u_j) v_j^{-1} \},
\]
which is a finite set.
It follows that there exists a neighbourhood \( U'_0 \) of \( e_G \) in \( G \) such that \( \alpha(U'_0) \cap K = \{e_H\} \).
We take a neighbourhood \( U_1 \) of \( e_G \) in \( G \) such that \( U_1^{-1}U_1 \subseteq U'_0 \).
Now if \( g_1, g_2 \in U_1 \) and \( k_1, k_2 \in K \) are such that
\( \alpha(g_1)k_1^{-1} = \alpha(g_2)k_2^{-1} \), then \( \alpha(g_2^{-1} g_1) = k_2^{-1} k_1 \) and \( g_2^{-1} g_1 \in U'_0 \) and \( k_2^{-1} k_1 \in K \).
It follows that \( g_1 = g_2 \) and \( k_1 = k_2 \), and \( \omega|_{U_1 \times K} \) is a bijection.
The hypothesis on \( \omega \) implies that \( \omega|_{U_1 \times K} \) is open, and it is continuous by definition.

Part (iii) follows immediately from (ii).
Indeed, \( \alpha^{-1}(K) \cap U_1 = \{e_G\} \), so the point \( e_G \) is isolated in \( \alpha^{-1}(K) \).
By a translation argument, every point of \( \alpha^{-1}(K) \) is isolated, and
\( \alpha^{-1}(K) \) is discrete.
Further, standard isomorphism theorems show that \( \alpha \) induces a continuous bijection, \( \dot\alpha \) say, of \( G / \alpha^{-1}(K) \) onto \( H/K \).
The hypo\-thesis on \( \omega \) implies that \( \dot\alpha \) is open, so \( \dot\alpha \) is indeed a homeomorphism.

We now prove one implication of (iv).
If \( \alpha(G) \) is closed in \( H \), then \( \alpha(G)\cap K \) is a closed subgroup of \( K \), so is compact.
Now \( G \) is connected and locally compact by hypothesis, and so is \( \sigma \)-compact; further, \( \alpha^{-1}(K) \) is a discrete subgroup of \( G \), and hence there is a neighbourhood \( W \) of \( e_G \) such that the sets \( xW \), as \( x \) ranges over \( \alpha^{-1}(K) \), are disjoint.
It follows that \( \alpha^{-1}(K) \) is countable, whence \( \alpha(G)\cap K \) is a countable compact group, hence finite (see the notes and remarks at the end of this chapter).

Conversely, to complete the proof of (iv), we assume that \( \alpha(G)\cap K \) is finite, and take a net \( (g_\nu) \) in \( G \) such that \( \alpha(g_\nu) \to h \) in \( H \); we must show that \( h = \alpha(g^*) \) for some \( g^* \) in \( G \), and \( g_\nu \to g^* \) in \( G \).
By the transitivity of the \( G \) action on \( H/K \), proved in (i), there exists \( g \) in \( G \) such that \( h \in \alpha(g)K \); then \( \alpha(g^{-1}g_\nu) \to \alpha(g^{-1})h \) in \( H \), and, by replacing \( g_\nu \) and \( h \) by \( g^{-1}g_\nu \) and \( \alpha(g^{-1})h \), we may assume that \( h \in K \).
Next, from (ii), if \( \nu \) is large enough, there exists \( \tilde g_\nu \) in \( U_1 \) such that \( \alpha(\tilde g_\nu)K = \alpha(g_\nu)K \), and \( \tilde g_\nu \to e \) in \( G \); by replacing \( g_\nu \) by \( \tilde g_\nu^{-1} g_\nu \), we may assume that \( \alpha(g_\nu) \in K \).
Since \( \alpha(G) \cap K \) is finite, the convergent net \( g_\nu \) is eventually constant, so the limit is in \( G \).

Finally, if \( \alpha(G)\cap K=\{e_H\} \), then \( \alpha(G) \) is closed in \( H \) from part (iv).
By Remark \ref{rem:more-on-G-dot-K}, \( H = \alpha(G) \cdot K \).
\end{proof}

We now clarify when two connected locally compact groups may be made isometric.

\begin{corollary}\label{cor:isometric-groups-complementary-subgroups}
Let \( G_1 \) and \( G_2 \) be connected locally compact groups.
Then \( G_1 \) and \( G_2 \) may be made isometric if and only if there exists a metrisable locally compact group \( H \) with a compact subgroup \( K \) such that \( H = G_1 \cdot K = G_2 \cdot K \).
\end{corollary}

\begin{proof}
If \( G_1 \) and \( G_2 \) may be made isometric, then we may assume that the isometry sends \( e_1 \) to \( e_2 \), and that they have a common isometry group, \( H \) say.
Then we may take \( K \) to be the stabiliser of \( e_1 \) in \( G_1 \) or \( e_2 \) in \( G_2 \).

Conversely, given \( H \) and \( K \), Corollary \ref{cor:H/K-is-metrisable} constructs a metric \( d \) on \( H/K \) so that \( H \) acts isometrically on \( (H/K, d) \).
Since \( G_j \) acts simply transitively on \( H/K \), we may transport the metric \( d \) on \( H/K \) to \( G_j \) by the formula
\[
d_j(x,y) = d(xK, yK)
\qquad\forall x, y \in G_j,
\]
and obtain left-invariant metrics on \( G_j \), when \( j \) is \( 1 \) or \( 2 \).
Now \( (G_1,d_1) \) and \( (G_2, d_2) \) are both isometric to \( (H/K,d) \), and so are isometric to each other.
\end{proof}

\subsection{Invariant measure and growth}\label{ssec:polgr}
Every locally compact group \( G \) admits a Haar measure \( \mu \), that is, a left-invariant Radon measure that gives positive mass to all nonempty open sets; the Haar measure is unique up to a multiplicative constant.

If \( K \) is a compact subgroup of a locally compact group \( G \), with a left-invariant Haar measure \( \mu \), and \( \pi:G\to G/K \) is the quotient map, then there is a unique \( G \)-invariant Radon measure \( m \) on \( G/K \) such that
\begin{equation}\label{eq03162337}
m(U) = \mu(\pi^{-1}(U) )
\end{equation}
for all Borel subsets  \( U \) of \( G/K \);  see \cite[\S15]{Hewitt-Ross}.
From Theorem~\ref{thm:general-isometry-group} and Corollary \ref{cor:H_0-transitive}, if \( (M, d) \) is a homogeneous metric space and \( G \) is \( \Iso(M, d) \), then \( M \) may be identified with \( G/K \) for some compact subgroup \( K \) of \( G \).
Thus every homogeneous metric space \( (M,d) \) admits a unique (up to scalar multiplication) Radon measure that is invariant under \( \Iso(M,d) \).

A compactly generated locally compact group \( G \) with Haar measure \( \mu \) is said to be of \introd{polynomial growth} if there is a compact generating neighbourhood \( U \) of the identity in \( G \) such that
\begin{equation}\label{eq:def-poly-growth-group}
\mu(U^n) \le C n^Q
\qquad\forall n\in\Z^+.
\end{equation}
If \( G \) is of polynomial growth and \( V \) is another compact generating neighbourhood of the identity in \( G \), then the same equation holds but with a possibly different constant \( C \).
From part (i) of Lemma \ref{lem:prelim}, \( m(V_{n}(o,\ell)) \) grows no faster than exponentially in \( n \); however, it may grow only polynomially, or even be bounded.

The following definition is standard, at least for quasigeodesic metrics.

\begin{definition}
Let \( (M,d) \) be a homogeneous metric space.
We say that \( (M,d) \) is of \introd{polynomial growth}  if for a given point and hence for an arbitrary point \( o \) in \( M \),
\begin{equation}\label{eq:poly-growth-space}
m(B(o, r)) \le C r^Q
\end{equation}
for all sufficiently large \( r \).
\end{definition}

At this point, for a metric Lie group we have two notions of polynomial growth, which in general are not equivalent.
For instance, \( \R \) is a group of polynomial growth, but if we define the metric \( d \) on \( \R \) by
\[
d(x,y) \coloneqq  \log(|x-y|+1) \qquad\forall x,y \in \R,
\]
then \( (\R, d) \) is not of polynomial growth.
More generally, \( m(B(o,r)) \) may grow much faster in \( r \) than \( m(V_{n}(o,\ell)) \) grows in \( n \).

A {proper quasigeodesic} homogeneous metric space is of polynomial growth if and only if its isometry group is of polynomial growth.
For general metric spaces, only one implication may be proved, as follows.

\begin{lemma}\label{lem03171107}
If \( M \) is a homogeneous metric space of polynomial growth, and \( G \) is a subgroup of \( \Iso(M,d) \) that acts transitively on \( M \), then \( G \) is of polynomial growth.
\end{lemma}

\begin{proof}
By part (v) of Theorem \ref{thm:general-isometry-group}, we may fix \( o\in M \) and \( \ell \in \R^+ \) such that the set \( U\coloneqq \{f\in G:f(o)\in  \barB(o, \ell)\} \) is a compact neighbourhood of the identity element in \( G \) and
\[
U^n=\{f\in G:f(o)\in V_n(o, \ell)\}.
\]

Let \( \mu \) be a Haar measure on \( G \) and \( m \) be a \( G \)-invariant measure on \( M \) such that \eqref{eq03162337} holds, as discussed at the beginning of this section, and suppose that \( m(B(o, r))\le Cr^Q \) for all sufficiently large \( r \).
Then
\[
\mu(U^n) = m(V_n(o, \ell)) \le C\ell^Q (n+1)^Q
\]
since \( V_n(o, \ell)\subseteq B(o, (n+1) \ell) \).
\end{proof}

We now connect growth to the doubling property.

\begin{definition}
Let \( (M,d) \) be a homogeneous metric space.
We say that \( (M,d) \) is \introd{doubling} if there is a constant \( N \) such that each ball of radius \( 2r \) may be covered by at most \( N \) balls of radius \( r \) for all \( r\in\R^+ \).
We say that \( (M,d) \) is doubling \emph{at small scale} or \emph{at large scale} if the covering property holds for all sufficiently small \( r \) or sufficiently large \( r \).
\end{definition}

Polynomial growth is often linked with the property of being doubling at large scale.
Indeed, if \( (M,d) \) is {proper quasigeodesic}, then it is of polynomial growth if and only if it is doubling at large scale; see, for instance,~\cite{Cornulier-sublinear}.
However, these two notions are not equivalent in our setting.
More precisely, if a metric space \( (M,d) \) is doubling at large scale, it may fail to be of polynomial growth; see Remark \ref{rem:doubling-no-polygr}.
However, if \( (M,d) \) is doubling at large scale and proper, then it is of polynomial growth; see Remark~\ref{rem01041721}.
Conversely, if \( (M,d) \) is of polynomial growth, then it is proper, but it does not need to be doubling at large scale; see Remarks~\ref{rem03202156} and~\ref{rem04111303}.
This paradoxical behaviour reflects the fact that polynomial growth and properness are not quasi-isometric invariants when metrics are not {proper quasigeodesic}.

\begin{remark}\label{rem:doubling-no-polygr}
The space \( (\R,d) \), where \( d \) is given by \( d(x, y)=\min\{|x-y|, 1\} \), is trivially doubling at large scale, but is evidently not of polynomial growth.
\end{remark}

\begin{remark}\label{rem01041721}
If a homogeneous metric space is proper and doubling, then it is of polynomial growth.
Indeed, if one and hence every ball of radius \( 2r \) may be covered by \( N \) balls of radius \( r \), then it may be seen that
\[
m(B(o, r)) \le N m(B(o, 1)) r^{\log_2(N)}
\]
when \( r>1 \).
\end{remark}

\begin{remark}\label{rem03202156}
It is easy to construct homogeneous metric spaces of polynomial growth that are not locally doubling (consider the product \( \prod_{n \in \N} (\R/ 2^{-n}\Z)  \), where each factor has the metric induced from the euclidean metric on \( \R \) and the product has the \( \ell^{\infty} \)-metric) and to construct nonhomogeneous metric spaces of polynomial growth that are not doubling at large scale (consider sparsely branching \( \R \)-trees of unbounded degree).
The next example shows that having polynomial growth does not even imply being doubling at large scale for proper connected homogeneous metric spaces.

Consider the piecewise linear function \( D: [0, \pinfty) \to [0, \pinfty) \) with nodes at \( (0, 0) \), \( (1, 1) \), and \( (x_n, y_n) \), where \( n \in \N \), given by \( x_n = 2^{2^{n+1}} \) and \( y_n = 2^{2^n} \).
The nodes all lie on the graph \( y = x^{1/2} \), so \( D \) is evidently increasing and concave.
Hence \( d(x, y) \coloneqq  D(|x-y|) \) is a translation-invariant metric on \( \R \), and \( |B(x_0, r)| = 2D^{-1}(r) \) for all \( r \in \R^+ \).

Take \( r = y_n \), and consider the ratio
\[
\frac{|B(0, 2r)|}{|B(0, r)|} = \frac{D^{-1}(2y_n)}{D^{-1}(y_n)} = \frac{D^{-1}(2y_n)}{x_n} \, .
\]
We shall now show that the right-hand fraction is unbounded in \( n \), which shows that \( d \) is not a doubling metric.

If \( (x, y) \) lies on the line segment between \( (x_n, y_n) \) and \( (x_{n+1}, y_{n+1}) \), then
\[
\frac{y - y_n}{x - x_n} = \frac{y_{n+1} - y_n}{x_{n+1} - x_n} = \frac{y_n^2 - y_n}{y_n^4 - y_n^2} = \frac{1}{y_n(y_n+1)},
\]
so
\[
x = x_n + y_n(y_n + 1)(y-y_n).
\]
Since \( 2y_n \leq y_{n+1} \), if \( D(x) = 2y_n \), then \( (x, 2y_n) \) lies on the line segment, and so \( x = x_n + x_n(y_n +1) \) and
\[
\frac{D^{-1}(2y_n)}{x_n} = \frac{x}{x_n} = y_n + 2,
\]
which tends to infinity as \( n \) increases.

The same argument also shows that if \( (x, y) \) lies on this line segment, then
\begin{align*}
|B(0, y)| = 2x &= 2x_n + 2y_n(y_n + 1)(y-y_n) \\
&\leq 2y_n^2 + 2y_n y(y_n+1) \leq 2y^2 + 2y^2(y+1),
\end{align*}
and it follows that \( d \) is of polynomial growth.
\end{remark}

\begin{remark}\label{rem04111303}
If \( (M, d) \) is a homogeneous metric space of polynomial growth, then it is proper.
Indeed, if there were a noncompact closed ball \( \barB(p, r) \), then there would be \( \epsilon \in \R^+ \) and points \( x_i \) in \( \barB(p, r) \), where \( i\in \N \), such that \( d(x_i, x_j)>2\epsilon \) if \( i \neq j \).
But then it would follow that
\[
C (r+\epsilon)^Q
\ge m(\barB(p, r+\epsilon))
\ge \sum_{i\in\N} m(B(x_i, \epsilon))
= \pinfty ,
\]
which would be a contradiction.
\end{remark}

\subsection{Notes and remarks}\label{sec:notes-basics}
Here we include some additional comments on the results established above.

\subsubsection*{2.2.\enspace Homogeneous metric spaces}
If \( f \) is a metric preserving mapping of a homogeneous metric space \( (M,d) \), in the sense that condition \eqref{eq:metric_preserving} holds, then \( f \) is surjective; this need not be true for metric preserving mappings of general metric spaces.
The proof involves first composing with an isometry, so that \( f(o) = o \), then using compactness to show that \( f \) is bijective on closed balls (defined relative to the Busemann gauge), and finally letting the radius of the balls go to infinity.

\subsubsection*{2.3.\enspace Metric spaces and coset spaces}
The simple observations of this section raise further questions about isometry groups.
Given a metric space \( (M,d) \) and a transitive isometry group \( G \) of \( M \), let \( o \) be a point in \( M \) and \( K \) be the stabiliser of \( o \) in \( G \).
Is there a left-invariant metric \( d_G \) on \( G \) such that
\[
d(p,q) = \min \{ d_G(g,h) : g(o) = p, h(o) = q \}?
\]
Under what circumstances do the stabilisers of all points in \( M \) have the same diameter?
And if we equip \( M \) with the metric \( d' \) that is defined by the right-hand side of the above formula, is it true that \( G = \Iso(M,d') \)?

\subsubsection*{2.4.\enspace Modifying metrics}
The use of pseudometrics leads to another interpretation of Theorem \ref{thm:isom-group-props}.
Given a metric on a homogeneous metric space \( (M,d) \), we may define a family of pseudometrics \( \dot d_x \), where \( x \) runs over \( M \), on the isometry group \( H \), by setting \( \dot d_x (g,h) \coloneqq  d(gx, hx) \) for all \( g,h \in H \).
If \( g,h \in H \) and \( \dot d_x (g,h) = 0 \) for all \( x \) in \( M \), then \( g^{-1}h \) acts trivially on \( M \), so \( g = h \).
Thus expressions such as \( \sup_{x \in M} \dot d_x(g,h) \), where \( x \) runs over \( M \), only vanish when \( g = h \).
The pseudometrics \( \dot d_x \) satisfy the inequality
\[
\begin{aligned}
\dot d_x (g,h)
&= d(gx, hx)
\leq d(gx, gy)+ d(gy, hy) + d(hy, hx) \\
&\leq d(gy, hy) + 2 d(x,y)
= \dot d_y (g,h) + 2 d(x,y)
\end{aligned}
\]
for all \( g,h \in H \), and if \( M \) is unbounded, then \( \sup_{x \in M} \dot d_x(g,h) \) might well be infinite.
However, the formula given in Theorems \ref{thm:isom-group-props} and \ref{thm:general-isometry-group} is but one of many ways of combining these pseudometrics to get a metric on \( H \).

We will use Corollary \ref{cor:left-G-right-K-invariant metric} later.
For future purposes, we note that if \( K_o \) and \( K \) are compact subgroups of a Lie group \( G \) and \( K_o \subset K \), then there exists a riemannian metric \( d \) on \( G/K_o \) such that
\[
d(gg'kK_o, gg''kK_o) = d(g'K_o, g''K_o)
\qquad\forall g,g',g'' \in G \quad\forall k \in K.
\]
All riemannian metrics are bi-Lipschitz equivalent.

{
The reader may wish to check whether the new metrics produced in  Corollary \ref{cor:quotient-metric} or Lemma \ref{cor:quotient-metric} are proper or derived semi-intrinsic (as defined just before Lemma \ref{lem:Busemann-gauge}) or proper quasigeodesic or geodesic if the initial metric has this property.}

Corollary \ref{cor:H/K-is-metrisable} shows that every pair consisting of a connected metrisable group \( H \) and a suitable compact subgroup \( K \) thereof arises as \emph{a} group of isometries of a homogeneous metric space and the stabiliser of a point therein.
However, this does not answer the subtler question, whether \( H \) is necessarily the whole connected component of the identity in the full group of isometries of the metric space.
For example, if we take \( H \) and \( K \) to be \( \R^n \) and \( \{e\} \), and equip \( H/K \) (that is, \( \R^n \)) with a translation invariant metric, the connected component of the identity in the isometry group may be larger that \( H \).
Indeed, the isometry group of \( \R^n \) equipped with any translation invariant riemannian metric is isomorphic to \( \R^n \rtimes \mathrm{O}(n) \).
However, if we use the \( \ell^\infty \) metric, then we can ensure that the connected component of the isometry group is \( H \).
We do not know whether, given a general pair \( H \) and \( K \) as in the corollary, there exists an admissible \( H \)-invariant metric \( d \) on \( H/K \) such that \( H \) is the connected component of the identity in \( \Iso(H/K,d) \).

\subsubsection*{2.5.\enspace Simply transitive isometry groups}

Definition \ref{def:various-products} deserves a further comment.

In the definition of a semidirect product, it suffices to suppose that \( G \) and \( K \) are closed subgroups and \( G \) is normal, and the mapping \( (g,k) \mapsto gk \) is a bijection.
Indeed, if \( g_\nu \to g \) in \( G \) and \( k_\nu \to k \) in \( K \), then \( g_\nu k_\nu \to gk \) in \( H \) by definition.
Conversely, if \( g_\nu k_\nu \to gk \) in \( H \), then \( Gk_\nu \to Gk \) in the quotient group \( G\backslash H \), which is homeomorphic to \( K \) by \cite[Theorem 5.26]{Hewitt-Ross}, that is, \( k_\nu \to k \) in \( K \), and hence also \( g_\nu \to g \) in \( G \).

The results of this and the previous section offers us an alternative viewpoint on homogeneous metric spaces and their isometry groups.
We begin by taking the basic object to be a metric space \( (M,d) \) with a topology that is compatible with the metric, and showed that a closed subgroup \( H \) of the isometry group that acts transitively is a topological group with a metric compatible with the topology, and that the projection from \( H \) to \( M \) is both a metric projection and a topological projection (that is, it is continuous and open).
However, we might also take the basic object to be a metrisable topological group \( H \), and consider various quotient spaces \( H/K \) with the quotient topologies and quotient metrics, or even just a topological group \( H \) acting on a quotient space \( H/K \) that may be endowed with a metric that is compatible with the quotient topology.

\subsubsection*{2.6.\enspace Invariant measure and growth}

Suppose that \( M \) is the coset space \( G/K \), where \( G \) is a (not necessarily connected) locally compact group and \( K \) is a compact subgroup.
We claim that if \( M \) is compact and countable, then \( M \) is finite.
Indeed, \( M \) admits a \( G \)-invariant Radon measure \( m \), and the regularity of \( M \) implies that there is an open set \( U \) of positive but finite measure.
Since \( M \) is compact, \( m(M) \) must be finite, since it may be covered by finitely many translates of \( U \).
All points of \( M \) have the same measure.
If points had measure \( 0 \), then \( M \) would have measure \( 0 \); hence points have positive measure and the cardinality of \( M \) is \( m(M)/m(\{p\}) \) for any point \( p \).

\section{Lie theory and metric spaces}\label{sec:Lie-theory-metric-space}

This chapter is concerned with homogeneous metric manifolds, which for us are locally euclidean, but not \emph{a priori} smooth.
However, as a consequence of the solution of Hilbert's fifth problem, they are quotient spaces of Lie groups, and hence may be given analytic structures such that the connected component of the identity in the isometry group acts analytically.

In this chapter, we review the Gleason--Iwasawa--Montgomery--Yamabe--Zippin structure theorem of almost connected locally compact groups in Section \ref{ssec:GIMYZ-thm} and then the Iwasawa theory of maximal compact subgroups in \ref{ssec:comp-subgps}.
We prove our first main theorem, that homogeneous metric spaces may be approximated by homogeneous metric manifolds, in Section \ref{ssec:main-1}.
We next look at more Lie theory and its interaction with metric spaces in Sections \ref{ssec:lieth-1} and \ref{ssec:lieth-2}.

We consider more sophisticated Lie theory, such as the Levi and Iwasawa decompositions in Section \ref{ssec:lieth-3} and polynomial growth and amenability in Section \ref{ssec:polygr-amenability}, and see how this enables us to prove our second main theorem, on the finer structure of homogeneous metric manifolds in Section \ref{ssec:proof-thm-main-2}.
We should mention that there have been exhaustive investigations into the homogeneous spaces of semisimple Lie groups and those of solvable Lie groups, but the general case seems less well known.

Many of the results here may be proved by a reduction to the riemannian case and then appealing to the appropriate classical result.
Indeed, as we shall see in Corollary \ref{cor:iso-of-two-manifolds-is-Lie-Riemannian}, if two homogeneous metric manifolds are isometric, then they admit riemannian structures for which they are isometric.
However, classical riemannian geometers did not consider quasi-isometries, and at least some of our theorems are not true in the context of isometries, and are certainly not in the literature (at least in forms that we are able to recognise).

\subsection{The main structure theorem}\label{ssec:GIMYZ-thm}

A locally compact group \( G \) is said to be \introd{almost connected} if \( G/G_0 \) is compact, where \( G_0 \) is the connected component of the identity in \( G \); this is closed and normal.
The isometry groups of homogeneous metric spaces are almost connected, by Theorem \ref{thm:isom-group-props}.

We recall without proof one version of the solution to Hilbert's fifth problem by Gleason, Iwasawa, Yamabe, Montgomery and Zippin.
See, for instance, \cite[Section 4.6]{Montgomery-Zippin-TTG} or \cite[Theorem 1.6.1]{Tao-Hilbert-book}.

\begin{theorem}\label{thm:GIMYZ}
Let \( G \) be an almost connected locally compact group.
Then every neighbourhood \( U \) of the identity in \( G \) contains a compact normal subgroup \( N \) such that \( G/N \) is locally euclidean.
If \( G \) is locally euclidean, then \( G \) may be given a unique analytic structure for which it is a Lie group.
\end{theorem}

The following related result was first stated by Szenthe \cite{Szenthe}.
Unfortunately, there was a mistake in his argument, discovered by Antonyan, but the gap was filled independently by Antonyan and Dobrowolski and by George Michael.
See Glockner's review \cite{Glockner} for the history and location of the proof.

\begin{theorem}\label{thm:Szenthe}
If \( K \) is a compact subgroup of an almost connected locally compact group \( H \), and \( \bigcap_{h \in H} hKh^{-1} = \{e\} \), then the following are equivalent:
\begin{enumerate}
\item
\( H \) is a Lie group and \( H/K \) is a manifold;
\item
\( H/K \) is locally contractible.
\end{enumerate}
\end{theorem}

\begin{corollary}\label{cor:isometry-group-is-Lie-group}
Let \( K \) be a compact subgroup of an almost connected locally compact group \( H \) such that \( H/K \) is connected and \( \bigcap_{h \in H} hKh^{-1} = \{e\} \).
Suppose also that \( H/K \) is locally euclidean or that \( H \) is locally euclidean.
Then \( H \) and hence \( H/K \) may be given analytic structures, compatible with their topologies, such that \( H \) is a Lie group and the action of \( H \) on \( H/K \) is analytic.
\end{corollary}

\begin{proof}
If \( H/K \) is locally euclidean, then so is \( H \), by Theorem \ref{thm:Szenthe}, so we may assume that \( H \) is locally euclidean.

By Theorem \ref{thm:GIMYZ}, we may endow \( H \) with an analytic structure so that \( H \) becomes a Lie group, and this analytic structure on \( H \) induces an analytic structure on \( H/K \).
These analytic structures are compatible with the topologies of \( H \) and \( H/K \).
Further, \( H \) acts analytically on \( H/K \).
\end{proof}

In particular, if \( (M,d) \) is a homogeneous metric manifold, \( H \) is its isometry group, and \( K \) is the stabiliser of a point \( o \) in \( M \) in \( H \), then we may identify \( M \) with \( H/K \) and apply this corollary to deduce that \( H \) and \( M \) have analytic structures such that \( H \) acts analytically on \( M \).

In light of Theorems \ref{thm:GIMYZ} and \ref{thm:Szenthe} and Lemma \ref{lem:contractible} below, there are several criteria which ensure that \( H \) is locally euclidean or \( H/K \) is locally euclidean.

\begin{corollary}\label{cor:iso-of-two-manifolds-is-Lie-Riemannian}
Let \( (M_1, d_1) \) and \( (M_2, d_2) \) be homogeneous metric manifolds.
Then there exist analytic structures and left-invariant analytic infinitesimal riemannian metrics \( g_1 \) and \( g_2 \) on \( M_1 \) and \( M_2 \) such that
\begin{enumerate}
  \item \( \Iso(M_1, d_1) \subseteq \Iso(M_1, g_1) \) and \( \Iso(M_2, d_2) \subseteq \Iso(M_2, g_2) \); and
  \item each isometry \( f \) from \( (M_1, d_1) \) to \( (M_2, d_2) \) is also an isometry from \( (M_1, g_1) \) to \( (M_2, g_2) \).
\end{enumerate}
\end{corollary}

\begin{proof}
Write \( H_1 \) and \( H_2 \) for \( \Iso(M_1,d_1) \) and \( \Iso(M_2,d_2) \), and let \( K_1 \) and \( K_2 \) be the stabilisers in \( H_1 \) and \( H_2 \) of points \( o_1 \) in \( M_1 \) and \( o_2 \) in \( M_2 \); we may and shall identify \( M_1 \) and \( M_2 \) with \( H_1/K_1 \) and \( H_2/K_2 \).
By the previous result, \( H_1 \) and \( H_2 \) are Lie groups and act analytically on \( H_1/K_1 \) and \( H_2/K_2 \).

The action of \( K_1 \) on \( H_1/K_1 \) induces an action of \( K_1 \) on the tangent space to \( H_1/K_1 \) at the point \( K \).
Take an inner product on this tangent space; then by averaging over the action of \( K_1 \) using the Haar measure of \( K_1 \), we may assume that the inner product is \( K_1 \)-invariant.
We may extend this inner product to an analytic left-invariant infinitesimal riemannian metric \( g_1 \) on \( H_1/K_1 \); the key is that if \( h \) and \( h' \) in \( H_1 \) both map \( K_1 \) to \( hK_1 \), then \( h' = hk \) for some \( k \in K_1 \), and the \( K_1 \)-invariance of the inner product at the point \( K_1 \) implies that \( h \) and \( h' \) induce the same inner product at \( hK_1 \).
It follows immediately that \( H_1 \) acts on \( (H_1/K_1, g_1) \) by riemannian isometries, and we conclude that \( \Iso(M_1, d_1) \subseteq \Iso(H_1/K_1, g_1) \).

If there are no isometries from \( (M_1,d_1) \) to \( (M_2,d_2) \), we repeat this argument to put a riemannian metric on \( M_2 \), and there is nothing more to prove.

Otherwise, we take one isometry \( f \) from \( (M_1,d_1) \) to \( (M_2,d_2) \); we may and shall suppose that \( f(o_1) = o_2 \).
Conjugation with \( f \) induces a homeomorphic isomorphism \( F \) of the isometry groups \( \Iso(M_1,d_1) \) and \( \Iso(M_2,d_2) \), and \( F(K_1) = K_2 \).
Hence we may identify \( f \) with the map \( xK_1 \mapsto F(x)K_2 \) from \( H_1/K_1 \) to \( H_2/K_2 \).
The groups \( H_1 \) and \( H_2 \) are Lie groups, and continuous homomorphisms of Lie groups are automatically analytic, so \( F \) is analytic.

We transport the infinitesimal riemannian metric \( g_1 \) on \( H_1/K_1 \) to an infinitesimal riemannian metric \( g_2 \) on \( H_2/K_2 \), and then \( f \) is also an analytic riemannian isometry from \( (M_1, g_1) \) to \( (M_2, g_2) \); further, \( \Iso(M_2, d_2) \subseteq \Iso(M_2, g_2) \).

Finally, if \( f' \) is any isometry from \( (M_1,d_1) \) to \( (M_2,d_2) \), then \( f^{-1} \circ f' \in H_1 \).
It follows that \( f' \) is also a riemannian isometry from \( (M_1,g_1) \) to \( (M_2,g_2) \).
\end{proof}

This result was proved for metric Lie groups in \cite[Proposition 2.4]{Kivioja-LeDonne}.

\subsection{Compact subgroups}\label{ssec:comp-subgps}
We summarise some results about compact subgroups of connected locally compact groups, and establish some corollaries of the structure theorems above.

\begin{lemma}[After Iwasawa {\cite{Iwasawa}}]\label{lem:Iwasawa-max-cpct}
Let \( G \) be a connected locally compact group.
Then every compact subgroup of \( G \) is contained in a maximal compact subgroup \( K \) of \( G \), and all maximal compact subgroups are connected and conjugate to each other.
The subgroup \( K \) is a deformation retract of \( G \).

If \( N \) is a connected normal subgroup of \( G \) and \( K \) is a maximal compact subgroup of \( G \), then \( N \cap K \) is a maximal compact subgroup of \( N \) and \( KN/N \) is a maximal compact subgroup of \( G/N \); conversely, if \( K_N \) is a maximal compact subgroup of \( N \) and \( K_{G/N} \) is a maximal compact subgroup of \( G/N \), then  there exists a maximal compact subgroup \( K \) of \( G \) such that \( K \cap N =  K_N \) and \( KN/N = K_{G/N} \).
\end{lemma}

\begin{proof}
The first result is \cite[Theorem 13]{Iwasawa}, and the second is \cite[Lemma 4.10]{Iwasawa}.
In both cases, the results are first proved for Lie groups and then for groups that admit approximations by Lie groups, as in Theorem \ref{thm:GIMYZ}.
\end{proof}

It follows that the intersection of all maximal compact subgroups is the unique maximal compact normal subgroup of a connected locally compact group.

The following result is almost standard and may be extended (see \cite{Antonyan}); compact contractibility is the only new ingredient.
We say that a topological space \( M \) is \introd{compactly contractible} if, for each compact subset \( S \) of \( M \), there are \( x\in M \) and a continuous map \( F: [0, 1] \times S \to M \) such that \( F(0, s) =s \) and \( F(1, s) = x \) for all \( s \in S \).

\begin{lemma}\label{lem:contractible}
If \( K \) is a compact subgroup of a connected locally compact group \( H \), then the following are equivalent:
\begin{enumerate}
\item
\( K \) is a maximal compact subgroup of \( H \);
\item
\( H/K \) is homeomorphic to a euclidean space;
\item
\( H/K \) is contractible;
\item
\( H/K \) is compactly contractible.
\end{enumerate}
\end{lemma}

\begin{proof}
By \cite[page~188]{Montgomery-Zippin-TTG}, (i) implies (ii).
It is trivial that (ii) implies (iii) and (iii) implies (iv).
We prove that (iv) implies (i) by modifying the argument of \cite[Theorem 1.3]{Antonyan} that shows that (iii) implies (i).

Suppose that (iv) holds.
By \cite{Bagley-Peyrovian}, there is a maximal compact subgroup \( K_0 \) of \( H \) that contains \( K \), and then by \cite[page~188]{Montgomery-Zippin-TTG}, there is a map \( \Phi:\R^n\to H \) such that the map \( (x, y)\mapsto \Phi(x)y \) is a homeomorphism from \( \R^n\times K_0 \) to \( H \).
Hence \( H/K \) is homeomorphic to \( \R^n \times K_0/K \).
The contraction of the compact set \( K_0/K \) in \( H/K \) composed with the projection onto \( K_0/K \) is a contraction of \( K_0/K \).
From Antonyan~\cite{Antonyan}, \( K_0/K \) is contractible if and only if \( K = K_0 \), so \( K \) is maximal.
\end{proof}

\subsection{Proof of Theorem A}\label{ssec:main-1}

In this section, we prove our first main theorem, which we restate in more detailed form.

\begin{theorem}\label{thm:main-1}
Let \( (M,d) \) be a homogeneous metric space, and \( H \) be the connected component of the identity in \( \Iso(M,d) \).
\begin{enumerate}
\item
For all positive \( \epsilon \), there is a connected metric Lie group \( (H_\epsilon,d_\epsilon) \) and a \( (1, \epsilon) \)-quasi-isometry \( \phi:M \to H_\epsilon \).

\item
There are an \( H \)-invariant metric \( d_0 \) of \( M \), a contractible metric manifold \( (M',d') \) and an \( H \)-equivariant projection \( \pi \) from \( (M,d_0) \) to \( (M',d') \), such that the identity mapping is a homeomorphic rough isometry from \( (M,d) \) to \( (M,d_0) \), and \( \pi \) is a homogeneous metric projection with compact fibre, and hence a rough isometry.
\end{enumerate}
\end{theorem}

\begin{proof}
Let \( K_o \) be the stabiliser of a point \( o \) in \( M \), so that \( M \) may be identified with \( H/K_o \).

To prove part (a), take a compact normal subgroup \( N \) of \( H \) such that \( H/N \) is a Lie group and \( No \) has diameter less that \( \epsilon \).
Define
\[
\dot{d}(g,h) \coloneqq  \sup_{k \in N} d(gko, hko).
\]
By Lemma \ref{lem:G-Z-K-invariant-metric}, \( \dot{d} \) is a continuous admissible left-invariant and right-\( K_oN \)-invariant pseudometric on \( H \), and
\[
d(go,ho) \leq \dot{d}(g,h) \leq d(go,ho) + 2 \diam(No)
\qquad\forall g, h \in H.
\]
By the second part of Lemma \ref{lem:quotient-metric}, there is an admissible metric \( d' \) on \( M' \coloneqq  H/K_oN \) such that
\[
d'(gK_oN, hK_oN) = \dot d(g,h)
\qquad\forall g, h \in H.
\]
Hence \( (M,d) \) is \( (1,\epsilon) \)-quasi-isometric to the homogeneous metric manifold \( (M',d') \).
By Theorem \ref{thm:general-isometry-group}, \( (M',d') \) is itself \( (1,\epsilon) \)-quasi-isometric to the metric Lie group \( (H/N, d'_\epsilon) \).

The proof of part (b) is similar.
Let \( K \) be a maximal compact subgroup of \( H \) such that \( K_o \subseteq K \), whence \( K_oN \subseteq K \), and take \( M' \) to be \( G/K \).
As before, we lift the metric \( d \) on \( M \) to a pseudometric \( \dot d \) on \( H \) with kernel \( K_o \), using Lemma \ref{lem:pseudometrics}, and then using Lemma \ref{lem:G-Z-K-invariant-metric}, we define a left-invariant, right-\( K \)-invariant pseudometric \( \ddot d \) on \( H \) by
\[
\ddot d(g,h) \coloneqq  \max\{ \dot d(gk,hk) : h \in K \}.
\]
This then induces \( H \)-invariant metrics \( d_0 \) on \( G/K_o \) and \( d' \) on \( M' \coloneqq  G/K \) by Lemma \ref{lem:quotient-metric}, and the projection from \( G/K_o \) to \( G/K \) has the required properties by construction.
\end{proof}

\subsection{Lie groups and algebras}\label{ssec:lieth-1}
To say more about homogeneous metric spaces, we need more background on Lie theory; we review some aspects thereof in this section.
We begin with some standard definitions and results.

Recall that the \introd{adjoint group} of a Lie algebra \( \mathfrak{h} \) is the Lie group of linear transformations of \( \mathfrak{h} \) generated by the elements \( \exp(\ad(X)) \), where \( X \in \mathfrak{h} \).
Recall also that if \( H \) is a Lie group with Lie algebra \( \mathfrak{h} \), and \( \mathfrak{g} \) is a subalgebra of \( \mathfrak{h} \), then there is an \introd{analytic subgroup} (or Lie subgroup) \( G \) of \( H \) whose Lie algebra is \( \mathfrak{g} \), but \( G \) need not be closed.
Next, if \( G \) is an analytic subgroup of \( H \), then \( G \) with its own Lie structure is analytically immersed, but not necessarily embedded, in \( H \).
Of course, \( G \) is embedded if and only if it is closed.
In light of this correspondence between Lie groups and algebras, we denote the Lie algebra of a Lie group \( G \) by the corresponding fraktur letter \( \mathfrak{g} \).

We recall also that a discrete normal subgroup \( \Gamma \) of a connected Lie group \( G \) is central, since \( \{ x \in G : x\gamma x^{-1} = \gamma\} \) is both open and closed in \( G \) for each \( \gamma \in \Gamma \).
Hence if \( G \) is connected and \( \Gamma \) is a discrete central subgroup, then a discrete subgroup \( \Delta \) of \( G \) that contains \( \Gamma \) is central in \( G \) if and only if \( \Delta/\Gamma \) is central in \( G/\Gamma \).

Finally, we recall that the differential of a homomorphism \( \phi \) of Lie groups is a homomorphism of the Lie algebras, written \( \phi_* \).

\begin{definition}\label{def:torus}
A \emph{torus} or \emph{toral group} is a connected compact abelian Lie group, that is, a finite power of the multiplicative group of complex numbers of modulus \( 1 \).
A subalgebra \( \mathfrak{t} \) of a Lie algebra \( \mathfrak{h} \) is \emph{compact} if \( \ad(U) \) is semisimple and has purely imaginary eigenvalues on \( \mathfrak{h} \) for all \( U \in \mathfrak{t} \) and is \emph{toral} if it is abelian and compact.
The subgroup \( T \) corresponding to a compact subalgebra need not be compact, but \( \Ad(T) \) is a compact subgroup of \( \Aut(\mathfrak{h}) \), and is a torus if \( \mathfrak{t} \) is toral.
\end{definition}

If \( K \) is a compact subgroup of a connected Lie group \( H \), then \( \mathfrak{k} \) is a subalgebra of \( \mathfrak{h} \), and \( \ad(U) \) is semisimple and has purely imaginary eigenvalues for all \( U \in \mathfrak{k} \).
Indeed, by averaging an arbitrary inner product over \( K \), using the Haar measure, we may produce an \( \Ad(K) \)-invariant inner product on \( \mathfrak{h} \); then \( \Ad(K) \) is a group of orthogonal mappings of \( \mathfrak{h} \).
Hence if \( U \) in \( \mathfrak{k} \), then \( \exp(t\ad(U)) \) is semisimple with eigenvalues of modulus \( 1 \) for all \( t \in \R \), and \( \ad(U) \) is semisimple with purely imaginary eigenvalues.
If moreover \( K \) is a torus, then \( K \) is abelian and \( \mathfrak{k} \) is abelian; in this case we may simultaneously diagonalise \( \ad(K) \) acting on the complexification of \( \mathfrak{g} \).
(For information about complexifications of Lie algebras, see, for example, \cite[p.~47]{Varadarajan}.)

In general, the implicit use of an inner product to construct complements of subspaces that are invariant under the action of a compact group \( K \), or to decompose a space into a direct sum of minimal invariant subspaces, or to show that \( \ad(U) \) acts semisimply with purely imaginary eigenvalues for all \( U \) in its Lie algebra \( \mathfrak{k} \) will be referred to here as \introd{Weyl's unitarian trick}, though for Weyl this was just the starting point.
See \cite[p.~342]{Varadarajan} for more information.
Quite often the compact group \( K \) will be a torus, and we usually write \( T \) rather than \( K \) in this case.

Finally we recall that, if \( G \) is a Lie group with Lie algebra \( \mathfrak{g} \), then the \introd{radical} \( R \) of \( G \) is the maximal connected solvable normal subgroup of \( G \), while the \introd{nilradical} \( N \) is the maximal connected nilpotent normal subgroup of \( G \); both are closed.
Their Lie algebras \( \mathfrak{r} \) and \( \mathfrak{n} \) are the maximal solvable and nilpotent ideals of \( \mathfrak{g} \), also called the radical and nilradical or \( \mathfrak{g} \).
The existence of these ideals may be established by showing that the sum of nilpotent or solvable ideals is a nilpotent or solvable ideal respectively, whence the sum of all nilpotent or solvable ideals is the largest nilpotent or solvable ideal respectively; it may then be seen that \( \mathfrak{r} \) and \( \mathfrak{n} \) are characteristic ideals, in the sense that \( \phi(\mathfrak{r}) = \mathfrak{r} \) and \( \phi(\mathfrak{n}) = \mathfrak{n} \) for all automorphisms \( \phi \) of \( \mathfrak{g} \), which implies that they are normal in \( G \), whether or not \( G \) is connected.
Sometimes we write \( R = \rad(G) \) and \( N = \nil(G) \), or \( \mathfrak{r} = \rad(\mathfrak{g}) \) and \( \mathfrak{n} = \nil(\mathfrak{g}) \).

\begin{remark}\label{rem:derivations-radicals-1}
It is well-known that \( [\mathfrak{g}, \mathfrak{r}] \subseteq \mathfrak{n} \).
\end{remark}

For these results and much more, see Bourbaki \cite[pp.~44--47 and p.~354]{Bourbaki1-3} or Varadarajan \cite[pp.~204--207 and 244--245]{Varadarajan}.

We will need a structural result concerning tori in a connected Lie group \( H \); this illustrates the power of Lie theory in establishing results that are of interest in our study of homogeneous metric spaces.

\begin{lemma}\label{lem:central torus}
Let \( H \) be a Lie group with nilradical \( N \).
If \( T \) is a normal torus in \( H \), then \( T \subseteq N \).
If \( K \) is a maximal compact subgroup of \( N \), then \( K \) is a normal torus in \( H \), and central in the connected component of the identity in \( H \).
\end{lemma}

\begin{proof}

Let  \( \mathfrak{t} \) and \( \mathfrak{n} \) be the Lie algebras of  \( T \) and \( N \).
Since \( \mathfrak{t} \) and \( \mathfrak{n} \) are nilpotent ideals, so is \( \mathfrak{t} + \mathfrak{n} \), and since \( \mathfrak{n} \) is the maximal nilpotent ideal, \( \mathfrak{t} \subseteq \mathfrak{n} \), that is, \( T \subseteq N \).

We now take a maximal compact subgroup  \( K \) of \( N \), and show that \( K \) is normal in \( H \).
As \( K \) is also connected and nilpotent, \( K \) is a torus.

Let \( Z \) be the centre of \( N \), which is closed and connected \cite[Corollary 3.6.4]{Varadarajan}, and so of the form \( T \times V \), where \( T \) is a torus and \( V \) is a vector space.
Then \( T \) is the unique maximal compact subgroup of \( Z \).
Since \( KZ/Z \) is a compact subgroup of the simply connected nilpotent group \( N/Z \), whose only compact subgroup is trivial, \( K \subseteq Z \), and hence \( K = T \).
Now \( Z \) is characteristic in \( H \), in the sense that \( hZh^{-1} = Z \) for all \( h \in H \), and so \( hTh^{-1} \) is also a maximal compact subgroup of \( Z \); thus \( hTh^{-1} = T \).
It follows that \( K = T \) is normal in \( H \).

The  differential \( \phi_* \) of a continuous automorphism \( \phi \) of the torus \( T \) is a linear mapping of the Lie algebra \( \mathfrak{t} \) that preserves the lattice \( \Lambda \) of points \( U \) such that \( \exp(U) = e \); these mappings form a discrete subgroup of \( \mathrm{GL}(\mathfrak{t}) \).
Hence \( \Aut(T) \) is a discrete group.
The mapping \( h \mapsto ( t \mapsto hth^{-1}) \) is continuous from \( H \) to \( \Aut(T) \), and so the connected component of the identity in \( H \) lies in its kernel.
In other words, \( T \) is central in the connected component of the identity in \( H \).
\end{proof}

\subsection{Lie theory and metric spaces}\label{ssec:lieth-2}

We return to the situation that arises in the context of isometry groups.

The main result of this section, Corollary \ref{cor:isometry-isomorphism-Lie}, is an algebraic criterion for when a Lie group \( G_2 \) may be made isometric to a metric Lie group \( (G_1, d_1) \).

The material in this section is largely an extension to the case of more general metrics of ideas that go back many years to deal with riemannian Lie groups, which may be found in Helgason \cite{Helgason-DGLGSS} or Kobayashi and Nomizu \cite{Kobayashi-Nomizu-1, Kobayashi-Nomizu-2}.

\begin{lemma}\label{lem:Lie-algebra-criterion}
Suppose that \( K \) is a compact subgroup of a connected Lie group \( H \) and denote by \( \pi \) the quotient map from \( H \) to \( H/K \).
Let \( G \) be an analytic subgroup of \( H \) (not necessarily closed) such that
\( \mathfrak{h} = \mathfrak{g}\oplus\mathfrak{k} \) as vector spaces.
Then
\begin{enumerate}
\item \( H=GK \),
\item the map \( \pi|_G:G\to H/K \) is a covering map,
\item \( G \) is closed in \( H \) if and only if \( G\cap K \) is finite, and
\item if \( H/K \) is simply connected,  then \( H = G \cdot K \).
\end{enumerate}
\end{lemma}
\begin{proof}
The derivative of the mapping \( (X, Y) \mapsto \exp(X) \exp(Y) \) from \( \mathfrak{g} \oplus \mathfrak{k} \) to \( H \) is nonsingular at \( 0 \), whence \( H \), \( G \) and \( K \) satisfy the hypotheses of Lemma \ref{lem:H=alphaGK} (with \( \alpha \) taken to be the identity mapping).
Part (i) follows from Lemma \ref{lem:H=alphaGK} (i).

Lemma \ref{lem:H=alphaGK} (iii)  implies that \( \pi|_G: G\to H/K \) is a covering map, which proves (ii); part (iii) is just Lemma \ref{lem:H=alphaGK} (iv).

Finally, if \( H/K \) is simply connected, then the covering map \( \pi|_G \) is a homeomorphism, whence \( G\cap K=\{e_H\} \).
Part (iv) now follows from Lemma \ref{lem:H=alphaGK} (v).
\end{proof}

We remind the reader that when \( H = G \cdot K \), the spaces \( G \) and \( H/K \) are homeomorphic, and if \( H \) is connected, so is \( K \).

\begin{corollary}\label{cor:isometry-isomorphism-Lie}
Let \( {G_1} \) and \( {G_2} \) be connected simply connected Lie groups, let \( d_{1} \) be an admissible left-invariant metric on \( {G_1} \), let \( H \coloneqq \Iso({G_1},d_{1}) \), and let \( K \) be the stabiliser of \( e_{1} \) in \( H \).
The following are equivalent:
\begin{enumerate}
  \item \( {G_2} \) may be made isometric to \( ({G_1}, d_{1}) \);
  \item there is a Lie group monomorphism \( \alpha: G_2 \to H \) such that we may write \( H = G_1 \cdot K = \alpha(G_2) \cdot K \);
  \item there is a Lie algebra monomorphism \( \tau: \mathfrak{g}_2 \to \mathfrak{h} \) such that \( \tau(\mathfrak{g}_2) \oplus \mathfrak{k} = \mathfrak{h} \).
\end{enumerate}
\end{corollary}

\begin{proof}
This follows by combining Lemma \ref{lem:Lie-algebra-criterion} above with Theorem \ref{thm:isometry-isomorphism} (we may write \( H = G_2 \cdot K \)) and Corollary \ref{cor:isometry-group-is-Lie-group} (the isometry group of a metric Lie group is a Lie group).
\end{proof}

In the context of riemannian metrics, this result was well-known.

We conclude with the remark that, given a homeomorphism \( f: G_1 \to G_2 \) of Lie groups, there is a corresponding homeomorphism \( \tilde f: \tilde G_1 \to \tilde G_2 \) of their universal covering groups that induces an isomorphism of their fundamental groups.
Conversely, if \( f: \tilde G_1 \to \tilde G_2 \) is a homeomorphism that gives rise to an isomorphism of their fundamental groups, then \( G_1 \) and \( G_2 \) are homeomorphic.
By replacing homeomorphism with riemannian isometry, we see that Lie groups \( G_1 \) and \( G_2 \) may be made isometric if and only if \( \tilde G_1 \) and \( \tilde G_2 \) may be made isometric with an isometry that induces an isomorphism of the fundamental groups of \( G_1 \) and \( G_2 \).

\subsection{Decompositions of Lie groups}\label{ssec:lieth-3}

We are going to deal with semidirect products \( R \rtimes L \), and refer the reader to Definition \ref{def:various-products} for the details.
We shall also use the following nomenclature.

\begin{definition}
Suppose that \( \Gamma \) is a subgroup of the semidirect product \( R \rtimes L \).
We say that \( \Gamma \) is \introd{strongly central} if both \( (r,e) \) and \( (e,l) \) are central in \( R \rtimes L \) whenever \( (r,l) \in \Gamma \).
\end{definition}

It will be useful to recall some features of the \introd{Levi decomposition} of a connected Lie group \( G \).
Write \( \mathfrak{g} \) for the Lie algebra of \( G \).
The Lie algebra of the universal covering group \( \tilde G \) of \( G \) is also \( \mathfrak{g} \), and \( G \) is a quotient of \( \tilde G \) by a discrete central subgroup \( \Gamma \).
The Levi decomposition writes \( \mathfrak{g} \) as the sum \( \mathfrak{r} \oplus \mathfrak{l} \), where \( \mathfrak{r} \) is the radical and \( \mathfrak{l} \) is a semisimple subalgebra of \( \mathfrak{g} \), known as a Levi subalgebra.
While \( \mathfrak{r} \) is uniquely determined, \( \mathfrak{l} \) need not be, but all choices of \( \mathfrak{l} \) are conjugate under the adjoint group of \( \mathfrak{g} \).

Let \( \tilde R \) and \( \tilde L \) be the analytic subgroups of \( \tilde G \) and \( R \) and \( L \) be the analytic subgroups of \( G \) corresponding to \( \mathfrak{r} \) and \( \mathfrak{l} \); \( \tilde R \) and \( R \) are the radicals of \( \tilde G \) and \( G \), while \( \tilde L \) and \( L \) are called Levi subgroups.
The subgroup \( \tilde L \) is closed in \( \tilde G \) but \( L \) need not be closed in \( G \).
Denote \( \Gamma \cap \tilde R \) and \( \Gamma \cap \tilde L \) by \( \Gamma_R \) and \( \Gamma_L \).

The centre \( Z(\tilde L) \) of the simply connected semisimple group \( \tilde L \) is discrete and contains a finite index subgroup \( Z^+(\tilde L) \) which is the intersection of the kernels of all finite dimensional representations of \( \tilde L \); in particular, \( Z^+(\tilde{L}) \) is contained in the kernel of the restriction of the adjoint representation of \( \tilde{G} \) to \( \tilde{L} \), and hence \( Z^+(\tilde{L}) \subseteq Z(\tilde{G}) \cap \tilde{L} \).
Hence \( Z(\tilde{G}) \cap \tilde{L} \) is of finite index in \( Z(\tilde{L}) \).
Similarly we consider \( Z^+(L) \), the intersection of the kernels of all finite dimensional representations of \( L \), and show that \( Z(G) \cap L \) is of finite index in \( Z(L) \).
The subgroups \( Z^+(\tilde{L}) \) and \( Z^+(L) \) do not depend on the choice of \( \tilde{L} \) and \( L \) in the Levi decomposition, since all Levi subgroups are conjugate to each other.

The next lemma summarises many properties of the Levi decomposition.
These are certainly known, but we are not aware of a reference in which they may all be found in the one place.
Hence we hope that our formulation will prove useful.

\begin{lemma}\label{lem:Levi-decomp}
Let \( G \), \( Z(G) \), \( R \), \( L \), \( Z^+(L) \), \( \tilde G \), \( \tilde R \), \( \tilde L \), \( Z(\tilde L) \), \( Z^+(\tilde L) \), \( \Gamma \), \( \Gamma_R \) and \( \Gamma_L \) be as defined above.
Then the following hold.
\begin{enumerate}
  \item \( \tilde R \) and \( \tilde L \) are simply connected and closed in \( \tilde G \), and \( \tilde R \) is normal; further, \( \tilde G \) is the semidirect product \( \tilde R \rtimes \tilde L \) of these subgroups.
  \item \( \tilde R \) and \( \tilde L \) are the universal covering groups of \( R \) and \( L \), and \( R \) and \( L \) may be identified with \( \tilde R / \Gamma_R \) and \( \tilde L / \Gamma_L \).
  \item \( R \) is closed and normal in \( G \), but \( L \) need not be closed.  However, the subgroup \( Z^+(L)\afterbar  L \) is closed in \( G \).
  \item \( G \) may be identified with \( (R \rtimes L) / \Gamma_0 \), where \( \Gamma_0 = \Gamma /(\Gamma_R  \times \Gamma_L) \); moreover, \( | \Gamma_0| = | R \cap L| \).
  \item \( G \) is a semidirect product of its radical and a Levi subgroup if and only if \( R \cap L = \{e\} \) if and only if \( \Gamma_0 = \{e\} \) if and only if \( \Gamma = \Gamma_R \Gamma_L \).
  \item \( R \rtimes L \) is the smallest covering group of \( G \) that is a semidirect product of its radical and a Levi subgroup, in the sense that every covering group that is a semidirect product of its radical and a Levi subgroup also covers \( R \rtimes L \).
  \item \( L \) is closed in \( G \) if and only if \( \Gamma_0 L \) is closed in \( R \rtimes L \) if and only if the projection of \( \Gamma_0 \) onto \( R \) is closed in \( R \).
  \item \( \Gamma_0 \) has a largest strongly central subgroup \( \Gamma_1 \), whose index in \( \Gamma_0 \) is bounded by \( |Z(\tilde L) / Z^+(\tilde L)| \).
      We may identify \( G \) with \( (R \rtimes L / \Gamma_1) / (\Gamma_0/ \Gamma_1) \), which is a finite quotient.
  \item the subgroup \( R \cap L \) is discrete and central in \( L \), and so is finite if \( L \) has finite centre.
      The connected component of the identity in its closure \( (R \cap L)\bar{\phantom{x}} \)
      in \( G \) is central in \( G \).
      If \( \Gamma_0 \) is strongly central in \( R \rtimes L \), then \( R \cap L \) is
      central in \( G \).
\end{enumerate}
\end{lemma}
\begin{proof}
Item (i), the structure of \( \tilde G \), is well-known; see, for instance, \cite[p.~244]{Varadarajan}.
Item (ii) and the first part of item (iii) are also standard; we prove the second part of (iii) below.
Item (iv) is a consequence of a standard isomorphism theorem.
Items (v), (vi) and (vii) are trivial.

To prove item (viii), observe that if \( (r_0,l_0) \in \Gamma \) and \( (e,l_0) \) lies in the centre of \( \tilde G \), then so does \( (r_0,e) \).
We define \( \Gamma_1 \coloneqq  \{(r_0, l_0) \in \Gamma_0 : (e, l_0) \in Z(\tilde{G})\} \); then \( \Gamma_1 \) is a subgroup of \( \Gamma_0 \).

In the semisimple group \( \tilde L \), the set \( Z^+(\tilde L) \) of elements that lie in the kernel of every finite dimensional representation of \( \tilde L \) is a subgroup of finite index in the centre \( Z(\tilde L) \) of \( \tilde L \).
The index of \( \Gamma_1 \) in \( \Gamma_0 \) is bounded by \( Z(\tilde L)/Z^+(\tilde L) \).

Now we prove (ix).
Since \( \mathfrak{r} \cap \mathfrak{l} = \{0\} \) and \( R \) is closed and normal in \( G \), \( R \cap L \) is a closed normal zero-dimensional subgroup of \( L \), so it is discrete and central in \( L \), but it may not be closed in \( R \).
Obviously \( R \cap L \) is finite if \( L \) has finite centre (for instance, if \( L \) is compact).

As noted before the statement of this lemma, \( L \cap Z(G) \) is a subgroup of finite index of \( Z(L) \).
Hence \( R \cap L \cap Z(G) \) is of finite index in \( R \cap L \).
Thus the closures of \( R \cap L \cap Z(G) \) and of \( R \cap L \) in \( G \) have the same connected component of the identity, and the closure of \( R \cap L \cap Z(G) \) in \( G \) is of finite index in the closure of \( R \cap L \) in \( G \).
Since the closure of a central subgroup is central, the closure of \( R \cap L \cap Z(G) \) in \( G \) is central.
We conclude that the connected component of the identity in \( (R \cap L)\bar{\phantom{x}} \) is central, as required.

If moreover \( \Gamma_0 \) is strongly central in \( R \rtimes L \) and \( h \in R \cap L \), then both \( (h,e) \) and \( (e,h) \) in \( R \rtimes L \) map to \( h \) under the canonical quotient mapping, and so \( (h, h^{-1}) \in \Gamma_0 \), whence \( h \) is central in \( G \).

Finally, we prove the second part of (iii).
We repeat the above proofs for the quotient group \( G/Z^+(L)\afterbar  \).
The semisimple subgroup \( L' \) in the Levi decomposition \( R'L' \) of \( G/Z^+(L)\afterbar  \) is such that \( Z^+(L') \) is trivial, and hence \( Z(L') \) is finite, so that \( R' \cap L' \) is finite and \( L' \) is closed in \( G/Z^+(L)\afterbar  \), whence \( Z^+(L)\afterbar L \) is closed in \( G \).
\end{proof}

Note in particular that (iv) and (vii) of the lemma imply that if \( R \cap L \) is finite, then \( L \) is closed in \( G \), while if \( R \cap L \) is infinite, and \( L \) may or may not be closed.
Note also that every connected Lie group \( G \) has a covering group that is a semidirect product of its radical and a Levi subgroup, and the number of leaves in the cover is equal to the cardinality of \( R \cap L \), or equivalently, the cardinality of \( \Gamma_0 \).
By contrast, to obtain a quotient that is a semidirect product of its radical and a Levi subgroup, it may be necessary to factor out a subgroup of positive dimension: this is illustrated by the following example.

\begin{example}\label{ex:twist-not-semi-direct}
Consider the connected, simply connected Lie group \( \tilde{G} \) that is the semidirect product \( \C^n \rtimes (\mathrm{SU}(n) \times \R) \), where the action of \( \mathrm{SU}(n) \times \R \) on \( \C^n \) is given by \( \alpha(u,t) v = \expe^{it} uv  \).
The centre of this group may be identified with the subgroup of \( (\mathrm{SU}(n) \times \R) \) of elements \( (u,t) \) such that \( \expe^{it}u \) is the identity matrix.

The centre \( \Gamma \) of \( \tilde{G} \) is discrete but is not the product of the groups of central elements of the Levi subgroup \( L \) (which is \( \mathrm{SU}(n) \)) and of the central elements of the radical \( R \) (which is \( \C^n \rtimes \R \)); hence the group \( G \coloneqq  \tilde{G}/\Gamma \) has trivial centre and is not a semidirect product of the form \( R \rtimes L \), and has no quotient of the same dimension that is a semidirect product of its radical and a Levi factor.
The group \( \Gamma \) is central, but unless \( n=2 \), it is not strongly central, though by Lemma \ref{lem:Levi-decomp}, it has a subgroup \( \Gamma_1 \) of finite index that is strongly central.

\end{example}

We recall the \introd{Iwasawa decomposition} of a semisimple Lie algebra \( \mathfrak{l} \) and of a corresponding connected semisimple Lie group \( L \).
The Lie algebra \( \mathfrak{l} \) may always be decomposed as a direct sum of three subalgebras:
\[
\mathfrak{l} = \mathfrak{a} \oplus \mathfrak{n} \oplus \mathfrak{k},
\]
where \( \ad(X) \) is semisimple with real eigenvalues for all elements \( X \) in \( \mathfrak{a} \), is semisimple with purely imaginary eigenvalues for all elements \( X \) of \( \mathfrak{k} \), and is nilpotent for all elements \( X \) in \( \mathfrak{n} \).
Further, \( \mathfrak{a} \) is abelian and \( [ \mathfrak{a}, \mathfrak{n}] \subseteq \mathfrak{n} \), so \( \mathfrak{a} \oplus\mathfrak{n} \) is also a subalgebra.
The subalgebra \( \mathfrak{k} \) is in turn a direct sum \( \mathfrak{t} \oplus \mathfrak{k}' \), where \( \mathfrak{t} \), the centre of \( \mathfrak{k} \), is a toral subalgebra and \( \mathfrak{k}' \), the commutator subalgebra of \( \mathfrak{k} \), is a compact semisimple subalgebra.

The analytic subgroups \( A \) and \( N \) corresponding to \( \mathfrak{a} \) and \( \mathfrak{n} \) are closed in \( L \), and simply connected; further,  \( AN \) is solvable, closed, simply connected, and exponential, that is, the exponential mapping is a homeomorphism from \( \mathfrak{a}\oplus \mathfrak{n} \) to \( AN \).
The analytic subgroup \( K \) of \( L \) corresponding to \( \mathfrak{k} \) is also closed, and is a covering group of a compact Lie group; thus it may or may not be compact.
We may always write \( K \) as \( V \times K_c \), where \( V \) is a vector subgroup and \( K_c \) is a compact subgroup, and \( Z(L) \) is a discrete subgroup of \( K \).
The Iwasawa decomposition of \( L \) is the statement that
\begin{equation}\label{eq: Iwasawa-decomp}
L = A \cdot N \cdot K .
\end{equation}
All Iwasawa decompositions of \( L \) or of \( \mathfrak{l} \) are conjugate to each other by an inner automorphism of \( L \) or under the adjoint group of  \( \mathfrak{l} \).

\begin{remark}\label{rem:contractible-Lie-groups}

If \( L \) is a connected semisimple Lie group with Iwasawa decomposition \( A \cdot N \cdot K \), then \( K \) is a deformation refract of the semisimple Lie group \( L \), so \( L \) is contractible or simply connected if and only if \( K \) is.
From the classification of semisimple Lie groups (see, for instance, \cite[Chapter X]{Helgason-DGLGSS}), \( L \) is contractible if and only if it is a product of copies of the universal covering group of \( \mathrm{SL}(2,\R) \).
Other simple Lie groups have compact subgroups that are not contractible.

Thus if \( G \) is a contractible Lie group, then \( G = R \rtimes L \), where \( R \) is its radical and \( L \) a Levi subgroup; both \( R \) and \( L \) are contractible.
For connected solvable Lie groups, it is known that contractibility and simple connectedness coincide, while the contractible Levi factor is as just described.
\end{remark}

It is worth pointing out that, for a simply connected semisimple Lie group \( L \), the Lie algebra of \( V \) is \( \mathfrak{t} \) and that of \( K_c \) is \( \mathfrak{k}' \).
For a general semisimple Lie group \( L \), there is a projection \( \pi \) from its universal covering group \( \tilde{L} \) onto \( L \), and \( \pi(V) \) is the product of a torus (which is absorbed into \( K_c \)) and a vector subgroup of \( V \).

\begin{lemma}\label{lem:Iwasawa-and-Levi}
Let \( G \) be a connected Lie group, \( \mathfrak{r} \oplus \mathfrak{l} \) be a Levi decomposition of \( \mathfrak{g} \) and \( \mathfrak{a} \oplus \mathfrak{n} \oplus \mathfrak{k} \) be an Iwasawa decomposition of \( \mathfrak{l} \).
Let \( L \), \( AN \) and \( K \) be the analytic subgroups of \( G \) corresponding to \( \mathfrak{l} \), \( \mathfrak{a} \oplus \mathfrak{n} \) and \( \mathfrak{k} \).
Then \( AN \) is a closed, solvable, connected and simply connected subgroup of \( G \).
Further, \( Z(L) \) and \( Z^+(L) \) are subgroups of \( K \), \( \bar{K} = ZK \) and \( \bar{L} =  ZL \), where \( Z \) is the connected component of the identity in \( (Z^+(L))\afterbar \).
\end{lemma}

\begin{proof}
Let \( \pi_R \) be the canonical projection of \( G \) onto \( G/R \), where \( R \) is the radical of \( G \), which coincides with \( L^\flat \coloneqq  LR/R \simeq L/(R \cap L) \).
Let \( K^\flat \), \( A^\flat \) and \( N^\flat \) be the subgroups \( \pi_R(K) \), \( \pi_R(A) \) and \( \pi_R(N) \) of \( L^\flat \); then \( A^\flat \cdot N^\flat \cdot K^\flat \) is an Iwasawa decomposition of \( L^\flat \).
Thus \( AN \) and \( A^\flat N^\flat \) are simply connected exponential solvable Lie groups whose Lie algebras may be identified, and the restriction \( \pi_R|_{AN} \) of \( \pi_R \) to \( AN \) is a homeomorphic isomorphism onto \( A^\flat N^\flat \).

It follows immediately that \( AN \) is closed in \( G \).
Take \( a_j \in A \), \( n_j \in N \) such that \( a_j n_j \to g \in G \) as \( j \to \pinfty \); we must show that \( g \in AN \).
Now \( \pi_R(a_jn_j) \to \pi_R(g) \) in \( G/R \), and  the Iwasawa decomposition of \( L^\flat \) implies that there exist \( an \in AN \) such that \( \pi_R(g) = \pi_R(an) \).
The identification of \( AN \) and \( A^\flat N^\flat \) in the first paragraph of this proof implies that \( a_jn_j \to an \) in \( AN \) and hence \( a_jn_j \to an \) in \( G \).

We have already noted that \( Z(L) \) is a discrete subgroup of \( L \); \emph{a fortiori} \( Z^+(L) \) is a discrete subgroup of \( L \), central in \( G \).
The closure \( (Z^+(L))\afterbar \) is a central analytic subgroup of \( G \), and it is immediate that \( (Z^+(L))\afterbar = Z^+(L) Z \), where \( Z \) is the connected component of the identity in \( (Z^+(L))\afterbar \).
We have noted that \( K/Z^+(L) \) is compact, and so there is a compact subset \( S \) of \( K \) such that every element of \( K \) may be written as \( zs \) where \( z \in Z^+(L) \)  and  \( s \in S \).
It follows that
\[
\bar K \subseteq (Z^+(L))\afterbar S \subseteq Z^+(L) Z S = ZK;
\]
it is obvious that \( ZK \subseteq \bar K \), and so equality holds.

Finally to identify \( \bar L \), we see that if \( a_j \in A \), \( n_j \in N \), \( k_j \in K \), and \( a_j n_j k_j \to g \) in \( G \), then  \( \pi_R(a_jnj) \pi_R(k_j) \to \pi_R(g) \) in \( L^\flat \), whence \( \pi_R(a_jn_j) \to \pi_R(an) \) for some \( a \in A \) and \( n \in N \) from the properties of the Iwasawa decomposition of \( L^\flat \), and hence \( a_jn_j \to an \) from the identification of \( AN \) and \( A^\flat N^\flat \).
It is now immediate that \( k_j \) converges in \( G \) to some element of \( ZK \).
\end{proof}

Our next lemma links maximal compact subgroups to the Levi and Iwasawa decompositions.

\begin{lemma}\label{lem:Levi-and-max-cpct}
Suppose that \( G \) is a connected Lie group with radical \( R \).
Then the following hold.
\begin{enumerate}
  \item Given a Levi subgroup \( L \) with Iwasawa decomposition \( ANK \), there exists a maximal compact subgroup \( K_R \) of \( R \) such that \( K \) commutes with \( K_R \); if \( K \) is compact then \( K_R K \) is a maximal compact subgroup of \( G \).
  \item Given a maximal compact subgroup \( K_R' \) of \( R \), there exists a Levi subgroup \( L' \) of \( G \) with Iwasawa decomposition \( A'N'K' \)  such that \( K' \) commutes with \( K_R' \); if \( K' \) is compact then \( K_R' K' \) is a maximal compact subgroup of \( G \).
  \item Given a maximal compact subgroup \( K_0 \) of \( G \), there exists a Levi subgroup \( L \) of \( G \) with Iwasawa decomposition \( ANK \) such that \( K \) commutes with \( K_R \) and \( K_0 \subseteq K_R \bar{K} \), where \( K_R = K_0 \cap R \).
\end{enumerate}
\end{lemma}

\begin{proof}
To prove (i), take any Levi subgroup \( L \) of \( G \); then the group \( R \ltimes L \) is a covering group of \( G \) by Lemma \ref{lem:Levi-decomp}.
It is also a covering group of \( R \ltimes L/Z^+(L) \).
Hence \( G \) is locally isomorphic to \( R \ltimes L/Z^+(L) \).
Observe that two connected closed subgroups of \( G \) commute if and only if the two connected closed subgroups of \( R \ltimes L/Z^+(L) \) with the same Lie algebras commute.

Let \( ANK \) be an Iwasawa decomposition of \( L \); then \( AN(K/Z^+(L)) \) is an Iwasawa decomposition of \( L/Z^+(L) \), and \( K/Z^+(L) \) is a maximal compact subgroup of \( L/Z^+(L) \).
Extend \( K/Z^+(L) \) to a maximal compact subgroup \( K_m \) of \( R \ltimes L/Z^+(L) \).
Then \( K_R \coloneqq  K_m \cap R \) is a maximal compact subgroup of \( R \), and \( K_m R/R \), which is naturally isomorphic to \( K_m/K_R \), is a maximal compact subgroup of \( (R \ltimes L/Z^+(L))/R \), which is naturally isomorphic to \( L/Z^+(L) \).
Under this isomorphism, the image of \( K_mR/R \) is a maximal compact subgroup of \( L/Z^+(L) \) that contains \( K/Z^+(L) \), and hence these subgroups coincide.
Thus \( K_m = (K/Z^+(L)) K_R \), and \( K_R \) is a connected compact solvable normal subgroup of the connected compact Lie group \( K_m \), and hence is a central torus.
It follows that \( K/Z^+(L) \) and \( K_R \) commute, and hence \( K \) and \( K_R \) commute.

If \( K \) is compact, then \( K_RK \) is a compact subgroup of \( G \); further, \( K_RK \cap R \geq K_R \), but \( K_R \) is a maximal compact subgroup of \( R \) and so equality holds.
It follows that \( K_RK \) is a maximal compact subgroup of \( G \) from the fact that \( K_R \) and \( K \) are maximal compact subgroups of \( R \) and \( L \).

Now we prove (ii).
Given another maximal compact subgroup \( K_R' \) of \( R \), there exists \( r \in R \) such that \( K_R' = r K_Rr^{-1} \); then \( r Lr^{-1} \) is a Levi subgroup of \( G \) with Iwasawa decomposition \( r Ar^{-1}r Nr^{-1}r Kr^{-1} \) and \( K_R' \) commutes with \( r Kr^{-1} \), as required.

We prove part (iii) by induction on the dimension of \( R \), the radical of \( G \).
Suppose that the result holds whenever \( \dim(R) < r \), and suppose that \( \dim(R) = r \).
We consider two cases, according to the properties of \( Z^+(L) \).

If \( Z^+(L)\afterbar \) has dimension \( 0 \), then \( Z^+(L) \) is discrete in \( G \).
We write \( G^\flat  \) for \( G / Z^+(L) \) and consider the local isomorphism \( \pi: G  \to G^\flat  \).
Take a Levi decomposition \( R^\flat L^\flat \) of \( G^\flat \); then \( R^\flat \) and \( L^\flat \) are locally isomorphic to the subgroups \( R \) and \( L \) that arise in a Levi decomposition of \( G \), but \( Z^+(L^\flat ) = \{e\} \), which means that the subgroup \( K \) in an Iwasawa decomposition of \( L \) is compact.
It is evident \( \pi(K_0) \) is contained in a maximal compact subgroup of \( G^\flat \), and maximal compact subgroups of \( G^\flat \) are of the form \( K_R K^\flat \).
In this case, the desired result follows.

If \( Z^+(L)\afterbar \) has positive dimension, then \( Z \), the connected component of the identity in \( Z^+(L)\afterbar \), is a nontrivial closed connected normal subgroup of \( R \).
We let \( \pi: G \to G / Z \) be the canonical projection; the radical of the quotient group \( G / Z \) has dimension less than \( r \), while a Levi factor \( L^\flat \) of the quotient is locally isomorphic to a Levi factor of \( G \); the main difference is that \( Z^+(L^\flat) \) is trivial.
The result follows from the inductive hypothesis applied to \( G^\flat  \).
\end{proof}

We also need some information about maximal solvable subalgebras of a Lie algebra which follows from the Levi decomposition and an argument of Mostow.

\begin{lemma}[{After Mostow \cite{Mostow-subgroups}}] \label{lem:maximal-solvable}
Suppose that \( \mathfrak{h} \) is a Lie algebra.
There exist finitely many maximal solvable subalgebras \( \mathfrak{g}_j \) of \( \mathfrak{h} \) such that every maximal solvable subalgebra \( \mathfrak{g} \) is conjugate under the adjoint group to exactly one of the \( \mathfrak{g}_j \).
Exactly one of these subalgebras, \( \mathfrak{g}_0 \) say, has the property that there is a compact subalgebra \( \mathfrak{k} \) of \( \mathfrak{h} \) such that \( \mathfrak{g}_0 + \mathfrak{k} = \mathfrak{h} \).
\end{lemma}

\begin{proof}
Let \( \mathfrak{r} \) be the radical of \( \mathfrak{h} \) and \( \mathfrak{l} \) be a Levi subalgebra of \( \mathfrak{h} \), so that \( \mathfrak{h} = \mathfrak{r} \oplus \mathfrak{l} \).
Denote by \( \pi \) the canonical projection of \( \mathfrak{h} \) onto the quotient  \( \mathfrak{q} \coloneqq  \mathfrak{h}/\mathfrak{r} \), which may be identified with \( \mathfrak{l} \).

If \( \mathfrak{g} \) is a maximal solvable subalgebra of \( \mathfrak{h} \), then \( \mathfrak{r} \subseteq \mathfrak{g} \), since otherwise \( \mathfrak{g} + \mathfrak{r} \) would be a larger solvable subalgebra than \( \mathfrak{g} \).
Further, for subalgebras \( \mathfrak{g} \) of \( \mathfrak{h} \) that contain \( \mathfrak{r} \), \( \mathfrak{g} \) is solvable if and only if \( \pi(\mathfrak{g}) \) is solvable (this relies on that fact that if \( \mathfrak{s}_1 \) and \( \mathfrak{s}_2/\mathfrak{s}_1 \) are solvable, so is \( \mathfrak{s}_2 \)).
Consequently, \( \pi(\mathfrak{g}) \) is a maximal solvable subalgebra of \( \mathfrak{q} \) if and only if \( \mathfrak{g} \) is a maximal solvable subalgebra of \( \mathfrak{h} \).

Mostow \cite{Mostow-subgroups} classified the maximal solvable subalgebras of the semi\-simple Lie algebra \( \mathfrak{q} \) (showing that they correspond to Cartan subalgebras of \( \mathfrak{q} \)), and described finitely many maximal solvable subalgebras \( \mathfrak{s}_j \) of \( \mathfrak{q} \) with the property that every maximal solvable subalgebra is conjugate to exactly one of these.
The maximal solvable subalgebras \( \mathfrak{s} \) of \( \mathfrak{q} \) for which there exists a compact subalgebra \( \mathfrak{k} \) of \( \mathfrak{q} \) such that \( \mathfrak{s}+ \mathfrak{k} = \mathfrak{q} \) are all conjugates under the adjoint group of \( \mathfrak{q} \) of a particular subalgebra \( \mathfrak{s}_0 \), which is a toral extension of the subalgebra \( \mathfrak{a}+ \mathfrak{n} \) of \( \mathfrak{q} \) arising from an Iwasawa decomposition of \( \mathfrak{q} \).

We define \( \mathfrak{g}_j \coloneqq  \pi^{-1} (\mathfrak{s}_j) \) and \( \mathfrak{k}' \) to be the compact subalgebra of \( \mathfrak{l} \) that corresponds to \( \mathfrak{k} \) under the identification of \( \mathfrak{l} \) and \( \mathfrak{q} \).
Then \( \mathfrak{g}_j \) is a maximal solvable subalgebra of \( \mathfrak{h} \) (containing \( \mathfrak{r} \)), and every maximal solvable subalgebra of \( \mathfrak{h} \) is conjugate to one of these.
Further, \( \pi(\mathfrak{g}_0) + \pi(\mathfrak{k}') = \mathfrak{s}_0 + \pi(\mathfrak{k}') = \mathfrak{q} \), and \( \pi(\mathfrak{k}') \) is compact, whence \( \mathfrak{g}_0 + \mathfrak{k}' = \mathfrak{h} \).
If \( \mathfrak{s} \) is a maximal solvable subalgebra of \( \mathfrak{h} \) and \( \mathfrak{s} + \mathfrak{k}'' = \mathfrak{h} \) for some compact subalgebra \( \mathfrak{k}'' \) of \( \mathfrak{h} \), then \( \pi(\mathfrak{s}) \) is a maximal solvable subalgebra of \( \mathfrak{q} \) and \( \pi(\mathfrak{s}) + \pi(\mathfrak{k}'') = \mathfrak{q} \) for some compact subalgebra \( \pi(\mathfrak{k}'') \) of \( \mathfrak{q} \), whence \( \pi(\mathfrak{s}) \) is conjugate to \( \mathfrak{s}_0 \) under the adjoint group of \( \mathfrak{q} \) and hence \( \mathfrak{s} \) is conjugate to \( \mathfrak{g}_0 \).
\end{proof}

Suppose that \( H \) is a connected Lie group with centre \( Z(H) \).
The above result implies that there exist finitely many maximal connected solvable subgroups \( G_j \) of \( H \) such that every maximal solvable subgroup \( G \) is conjugate to exactly one of the \( G_j \).
Since the closure of a connected solvable group is connected and solvable, these maximal connected solvable subgroups are closed.
Exactly one of these subgroups, \( G_0 \) say, has the property that \( H/G_0Z(H) \) is compact.

While we are focussing on solvable Lie groups, we mention that for solvable groups, simply connected and contractible coincide.

We conclude this part with two results about compact semisimple Lie algebras that we will use later.

\begin{lemma}\label{lem:compact-semisimple-algebras}
Let \( \mathfrak{l} \) be a compact semisimple Lie algebra with a toral subalgebra \( \mathfrak{t} \) and a subalgebra \( \mathfrak{g} \) such that \( \mathfrak{l} = \mathfrak{t} + \mathfrak{g} \).
Then \( \mathfrak{l} = \mathfrak{g} \).
\end{lemma}

\begin{proof}
It suffices to show that \( \mathfrak{t} \subseteq \mathfrak{g} \).
By replacing \( \mathfrak{t} \) by a larger toral subalgebra if necessary, we may suppose without loss of generality that \( \mathfrak{t} \) is a maximal abelian subalgebra of \( \mathfrak{l} \).

It is well known (see, for instance, \cite[Section 4.3]{Varadarajan}) that the compact semisimple Lie algebra \( \mathfrak{l} \) with a maximal abelian subalgebra \( \mathfrak{t} \)  contains finitely many three-dimensional subalgebras \( \mathfrak{l}_\alpha \), each isomorphic to \( \mathfrak{su}(2) \), with a basis \( \{ X_\alpha, Y_\alpha, U_\alpha \} \), where \( U_\alpha \in \mathfrak{t} \), such that
\[
[X_\alpha, Y_\alpha] = U_\alpha, \qquad [Y_\alpha, U_\alpha] = X_\alpha \qquad\text{and}\qquad [U_\alpha, X_\alpha] = Y_\alpha.
\]
Further, the \( U_\alpha \) span \( \mathfrak{t} \), so it will suffice to show that each \( U_\alpha \in \mathfrak{g} \).

It is also well known (see, for instance, \cite[Section 4.3]{Varadarajan}) that \( \mathfrak{l} \) admits an inner product \( \left< \cdot, \cdot \right> \) (the negative of the Killing form), such that the orthogonal projection \( \pi_\alpha \) onto \( \Span\{ X_\alpha, Y_\alpha, U_\alpha\} \) has the property that
\[
\pi_\alpha V \in \R U_\alpha,
\qquad
[V, X_\alpha] = [\pi_\alpha V, X_\alpha]
\qquad\text{and}\qquad
[V, Y_\alpha] = [\pi_\alpha V, Y_\alpha]
\qquad\forall V \in \mathfrak{t}.
\]
Note that \( \left< U_\alpha, U_\alpha \right> \pi_\alpha V = \left< U_\alpha ,V \right> U_\alpha \) for all \( V \in \mathfrak{t} \).

Take a subalgebra \( \mathfrak{l}_\alpha \), and write \( X \), \( Y \), \( U \) and \( \pi \) instead of \( X_\alpha \), \( Y_\alpha \), \( U_\alpha \) and \( \pi_\alpha \) for ease of notation.
Since \( \mathfrak{l} = \mathfrak{g} + \mathfrak{t} \), there exist \( V, W \in \mathfrak{t} \) such that \( X - V \in \mathfrak{g} \) and \( Y - W \in \mathfrak{g} \).
Now
\[
\begin{aligned}
&   \left<U,U\right> [X-V, Y-W] \\
&\qquad= \left<U,U\right> ([X, Y]  - [V, Y] - [X, W] + [V, W]) \\
&\qquad= \left<U,U\right> ([X, Y]  - [\pi V, Y] - [X, \pi W] + [V, W]) \\
&\qquad= \left<U,U\right> U + \left<U, V\right> X + \left<U, W\right> Y \\
&\qquad= \left<U,U\right> U + \left<U, V\right> V  + \left<U, W\right> W + \left<U, V\right> (X-V) + \left<U, W\right> (Y-W) ,
\end{aligned}
\]
whence \( \left<U,U\right> U + \left<U, V\right> V + \left<U, W\right> W \in \mathfrak{g} \cap \mathfrak{t} \).
Now
\[
\left<U,U\right>[\left<U,U\right> U + \left<U, V\right> V + \left<U, W\right> W , X - V ]
= \left<U,U\right>^2 Y + \left<U, V\right>^2 Y + \left<U, W\right>^2 Y ,
\]
so \( Y \in \mathfrak{g} \), and
\[
\left<U,U\right>[\left<U,U\right> U + \left<U, V\right> V + \left<U, W\right> W , Y - W ]
= - \left<U,U\right>^2 X - \left<U, V\right>^2 X - \left<U, W\right>^2 X ,
\]
so \( X \in \mathfrak{g} \), whence \( U = [X,Y] \in \mathfrak{g} \) as required.
\end{proof}

\begin{lemma}\label{lem:Weyl-group}
Let \( \mathfrak{l} \) and \( \mathfrak{t} \) be the Lie algebras of a compact semisimple Lie group \( L \) and a maximal torus \( T \) thereof.
Then there exist \( w_1 \), \dots, \( w_J \) in \( L \) such that  \( \Ad(w_j) \mathfrak{t} = \mathfrak{t} \) and  \( \sum_{j=1}^{J}\Ad(w_j)U = 0 \) for all \( U \in \mathfrak{t} \).
\end{lemma}

We do not prove this lemma, but state only that the \( w_j \) are representatives of the finite group \( N(T)/T \), where \( N(T) \) is the normaliser of \( T \) in \( L \), that appear in the structure theory of compact Lie groups, and in particular, the Weyl group.
See, for instance, \cite[Sections 3.9 and 3.10]{Wallach} for much more on this.

\subsection{Polynomial growth and amenability}\label{ssec:polygr-amenability}

We now look at the structure of two particular types of Lie group in more detail.
If \( G \) is a connected Lie group, then it is said to be of polynomial growth if and only if the Haar measures of powers of sets \( U^n \) grow polynomially in \( n \).
This is equivalent to its Lie algebra \( \mathfrak{g} \) being of type (R), that is, the eigenvalues of \( \ad X \) are purely imaginary for each \( X\in \mathfrak{g} \).
For instance, nilpotent Lie groups and euclidean motion groups are of polynomial growth.
For more on this, see \cite{Guivarch-croissance, Jenkins-growth-main}.

\begin{lemma}\label{lem-poly-growth}
Let \( G \) be a connected Lie group with radical \( R \) and a Levi subgroup \( L \).
Then \( G \) is of polynomial growth if and only if \( R \) is of polynomial growth and \( L \) is compact.
If \( G \) is of polynomial growth, then maximal compact subgroups of \( G \) are tori if and only if \( G \) is solvable.
\end{lemma}

\begin{proof}
Both Guivarc'h \cite[p.~345]{Guivarch-croissance} and Jenkins \cite[p.~123]{Jenkins-growth-main} showed that Lie groups are of polynomial growth if and only if their radicals are of polynomial growth and their Levi subgroups are compact.
See also \cite[Theorem II.4.8]{Dungey-Elst-Robinson}.
Then a connected Lie group \( G \) of polynomial growth is a finite quotient of a group of the form \( R \rtimes L \), where \( R \) is solvable and \( L \) is a compact semisimple Lie group.

Let \( K \) be a maximal compact subgroup of \( G \); this is a connected compact Lie group.
Since \( K \) contains a subgroup that is locally isomorphic to \( L \), \( K \) is abelian if and only if \( G \) is solvable.
\end{proof}

Note that the universal covering group of \( \mathrm{SL}(2, \R) \) is contractible but not of polynomial growth.

\begin{definition}\label{def:amenable}
An almost connected Lie group \( G \) with Lie algebra \( \mathfrak{g} \) is said to be \introd{amenable} if each Levi subalgebra of \( \mathfrak{g} \) is compact.
\end{definition}

The standard definition of amenability of a group \( G \) involves the existence of a left-invariant mean on \( L^\infty(G) \).
The fact that for connected Lie groups this amounts to the definition above is well known (see, for instance, \cite[Corollary 4.1.9]{Zimmer}.
The extension to almost connected groups is straightforward.
It is also well known (and follows from the standard definition or from ours) that connected closed subgroups and quotients of amenable groups are amenable.

It is clear that connected Lie groups of polynomial growth are amenable, but examples such as the ``\( ax+b \)-group'', which is solvable but not of polynomial growth, show that the converse is false.

\begin{lemma}\label{lem:amenable-contractible}
Suppose that \( K \) is a maximal compact subgroup of a connected amenable Lie group \( H \), and that \( \bigcap_{h \in H} hKh^{-1} = \{e\} \).
Then there is a closed connected solvable normal subgroup \( G \) of \( H \) such that
\begin{enumerate}
  \item \( H = G \cdot K \), whence \( G \) acts simply transitively on \( H/K \);
  \item \( TG = G \) whenever \( T \) is an automorphism of \( H \) and \( TK = K \).
\end{enumerate}
\end{lemma}

\begin{proof}
Let \( N \) and \( R \) be the nilradical of and radical of \( H \); then \( N \subseteq R \).
Write \( H \) as \( RL \), where \( L \) is a necessarily compact Levi subgroup; in light of Lemma \ref{lem:Levi-and-max-cpct}, we may assume without loss of generality that \( L \subseteq K \).
The assumption on \( K \) implies that \( K \cap Z(H) = \{e\} \).
We write \( \mathfrak{n} \), \( \mathfrak{r} \) and so on for the Lie algebras of these groups; then \( \mathfrak{k} \cap Z(\mathfrak{h}) = \{0\} \).

We are going to use the \introd{Killing form}, a bilinear form on \( \mathfrak{h} \) defined by
\[
B(X,Y) \coloneqq  \trace( \ad(X) \ad(Y) )
\qquad\forall X, Y \in  \mathfrak{h}.
\]
This has many important properties, for which see, for instance, \cite[pp.~33--50]{Bourbaki1-3}; we will use the following:
\begin{enumerate}\renewcommand{\labelenumi}{\textrm{(\alph{enumi})}}
  \item if \( T_* \) is an automorphism of \( \mathfrak{h} \), then \( B(T_*X, T_*Y) = B(X,Y) \) for all \( X, Y \in \mathfrak{h} \);
  \item \( B( X, X ) < 0 \) for all \( X \in \mathfrak{k} \setminus \{0\} \) (because \( \mathfrak{k} \) is compact and \( \mathfrak{k} \cap Z(\mathfrak{h}) = \{0\} \));
  \item \( B( X, Y ) = 0 \) for all \( X \in \mathfrak{h} \) and all \( Y \in \mathfrak{n} \);
  \item \( B( [X, Y], Z) = 0 \) for all \( X \in \mathfrak{h} \) if and only if \( Z \in \mathfrak{r} \);
\end{enumerate}
We denote by \( \mathfrak{g} \) the subspace \( \{ X \in \mathfrak{h} : B(X,Y ) = 0 \ \forall Y \in \mathfrak{k} \} \).
Because \( \mathfrak{k} \) is semisimple, \( [\mathfrak{h}, \mathfrak{h}] \supseteq [\mathfrak{k}, \mathfrak{k}] = \mathfrak{k} \), and from (c) and (d) it follows that \( \mathfrak{n} \subseteq \mathfrak{g} \subseteq \mathfrak{r} \).
Then \( \mathfrak{g} \) is an ideal in \( \mathfrak{h} \), from Remark \ref{rem:derivations-radicals-1}, and \( \mathfrak{h} = \mathfrak{g} \oplus \mathfrak{k} \) from (b) and linear algebra.

Write \( K_R \) for \( K \cap R \), which is connected since it is a maximal compact subgroup of \( R \), and abelian, since it is both compact and solvable, and so is a torus; further, \( K_R \cap N = \{e\} \) by assumption and Lemma \ref{lem:central torus}.

Let \( T \) be an automorphism of \( H \) that fixes \( K \).
Then \( T_* \mathfrak{k} = \mathfrak{k} \), and from (a), \( T_* \mathfrak{g} = \mathfrak{g} \).

Since \( K \) is a maximal compact subgroup, \( H/K \) is simply connected, and Lemma \ref{lem:Lie-algebra-criterion} implies that the connected analytic subgroup \( G \) of \( R \) with Lie algebra \( \mathfrak{g} \) is closed, \( H = G \cdot K \), and \( G \) acts simply transitively on \( H / K \).
Further, if \( T \) is an automorphism of \( H \) and \( TK = K \), then its infinitesimal version \( T_* \) is an automorphism of \( \mathfrak{h} \) and \( T_*\mathfrak{k} = \mathfrak{k} \), whence \( T_* \mathfrak{g} = \mathfrak{g} \) and \( TG=G \).
\end{proof}

\begin{remark}\label{rem:amenable-implies-nice}
This lemma may be extended to more general connected Lie groups \( H \) at the cost of relaxing the requirement that \( G \) be normal.
However, the following example shows that no such result holds for all connected Lie groups.

Let \( H \) be the simply connected covering group of \( \mathrm{SU}(n,1) \), where \( n \geq 1 \), and \( K \) be a maximal compact subgroup.
Then \( H/K \) is contractible, but there is no solvable subgroup \( G \) of \( H \) that acts transitively on \( H/K \).
\end{remark}

\subsection{Proof of Theorem B}\label{ssec:proof-thm-main-2}
We are now ready to prove our next main theorem, which we restate, in a longer version.

\begin{theorem}\label{thm:main-2}
Let \( (M,d) \) be a homogeneous metric manifold.
Then there is a metric \( d' \) on \( M \) such that the identity mapping on \( M \) from \( (M,d) \) to \( (M, d') \) is a homeomorphic rough isometry, and {there is a transitive closed connected amenable subgroup \( H_\times \) of \( \Iso(M,d') \); hence \( M \) is homeomorphic to \( H_\times/K_\times \), where \( K_\times \) is a compact subgroup of \( H_\times \).}

If \( M \) is a metric Lie group, then we may take \( K_\times \) to be a finite group; if \( M \) is a simply connected metric Lie group, then we may take \( K_\times \) to be trivial.

If \( M \) is a contractible metric space, then we may take \( K_\times \) to be trivial and \( H_\times \) to be solvable, so that \( M \) is homeomorphically roughly isometric to a connected, simply connected solvable metric Lie group.
\end{theorem}

\begin{proof}
Let \( M \) be a homogeneous metric manifold, and suppose that \( H \) is a connected transitive isometry group of \( M \), so that we may identify \( M \) with \( H/K_o \), where \( K_o \) is the compact stabiliser of a point \( o \) in \( M \).
We may take \( H \) to be a Lie group that acts on \( M \) by analytic maps, by Theorem \ref{thm:Szenthe}.

We begin with a short outline of the proof.
Up to local isomorphism, the connected Lie group \( H \) is a semidirect product \( R \rtimes L \), where \( R \) is the solvable radical and \( L \) is a semisimple Levi subgroup.
Further, up to local isomorphism, \( L \) has an Iwasawa decomposition \( AN \cdot K \), where \( K \) is compact and \( AN \) is solvable.
Then \( H = S \cdot K \), where \( S \) is the closed solvable subgroup \( R \rtimes AN \) of \( H \).
If \( H \) has a left-invariant, right-\( K \)-invariant metric \( d \), then \( H \) is isometric to the group \( S \times K \), equipped with a left-invariant metric \( d_\times \), as described in Lemma \ref{lem:unwinding-product-groups}.
We need to deal with two additional complications: first, we need to deal with groups \( H \) that are not semidirect products, but quotients thereof, and second, we need to deal with the quotient \( H/K_o \).
Now we provide the details.

We recall from Lemma \ref{lem:Levi-decomp} that, in general, there is a continuous open projection \( \pi: R \rtimes L \to H \), with discrete kernel, \( \Gamma \) say, and \( \Gamma \) has a subgroup of finite index \( \Gamma_1 \) that is strongly central, that is, if \( (r,l) \in \Gamma_1 \), then both \( (r,e) \) and \( (e,l) \) are central in \( R \rtimes L \).
In particular, this implies that \( l \) lies in the subgroup \( K \) for any Iwasawa decomposition \( ANK \) of \( L \).

Now \( L \) has an Iwasawa decomposition (see \eqref{eq: Iwasawa-decomp}) \( ANK \), in which \( K = V \times K_c \), where \( V \) is a vector group which is compact modulo \( V \cap Z^+(L) \), and \( K_c \) is a maximal compact subgroup of \( L \), while \( AN \) is solvable; as above, we write \( S \) for the solvable group \( R \rtimes AN \), and then \( R \rtimes L = S \cdot (V \times K_c) \).
Let \( K_o \) be the stabiliser of a point \( o \) in \( M \) in \( H \), let \( K_m \) be a maximal compact subgroup of \( H \) that contains \( K_o \) and let \( K_R = K_m \cap R \).
Then \( K_o \subseteq (Z(L)^+)\afterbar K_R K_c \), by Lemmas \ref{lem:Levi-and-max-cpct} and \ref{lem:Iwasawa-and-Levi}.
The subgroup \( K_R K_c \times V \) of \( H \) is compact modulo the centre of \( H \), so that we may apply Corollary \ref{cor:left-G-right-K-invariant metric} to modify \( d \) and obtain a new admissible metric \( d' \) on \( M \), such that the identity map on \( M \) is a rough isometry from \( (M,d) \) to \( (M,d') \), and
\begin{equation}\label{eq:Pi}
d'(gg'kK_o, gg''kK_o) = d'(g'K_o, g''K_o)
\qquad\forall g,g',g'' \in H \quad\forall k \in K_R K_c V.
\end{equation}
For simplicity of notation, we replace \( d' \) by \( d \) and assume that \( d \) has the invariance property \eqref{eq:Pi}.

We define \( \omega: S \times (V \times K_c) \to S \cdot (V \times K_c) \) by
\[
\omega(s,k) \coloneqq  sk^{-1}
\qquad\forall s \in S \quad\forall k \in (V \times K_c).
\]
Then \( \omega \) is a homeomorphism.
We lift the metric \( d \) on the space \( H/K_o \), first to a pseudometric on \( H \) with kernel \( K_o \), and then to a pseudometric \( \dot d \) on the covering group \( S \cdot (V \times K_c) \) with kernel \( \pi^{-1} K_o \): more precisely, we define
\[
\dot d(x,y) \coloneqq  d(\pi x K_o, \pi y K_o)
\qquad\forall x,y \in S \cdot (V \times K_c).
\]
Now \( \dot d \) is continuous, admissible, left-invariant and right-\( \pi^{-1}(K_R K_c V) \)-invariant by construction.
By Lemma \ref{lem:unwinding-product-groups}, \( \dot d_\times \coloneqq  \dot d \circ (\omega \otimes \omega) \) is a continuous admissible left-invariant pseudometric on \( S \times (V \times K_c) \), whose kernel is a closed subgroup of \( S \times (V \times K_c) \), by Lemma \ref{lem:pseudometrics}.
When we identify points at \( \dot d_\times \)-distance \( 0 \), we obtain a \( S \times (V \times K_c) \)-invariant admissible metric \( d_\times \) on the quotient \( M_\times \).
Since the mapping \( \omega \) from \( S \times (V \times K_c) \) to \( S \cdot (V \times K_c) \) is an isometry of pseudometric spaces, the quotient metric space \( (M_\times, d_\times) \) is isometric to \( (M, d) \).

Trivially, the amenable Lie group \( S \times V \times K_c \) acts transitively and isometrically on \( (M_\times, d_\times) \), so there is a continuous homomorphism \( \alpha: S \times V \times K_c \to \Iso(M_\times, d_\times) \).
The image of \( \alpha \) is the product of the compact group \( \alpha(K_c) \) and the solvable group \( \alpha(S \times V) \), and so \( H_\times \), the closure of this image in \( \Iso(M_\times, d_\times) \), is the commuting product of the compact group \( \alpha(K_c) \) and the closed solvable group \( (\alpha(S \times V))\afterbar  \).
The intersection of these subgroups may be nontrivial, but \( H_\times \) is still amenable.
We may identify \( M_\times \) with the space \( H_\times/K_\times \), where \( K_\times \) is the compact stabiliser in \( H_\times \) of a point in \( M_\times \).

This proves the general part of the theorem.
However there are still some particular cases to consider.

First, if \( M \) is a metric group, then we may take \( K_o \) to be trivial, and trace through the argument above.
We see that \( M_\times \) is a finite quotient of an amenable metric group, and the order of the group that we factor out is bounded.
Indeed, in this case, \( M \) may be identified with the covering group \( R \rtimes L / \Gamma_0 \), and by Lemma \ref{lem:Levi-decomp}, \( \Gamma_0 \) has a subgroup \( \Gamma_1 \) of finite index such that \( \omega^{-1} \Gamma_1 \) is central and a fortiori normal in \( S \times (V \times K_c) \).
Then \( M_\times \) may be identified with \( (S \times V \times K_c) / \omega^{-1} \Gamma_0 \), which is a finite quotient of the amenable Lie group \( (S \times V \times K_c) / \omega^{-1} \Gamma_1 \).
If \( M \) is also simply connected, then no factoring out of discrete subgroups is involved, and we may identify \( M_\times \) with \( S \times V \times K_c \).

Another special case is when \( M \) is contractible.
In this case, \( M_\times \) is contractible, and so is of the form \( H_\times/K_\times \) where \( K_\times \) is a maximal compact subgroup of \( H_\times \), and \( H_\times \) is amenable.
Then there is a simply connected solvable group that acts simply transitively on \( M \) by Lemma \ref{lem:amenable-contractible}.
\end{proof}

\begin{remark}
We may choose the metric in Theorem \ref{thm:main-1}, in such a way that it is not necessary to change the metric at the beginning of this theorem.
Moreover, for any \( \epsilon \in \R^+ \), there is a homogeneous metric manifold of the form \( S \times K/K_0 \), where \( K_0 \) is a compact subgroup of \( S \times K \), that is \( (1, \epsilon) \)-quasi-isometric to the original space \( (M,d_0) \).
\end{remark}

Before we move on, we make a few observations.
It is evident that if we start with slightly different hypotheses, we may modify the argument of the proof above to prove slightly different results.
For example, if we start with a riemannian metric, we may work throughout with riemannian metrics and bi-Lipschitz mappings rather than general metrics and rough isometries.
Or if we start with a semisimple Lie group, we do not need to consider the Levi decomposition.
Or if we are allowed to choose the metrics, then we may do so to ensure that we have an isometry rather than a rough isometry.
By doing this, we easily obtain the following corollaries, which are really corollaries of the method of proof rather than of the result.

\begin{corollary}\label{cor:main-thm-2-Riemannian}
Let \( (M,d) \) be a homogeneous riemannian manifold.
Then there is a riemannian metric \( d' \) on \( M \) (so the identity mapping on \( M \) from \( (M,d) \) to \( (M, d') \) is bi-Lipschitz), such that \( (M,d') \) admits a transitive connected isometry group of the form \( S \times L \), where \( S \) is solvable and \( L \) is compact and semisimple; hence \( M \) is homeomorphic to  \( (S \times L)/K \), where \( K \) is a compact subgroup of \( S \times L \).

If \( M \) is a metric Lie group, then we may take \( K \) to be a finite group; if \( M \) is a simply connected metric Lie group, then we may take \( K \) to be trivial.

If \( M \) is a contractible metric space, then we may take \( L \) and \( K \) to be trivial, so that \( M \) is bi-Lipschitz to a connected, simply connected solvable metric Lie group.
\end{corollary}

\begin{corollary}\label{cor:contractible-Lie-groups}

Let \( G \) be a connected Lie group.
Then the following are equivalent:
\begin{enumerate}
  \item \( G \) may be made isometric to a connected simply connected solvable Lie group; and
  \item \( G = R \rtimes L \), where \( R \) is the solvable radical and \( L \) is a Levi subgroup of \( G \); further, \( R \) is simply connected  and \( L \) is a direct product of finitely many (possibly zero) copies of the universal covering group of \( \mathrm{SL}(2, \R) \).
\end{enumerate}
\end{corollary}

\begin{proof}

If (i) holds, then \( G \) is contractible; by Remark \ref{rem:contractible-Lie-groups}, (ii) holds.

On the other hand, if (ii) holds, then \( G \) may be made isomorphic to a solvable Lie group by Theorem \ref{thm:main-2}.
\end{proof}

\begin{corollary}\label{cor04111226}
Suppose that \( (G, d) \) is either a simply connected metric Lie group or a connected semisimple metric Lie group.
Then there exist a connected Lie group \( H \) that is the product of a solvable and a compact Lie group, and admissible left-invariant metrics \( d_G \) and \( d_H \) such that \( (G, d_G) \) and \( (H, d_H) \) are isometric and the identity map is a rough isometry from \( (G,d) \) to \( (G,d_G) \).
\end{corollary}

\begin{corollary}\label{cor:allsemisimple_are_NAxK}
Let \( G \) be a connected semisimple Lie group with Iwasawa decomposition \( ANK \).
Write \( K \) as \( V\times K_0 \), where \( V \) is a vector group and \( K_0 \) is compact.
Then \( G \) may be made isometric to the direct product \( AN\times V\times K_0 \).
\end{corollary}

It seems reasonable to ask whether a general connected metric Lie group \( (G,d) \) is homeomorphically quasi-isometric to an amenable connected metric Lie group.
Example \ref{ex:twist-not-semi-direct} provides a counterexample.
Indeed, with the notation of that example, we consider the group \( G = \tilde{G}/\Gamma \), and observe that the arguments used to prove Theorem \ref{thm:main-2} show that \( \tilde{G} \) is homeomorphic to an amenable direct product group \( \tilde{G}^* \), and that \( \tilde{G}/\Gamma \) is isometric to \( \tilde{G}^*/\Gamma^* \), where \( \Gamma^* \) is the group \( \{ (r,l^{-1}) : (r,l) \in \Gamma \} \).
However, unless \( n=2 \), the subgroup \( \Gamma^* \) is not normal, but has a subgroup of finite index that is normal.
In this case, \( G^* / \Gamma^* \) is not a group but is a finite quotient of a group.
Further, \( \tilde{G}/\Gamma_1 \) is a finite covering group of \( G \), and is isometric to the group \( \tilde{G}^* / \Gamma_1^* \).
More generally, we state without proof the following variant of Theorem \ref{thm:main-2}.

\begin{theorem}
Let \( (M,d) \) be a homogeneous metric manifold.
Then there is a metric \( d' \) on \( M \) such that the identity mapping on \( M \) from \( (M,d) \) to \( (M, d') \) is a homeomorphic rough isometry, and \( (M,d') \) has a finite cover that admits a simply transitive connected isometry group of the form \( S \times L \), where \( S \) is solvable and \( L \) is compact and semisimple; hence \( M \) is homeomorphic to  \( (S \times L)/K \).
\end{theorem}

\subsection{Notes and remarks}\label{ssec:notes;Lie-thry}

\subsubsection*{3.1.\enspace The main structure theorem}
We have mentioned some of the contributors to the solution of Hilbert's fifth problem, on the structure of locally euclidean topological groups.
It is worth pointing out that earlier the structure of compact groups was elaborated by von Neumann, and that of solvable groups by Chevalley.
For much more, see \cite{Montgomery-Zippin-TTG}.

Apropos of Corollary \ref{cor:iso-of-two-manifolds-is-Lie-Riemannian}, riemannian geometers have known for a long time that spaces \( H/K \), where \( K \) is a compact subgroup of a connected Lie group \( H \), may be equipped with a riemannian metric such that \( H \) acts by isometries, by choosing a \( K_* \)-invariant infinitesimal metric at the point \( K \) of \( H/K \) and then translating this to the whole space.
For instance, this fact is described as well known in a 1954 paper of Nomizu \cite{Nomizu}.

\subsubsection*{3.2.\enspace Compact subgroups}

It is well known that connected compact Lie groups contain maximal connected abelian subgroups, or maximal tori, all of which are conjugate (see, for instance, \cite[Corollary 4.35, p.~255]{Knapp-Lie-Groups-Beyond}).
It is perhaps not so well known that all connected compact groups contain maximal connected abelian subgroups, which are automatically closed, and all of these are conjugate.
See \cite[Theorem 9.32]{Hofmann-Morris} for more details.

We have stated Corollaries \ref{cor:isometry-group-is-Lie-group} and \ref{cor:iso-of-two-manifolds-is-Lie-Riemannian} for connected groups for simplicity, and Lemma \ref{lem:Iwasawa-max-cpct} for connected groups since Iwasawa did so.
For the almost connected case, see \cite[Theorem 32.5]{Stroppel} and the references cited there.

For more classical theory of the topology of Lie groups, see \cite{Samelson}.

\subsubsection*{3.3.\enspace Proof of Theorem \ref{thm:main-1}}
Let \( o \) be a point in a homogeneous metric space \( (M,d) \).
Then there is a connected locally compact group \( H \) that acts effectively and transitively on \( (M,d) \) by isometries, and \( M \) may be identified with the space \( H/K_o \), where \( K_o \) is the stabiliser of \( o \) in \( H \).
Let \( K \) be a maximal compact subgroup of \( H \) that contains \( K_o \), and suppose that \( d \) is right-\( K \)-invariant, which may always be arranged as in the proof of the theorem.

Then the collection of compact subgroups \( K_\nu \) of \( H \) such that \( K_o \subseteq K_\nu \subseteq K \) is a partially ordered set, and in the corresponding collection of quotient spaces \( H/K_\nu \), and by extending the construction following Definition \ref{def:metric-projection}, we may find a family of homogeneous metric space projections \( \pi_{\nu,\nu'}: H/K_\nu \to H/K_\nu' \) whenever \( K_\nu \subseteq K_\nu' \), and the implicit constants in all these projections are uniformly bounded.
This family of projections is an inverse system, and \( H/K_o \) is (trivially) the limit of spaces \( H/K_\nu \) as \( K_\nu \) shrinks down to \( K_o \).
If we restrict to the subgroups \( K_\nu \) such that \( H/K_\nu \) is a Lie group, then the limit is no longer trivial if \( H/K_o \) is not a manifold.

When the spaces \( H/K_\nu \) and \( H/K_\nu' \) are manifolds, then \( H/K_\nu \) is a fibre bundle over \( H/K_\nu' \).
However, in general, we cannot assert this: local triviality is a problem.

\subsubsection*{3.4.\enspace Lie groups}
Apropos of the exponential mapping on a Lie group, it may be of interest that in some cases, \( G = \exp(\mathfrak{g}) \), while  \( G = \exp(\mathfrak{g}) \exp(\mathfrak{g}) \) always (see \cite{Mosk-Sack-2003}).
When the exponential mapping is surjective, we can relate the eigenvalues of the adjoint action \( \ad \) to those of \( \Ad \), but in general matters are somewhat murky.

\subsubsection*{3.6.\enspace Decompositions of Lie groups}

Under suitable conditions, a connected locally compact group \( H \) has a connected simply connected locally compact universal covering group \( \tilde H \) (an infinite dimensional torus is a counter-example).
We refer the reader to \cite{BerestovskiiPlaut} for more information.
Thus it would be possible to extend the Levi decomposition to more general locally compact connected groups, but to discuss this would take us too far from our main goals.

We give two more examples that illustrate what may happen in the Levi decomposition when \( L \) is not closed.
Let \( U \) denote the universal covering group of \( \mathrm{SL}(2,\R) \),
and \( \{k_t: t \in \R\} \) be the one-parameter subgroup of \( U \) that projects down to the standard rotation subgroup of \( \mathrm{SL}(2,\R) \), parametrised so that \( k_t \) projects to the rotation through an angle \( t \); thus the elements \( k_{2\pi n} \), where \( n \in \Z \), project to the identity of \( \mathrm{SL}(2,\R) \).

\begin{example}
Let \( G \) be the group \( (U \times T) / Z \), where \( T = \{ z \in \C : |z| = 1\} \) and \( Z \) is the central discrete subgroup \( \{ (k_{2\pi n}, \expe^{in} ) : n \in \Z\} \) of \( U \times T \).
The Levi subgroup of \( G \) is an analytic subgroup, which may be identified with \( U \), and the radical is a torus, which may be identified with \( T \); these have an intersection which is dense in the radical.
This group cannot be written as a semidirect product of its radical and a Levi factor, and nor can any finite covering group or finite quotient, though a compact quotient of lower dimension is trivially a semidirect product of its radical and a Levi factor.
\end{example}

\begin{example}
Let \( G \) be the group \( (U \times U \times \R)/Z \) and \( Z \) be the central discrete subgroup \( \{ (k_{2\pi m}, k_{2\pi n}, m+\alpha n ) : m,n \in \Z\} \), where \( \alpha \) is irrational.
Then the Levi subgroup of \( G \) is an analytic subgroup, which may be identified with \( U \times U \), and the radical is a line; these have an intersection which is dense in the radical.
This group cannot be written as a semidirect product of its radical and a Levi subgroup, and nor can any finite covering group or compact quotient.
\end{example}

\subsubsection*{3.7.\enspace Polynomial growth and amenability}

A propos of Definition \ref{def:amenable}, the term was apparently coined by M.M. Day, to indicate the existence of a left-invariant mean on a group.
For us, amenable groups are amenable because they are much more tractable than general Lie groups.

\subsubsection*{3.8.\enspace Proof of Theorem B}
{
Theorem \ref{thm:main-2} shows that the class of solvable Lie groups is not closed under isometries.
It was already known (see \cite{Agrachev-Barilari, Milnor-note}) that the infinite covering group of \( \mathrm{SL}(2,\R) \) and the direct product of \( \R \) and the ``\( ax+b \)-group'' may be made isometric, even though the former group is not solvable and the latter is.
}

We remark that rough isometry is connected to Cornulier's  \cite{Cornulier_commability} notion of \emph{commability}; two homogeneous spaces are \emph{commable} if they may be connected by a finite number of projections from a group \( G \) onto a quotient \( G/K \), where \( K \) is a compact subgroup of \( G \), and cocompact embeddings; the arguments above show that \( G \) and \( G/K \) may be metrised (subject to some topological separability) in such a way that the projection and section are rough isometries.
But when we allow metrics that are not {proper quasigeodesic}, then rough isometry need not imply commability.
For instance, infinite covering projections may be rough isometries, by Lemma \ref{lem:covering-space}, but a space and its infinite cover are not commable.

Finally, it may be useful to recall that there is significant literature showing that the topology alone comes close to determining compact Lie groups; see \cite{Hubbuck-Kane} and the works cited there.
On the other hand, relations such as \( (L,C) \)-quasi-isometry do not ``see'' compact factors at all if \( C \) is sufficiently large.

\section{Solvable Lie groups}\label{sec:solvable}

In the last chapter, in Section \ref{ssec:main-1}, we showed that homogeneous metric spaces are roughly isometric to connected solvable Lie groups.
In this chapter, we restrict our attention to such groups.
We discuss the classification of connected solvable Lie groups up to isometry, due to Gordon and Wilson \cite{Gordon-Wilson-fine, Gordon-Wilson-solvmanifolds} in the riemannian case; we extend their results to cover more general metrics.
We consider when two such groups may be made isometric, and make some small contributions to the problem of their classification up to quasi-isometry, which has not yet been achieved and seems to be a very difficult question.
We present a different point of view to previous authors and extend some existing definitions and results.

Before we describe our results in more detail, we remind the reader of Definition \ref{def:various-products}: \( H = G \cdot K \) means that \( G \) and \( K \) are subgroups of \( H \) and the map \( (g,k) \mapsto gk \) is a homeomorphism from \( G \times K \) to \( H \).
Also, we write \( Z(H) \) and \( Z(\mathfrak{h}) \) for the centres of \( H \) and \( \mathfrak{h} \).

Suppose that \( (G,d) \) is a connected solvable metric Lie group, and that \( H \) is a connected closed subgroup of the Lie group \( \Iso(G,d) \) that contains \( G \), acting on itself by left translations.
Let \( K \) be the stabiliser in \( H \) of the point \( e \) in \( G \).
Then \( H = G \cdot K \), by Remark \ref{rem:metric-groups-give-products}, and \( Z(H) \cap K = \{e\} \), by Remark \ref{rem:no-compact-normal-subgroups}.
If moreover \( H \) is connected and solvable, then \( K \) is connected, compact and solvable, and hence a torus; in this case, we usually write \( T \) instead of \( K \).

Up to now, we have been looking at homogeneous metric spaces of the form \( H/K \), where \( H \) is a connected group and \( K \) is a compact subgroup.
For example, we showed in Corollary \ref{cor:isometric-groups-complementary-subgroups}  that if \( G_1 \) and \( G_2 \) are connected groups that both act simply transitively by isometries on a homogeneous metric space \( (M,d) \), and \( H \) is the connected component of the identity in \( \Iso(M,d) \) and \( K \) is the stabiliser in \( H \) of a point in \( M \), then it is possible to write \( H = G_1 \cdot K = G_2 \cdot K \).
However, this does not tell us whether \( G_1 \) and \( G_2 \) are algebraically similar.
In this chapter, we use the additional information available from Lie theory to discuss when two connected solvable Lie groups are isometric, or may be made isometric, or even when they are roughly isometric (and here there are still many open problems).
The first main step in doing this is to show that if \( G_1 \) and \( G_2 \) are isometric connected solvable metric Lie groups, then there is a connected, solvable metric Lie group \( H \) and a toral subgroup \( T \) such that \( H = G_1 \cdot T = G_2 \cdot T \).
We do this by appealing to a now classical theorem of Mostow \cite{Mostow-subgroups}.
Then we proceed to a detailed analysis of solvable Lie groups and their subgroups.

In Section \ref{ssec:derivations}, we examine derivations of Lie algebras, and particularly solvable Lie algebras, in detail; in Section \ref{ssec:nilpotent}, we recall the definitions of the upper and lower central series and properties of the exponential mapping from nilpotent Lie algebras to nilpotent Lie groups;
and in Section \ref{ssec:modifications}, we discuss the modifications of Gordon and Wilson \cite{Gordon-Wilson-fine, Gordon-Wilson-solvmanifolds}.
In Section \ref{ssec:twists}, we briefly describe ``twisted versions'' of solvable Lie groups, and show that two isometric connected solvable groups are both twisted versions of the same connected solvable group.
We connect twisted versions of groups to the normal modifications of Gordon and Wilson \cite{Gordon-Wilson-fine, Gordon-Wilson-solvmanifolds}, and to \emph{hulls} and \emph{real-shadows} of solvable groups in Section \ref{ssec:hull-shadow}.
In Section \ref{ssec:consequences-thm-3}, we prove Theorem C and a number of consequences.
Much of what we do, or at least something similar, is known; we leave a brief description of the history of this development to Section \ref{sec:history}.

We end this introductory discussion with a  lemma and a remark.

\begin{lemma}\label{lem:existence-of-G0-1}
Suppose that \( H \) is a connected solvable Lie group, and \( T \) is a toral subgroup of \( H \) such that \( Z(H) \cap T = \{ e \} \).
Then there exists a closed connected subgroup \( G_0 \) of \( H \) such that \( H = G_0 \rtimes T \).
\end{lemma}

\begin{proof}
Let \( N \) be the nilradical of \( H \).
Then \( T \cap N = \{e\} \) by Lemma \ref{lem:central torus}, whence \( TN/T \) is a torus isomorphic to \( T \) in the connected abelian Lie group \( H/N \).
Connected abelian Lie groups are all products of tori and vector groups, and their structure is easy to understand.
In particular, there is a closed subgroup \( C \) of \( H/N \) such that \( H/N = C \cdot (TN/N) \); let \( G_0 \) be the closed subgroup of \( H \) containing \( N \) such that \( G_0/N = C \).
Then \( G_0/N \) is normal in \( H/N \), so \( G_0 \) is normal in \( H \).
Since \( T \cap N = \{e\} \), as noted above, \( H = G_0 \cdot T \).
\end{proof}

\begin{remark}\label{rem:central-tori}
Let \( (M,d) \) be a metric solvmanifold, that is, a homogeneous metric manifold such that the Lie group \( \Iso(M,d) \) contains a closed solvable subgroup \( H \) that acts transitively on \( M \).
Let \( K \) be the stabiliser in \( H \) of a point in \( M \); then \( K \) is a compact subgroup of \( H \), and \( M \) is homeomorphic to \( H/K \).
The space \( H/K \) has a finite covering space \( H/T\), where \( T \) is the connected component of the identity in \( K \), which is a torus, being compact, connected and solvable.
Further, since \( H \) is an isometry group, \( Z(H) \cap T = \{e\} \).
Application of Lemma \ref{lem:existence-of-G0-1} to \( H \) and \( T \) produces a subgroup \( G_0 \) that acts simply transitively on \( H/T \), and ``almost simply transitively'' on \( H/K \).
\end{remark}

\subsection{Derivations and automorphisms}\label{ssec:derivations}
Here we prove some preliminary results about derivations and introduce a little more notation.

\begin{remark}\label{rem:derivations-roots}
Suppose that \( L \) is a diagonalisable linear map on a Lie algebra \( \mathfrak{g} \); then there is a direct sum eigenspace decomposition \( \mathfrak{g} = \sum_{\lambda}  \mathfrak{g}_\lambda \), where \( LX = \lambda X \) for all \( X \in \mathfrak{g}_\lambda \).
It is well known that \( L \) is a derivation if and only if \( [ \mathfrak{g}_\alpha, \mathfrak{g}_\beta] \subseteq  \mathfrak{g}_{\alpha+\beta} \) for all eigenvalues \( \alpha \) and \( \beta \).
Indeed, if \( X \in \mathfrak{g}_\alpha \) and \( Y \in \mathfrak{g}_\beta \), then
\begin{equation*}
\begin{aligned}
[LX,Y] + [X,LY] & = (\alpha+\beta) [X,Y],
\end{aligned}
\end{equation*}
so if \( L \) is a derivation, then \( [X,Y] \in \mathfrak{g}_{\alpha+\beta} \).
Conversely if \( [X,Y] \in \mathfrak{g}_{\alpha+\beta} \) for all \( X\in \mathfrak{g}_\alpha \) and \( Y\in \mathfrak{g}_\beta \) and all eigenvalues \( \alpha \) and \( \beta \), then the linearity of \( L \) implies that \( L[X,Y] = [LX,Y] + [X,LY] \) for all \( X,Y \in \mathfrak{g} \) and \( L \) is a derivation.
\end{remark}

\begin{remark}\label{rem:derivations-radicals}
If \( D \) is any derivation of a Lie algebra \( \mathfrak{g} \), then
\( D \rad(\mathfrak{g}) \subseteq \nil(\mathfrak{g}) \),
by \cite[Theorem 7, p.~74]{Jacobson}.
In particular, \( [\mathfrak{g} , \rad(\mathfrak{g})] \subseteq \nil(\mathfrak{g}) \).
This implies that if \( \mathfrak{v} \) is a subspace of \( \mathfrak{g} \) and \( \nil(\mathfrak{g}) \subseteq \mathfrak{v} \subseteq  \rad(\mathfrak{g}) \), then \( \mathfrak{v} \) is an ideal in \( \mathfrak{g} \).
\end{remark}

The next lemma is certainly known, but we are not aware of a proof in the literature, so we provide one.

\begin{lemma}\label{lem:shadow-step-1}
Suppose that \( \mathfrak{g} \) is a real Lie algebra, and that \( \mathfrak{d} \) is an abelian algebra of semisimple derivations of \( \mathfrak{g} \).
Then there are commuting abelian algebras \( \mathfrak{d}_{\mathrm{r}} \) and \( \mathfrak{d}_{\mathrm{i}} \) of semisimple derivations of \( \mathfrak{g} \) such that every element of \( \mathfrak{d}_{\mathrm{r}} \) has purely real eigenvalues, every element of \( \mathfrak{d}_{\mathrm{i}} \) has purely imaginary eigenvalues, and every element \( D \) of \( \mathfrak{d} \) may be written as a sum \( D = D_{\mathrm{r}} + D_{\mathrm{i}} \), where \( D_{\mathrm{r}} \in \mathfrak{d}_{\mathrm{r}} \) and \( D_{\mathrm{i}} \in \mathfrak{d}_{\mathrm{i}} \).
\end{lemma}

\begin{proof}
By considering the simultaneous eigenvalue decomposition of \( \mathfrak{g} \) under the action of \( \mathfrak{d} \), we may write the complexification \( \mathfrak{g}_{\C} \) as a ``sum of root spaces'' \( \sum_{\alpha \in \Phi} \mathfrak{g}_{\alpha} \), where \( \Phi \) is a finite set of linear mappings from \( \mathfrak{g}_{\C} \) to \( \C \) and \( \mathfrak{g}_\alpha \) is the subspace of all \( X \in \mathfrak{g}_{\C} \) such that  \( DX = \alpha(D) X \) for all \( D \in \mathfrak{d} \).
We write \( \mathfrak{g}_{\gamma} = \{0\} \) if \( \gamma \notin \Phi \).

Define the linear mapping \( \bar D  \) on \( \mathfrak{g}_{\C} \) by requiring that \( \bar D X \coloneqq  \bar\alpha(D)X \) for all \( X \in \mathfrak{g}_{\alpha} \) and all \( \alpha \in \Phi \).
Since \( (\alpha(D) + \beta(D))\bar{\phantom{x}} = \bar\alpha(D) + \bar\beta(D) \), Remark \ref{rem:derivations-roots} implies that \( \bar D \) is a derivation.
Further, \( D + \bar D \) has real eigenvalues while \( D - \bar D \) has purely imaginary eigenvalues.
It remains to show that \( \bar D \) restricts to a linear mapping of \( \mathfrak{g} \).

By linear algebra, \( \mathfrak{g} \) has a basis
\[
\{ X_j, Y_j, W_k : j \in \{1, \dots, J\}, k \in \{ 1, \dots, K \}\}
\]
such that the subspaces \( \Span \{X_j, Y_j \} \) and \( \Span\{W_k\} \) are irreducible and invariant for \( \mathfrak{d} \).
In the complexification \( \mathfrak{g}_{\C} \), each \( D \in \mathfrak{d} \)  is diagonalised in the basis
\[
\{ X_j + i Y_j, X_j - i Y_j, W_k : j \in \{1, \dots, J\}, k \in \{ 1, \dots, K \} \},
\]
with eigenvalues \( \lambda_j \) and \( \bar\lambda_j \) and \( \mu_k \), say; here the \( \lambda_j \) are strictly complex while the \( \mu_k \) are real.
By definition, \( \bar D (X_j + i Y_j) = \bar\lambda_j(X_j + i Y_j) \) and \( \bar D (X_j - i Y_j) = \lambda_j(X_j - i Y_j) \); it follows that
\[
\bar D X_j =   \Re\lambda_j X_j + \Im\lambda_j Y_j
\qquad\text{and}\qquad
\bar D Y_j = - \Im\lambda_j X_j + \Re\lambda_j Y_j .
\]
Since also \( \bar D W_k = \mu_k W_k \), it follows by \( \R \)-linearity that \( \bar D \) preserves \( \mathfrak{g} \), as required.
\end{proof}

\begin{corollary}[After {\cite[Corollary 2.6]{Golo-LeDonne}}] \label{cor:super-Jordan}
Suppose that \( \mathfrak{g} \) is a Lie algebra and \( D \) is a derivation of \( \mathfrak{g} \).
Then we may write \( D = D_{\mathrm{sr}} + D_{\mathrm{si}} + D_{\mathrm{n}} \), where each summand is a derivation of \( \mathfrak{g} \), each summand commutes with the other summands, and \( D_{\mathrm{sr}} \) is semisimple with real eigenvalues, \( D_{\mathrm{si}} \) is semisimple with purely imaginary eigenvalues, and \( D_{\mathrm{n}} \) is nilpotent.
Moreover, the ranges \( \Ran(D_{\mathrm{sr}}) \),  \( \Ran(D_{\mathrm{si}}) \) and \( \Ran( D_{\mathrm{n}}) \) are all subspaces of the range \( \Ran(D) \).
\end{corollary}

\begin{proof}
Bourbaki \cite[Proposition 4, page 6]{Bourbaki7-9} shows that we may write \( D \) as \( D_{\mathrm{s}} + D_{\mathrm{n}} \), the commuting sum of a semisimple and a nilpotent derivation.
Further, \( D_{\mathrm{s}} \) decomposes as the commuting sum
\( D_{\mathrm{sr}} + D_{\mathrm{si}} \) of derivations, where the summands have real and purely imaginary eigenvalues, by Lemma \ref{lem:shadow-step-1}.
It remains to show that \( D_{\mathrm{sr}} \) and \( D_{\mathrm{n}} \) commute, and to examine the ranges.

We choose a Jordan basis for \( \mathfrak{g} \) so that \( D \) is in real Jordan normal form; then in each block, the nilpotent part commutes with the real and imaginary parts of the diagonal part, and the ranges behave as claimed.
\end{proof}

\subsection{Nilpotent Lie groups and algebras}\label{ssec:nilpotent}
We recall some standard definitions and properties of nilpotent Lie algebras.

The \introd{upper central series} of a Lie algebra \( \mathfrak{g} \) is defined recursively:
\[
\mathfrak{g}^{[0]} \coloneqq  \{0\}
\qquad\text{and}\qquad
\mathfrak{g}^{[j+1]} / \mathfrak{g}^{[j]} \coloneqq  Z(\mathfrak{g}/\mathfrak{g}^{[j]}).
\]
Then
\[
\mathfrak{g}^{[0]} \subseteq \mathfrak{g}^{[1]} \subseteq \mathfrak{g}^{[2]} \subseteq \dots .
\]
The subspaces in this series increase strictly and then stabilise, that is, all later terms coincide.
The series reaches \( \mathfrak{g} \) if and only if \( \mathfrak{g} \) is nilpotent; in this case, the least positive integer \( \ell \) such that \( \mathfrak{g}^{[\ell]} = \mathfrak{g} \) is called the \emph{nilpotent length} of \( \mathfrak{g} \).
Each \( \mathfrak{g}^{[j]} \) is a characteristic ideal, that is,
\( D\mathfrak{g}^{[j]} \subseteq \mathfrak{g}^{[j]} \) for any derivation \( D \) of \( \mathfrak{g} \).
The upper central series of the complexification \( \mathfrak{g}_{\C} \) is the complexification of the upper series on \( \mathfrak{g} \), that is,
\( (\mathfrak{g}_{\C})^{[j]} = (\mathfrak{g}^{[j]})_{\C} \).

The \introd{lower central series} of a Lie algebra \( \mathfrak{g} \) is also defined recursively:
\begin{equation}\label{eq:def-lcc}
\mathfrak{g}_{[0]} \coloneqq  \mathfrak{g}
\qquad\text{and}\qquad
\mathfrak{g}_{[j+1]} \coloneqq  [\mathfrak{g}, \mathfrak{g}_{[j]}].
\end{equation}
Then
\[
\mathfrak{g}_{[0]} \supseteq \mathfrak{g}_{[1]} \supseteq \mathfrak{g}_{[2]} \supseteq \dots .
\]
The subspaces in this series decrease strictly and then stabilise.
The series reaches \( \{0\} \) if and only if \( \mathfrak{g} \) is nilpotent; in this case, the least positive integer \( j \) such that \( \mathfrak{g}_{[j]} = \{0\} \) coincides with the nilpotent length of \( \mathfrak{g} \).
Each \( \mathfrak{g}_{[j]} \) is a characteristic ideal, and the lower central series of the complexification \( \mathfrak{g}_{\C} \) is the complexification of the lower central series on \( \mathfrak{g} \).

If \( \mathfrak{n} \) is a nilpotent Lie algebra, then there is a connected simply connected (indeed, contractible) Lie group \( N \) such that the exponential mapping is a diffeomorphism from \( \mathfrak{n} \) onto \( N \), and \( \exp(\mathfrak{g}) \) is a closed subgroup of \( N \) for every subalgebra \( \mathfrak{g} \) of \( \mathfrak{n} \)  (see \cite[Section 3.6]{Varadarajan}).
But not all nilpotent Lie groups are simply connected.

In general, every nilpotent Lie group \( N \) contains a maximal torus \( T \) that is central (see Lemma \ref{lem:central torus}), and \( N/T \) is simply connected.
This implies that if \( \mathfrak{g} \) is a subalgebra of \( \mathfrak{n} \) that contains \( \mathfrak{t} \), the Lie algebra of the maximal torus \( T \), then \( \exp(\mathfrak{g}/\mathfrak{t}) \) is closed in \( N/T \) and it follows that \( \exp(\mathfrak{g}) \) is closed in \( N \).

\subsection{Modifications}\label{ssec:modifications}
Many nonlinear problems on Lie groups may be solved by turning them into linear problems on Lie algebras.
This is certainly the case for us.
Corollary \ref{cor:isometric-groups-complementary-subgroups} shows that we are interested in examples of connected groups \( H \) with closed connected subgroups \( G_0 \), \( G_1 \) and \( K \) such that \( K \) is compact and \( H = G_0 \rtimes K = G_1 \cdot K \).
This implies that the corresponding Lie algebras satisfy \( \mathfrak{h} = \mathfrak{g}_0 \oplus \mathfrak{k} = \mathfrak{g}_1 \oplus \mathfrak{k} \) and \( \mathfrak{g}_0 \) is an ideal in \( \mathfrak{h} \).
In this situation, for all \( X \in \mathfrak{g}_0 \), there exists a unique \( \sigma X  \in \mathfrak{k} \)
such that \( X + \sigma X  \in \mathfrak{g}_1 \).
Evidently, \( \sigma: \mathfrak{g}_0 \to \mathfrak{k} \) is linear and
\[
\mathfrak{g}_1 = \{ X + \sigma X  : X \in \mathfrak{g}_0 \}.
\]
The map \( \sigma \) and algebra \( \mathfrak{g}_1 \) are examples of a \emph{modification map} and a \introd{modification} in the terminology of Gordon and Wilson \cite{Gordon-Wilson-solvmanifolds}.
We shall be interested in modifications in the case where \( \mathfrak{k} \) is the Lie algebra of a torus (so we write \( \mathfrak{t} \)).

The following technical lemma follows from Gordon and Wilson \cite[Theorem 2.5]{Gordon-Wilson-solvmanifolds}.
We give a different proof.

\begin{lemma}\label{lem:modifications-are-normal}
Suppose that \( \mathfrak{h} \) is a Lie algebra of the form \( \mathfrak{n} \oplus \mathfrak{t} \), where \( \mathfrak{n} \) is a nilpotent ideal and \( \mathfrak{t} \) is a toral subalgebra such that \( \mathfrak{t} \cap Z(\mathfrak{h}) = \{0\} \).
Suppose also that \( \mathfrak{g} \) is a subalgebra of \( \mathfrak{h} \) such that \( \mathfrak{h} = \mathfrak{g} \oplus \mathfrak{t} \).
Then \( \mathfrak{g} \) is an ideal in \( \mathfrak{h} \).
\end{lemma}

\begin{proof}
We consider \( \mathfrak{g} \) as a modification of \( \mathfrak{n} \), that is we choose \( \sigma: \mathfrak{n} \to \mathfrak{t} \) such that
\[
\mathfrak{g} = \{ X + \sigma X  : X \in \mathfrak{n} \} .
\]
Since \( \mathfrak{t} \) is abelian and \( \mathfrak{g} \) is a subalgebra,
\begin{equation*}
[X, Y]  + [\sigma X  ,Y] - [\sigma Y ,X]
= [ X + \sigma X , Y + \sigma Y  ] \in \mathfrak{g}
\qquad\forall X, Y \in \mathfrak{n}.
\end{equation*}
All terms on the left-hand side lie in \( \mathfrak{n} = \Dom(\sigma) \), and \( \sigma(\mathfrak{g} \cap \mathfrak{n}) = \{0\} \), so
\begin{equation}\label{eq:key-eqn-1}
\sigma[X,Y] = \sigma[\sigma Y ,X] - \sigma[\sigma X ,Y]
\qquad\forall X, Y \in \mathfrak{n}.
\end{equation}

We are going to use induction on \( \dim(\mathfrak{h}) \).
If \( \dim(\mathfrak{h}) \) is \( 0 \), \( 1 \) or \( 2 \), then \( \mathfrak{g} \) is trivially an ideal.
We assume for the rest of the proof that \( \mathfrak{g}' \) is an ideal in \( \mathfrak{h}' \) whenever \( \mathfrak{h}' \), \( \mathfrak{n}' \), \( \mathfrak{t}' \) and \( \mathfrak{g}' \) satisfy the hypotheses of the lemma and \( \dim(\mathfrak{h}') < \dim(\mathfrak{h}) \).

By Weyl's unitarian trick, we may equip \( \mathfrak{n} \) with an inner product such that each of the family of linear maps \( \ad(\mathfrak{t}) \) is skew-symmetric.
Since \( \mathfrak{n} \) is \( \ad(\mathfrak{t}) \)-invariant, so is each member \( \mathfrak{n}_{[j]} \) of the lower central series (see Section \ref{ssec:nilpotent}), and there are (unique) \( \ad(\mathfrak{t}) \)-invariant subspaces \( \mathfrak{v}_{[j]} \) such that \( \mathfrak{n}_{[j-1]} = \mathfrak{v}_{[j]} \oplus \mathfrak{n}_{[j]} \).
It is not hard to show inductively that \( \mathfrak{n}_{[j]} = \ad(\mathfrak{v}_{[1]})^{j}  \mathfrak{v}_{[1]} + \mathfrak{n}_{[j+1]} \), and hence \( \mathfrak{v}_{[1]} \) generates \( \mathfrak{n} \).
Further, as \( \mathfrak{t} \) is abelian, we may decompose the spaces \( \mathfrak{v}_{[j]} \) into minimal \( \ad(\mathfrak{t}) \)-invariant subspaces, of dimension \( 1 \) or \( 2 \), which we label \( \mathfrak{w}_k \).

\subsubsection*{Step 1: a consequence of \eqref{eq:key-eqn-1}}
Suppose that \( \mathfrak{w}_j \) and \( \mathfrak{w}_k \) are minimal \( \ad(\mathfrak{t}) \)-invariant subspaces of \( \mathfrak{n} \), and \( \sigma[ \mathfrak{w}_j, \mathfrak{w}_k ] = \{0\} \).
We claim, and shall now prove, that
\begin{equation}\label{eq:claim-1}
 \sigma[\sigma Y ,X] = 0
\qquad\forall X \in \mathfrak{w}_j \quad\forall Y \in \mathfrak{w}_k,
\end{equation}
or equivalently, \( \sigma[\sigma X ,Y] = 0 \), since in light of our hypothesis,
\begin{equation}\label{eq:claim-1a}
\sigma[\sigma Y ,X] = \sigma[\sigma X ,Y]
\qquad\forall X \in \mathfrak{w}_j \quad\forall Y \in \mathfrak{w}_k.
\end{equation}
If \( \dim(\mathfrak{w}_k) = 1 \), then \eqref{eq:claim-1} holds, since
\( \ad(\sigma X) \) is skew-symmetric, so \( [\sigma X ,Y] =0 \); similarly \eqref{eq:claim-1} holds if \( \dim(\mathfrak{w}_j) = 1 \).
If both \( \mathfrak{w}_j \) and \( \mathfrak{w}_k \) are \( 2 \)-dimensional, then, as the dimension of the space of skew-symmetric maps of \( \R^2 \) is \( 1 \)-dimensional, we may take an orthonormal basis \( \{ X_0, X_1 \} \) of \( \mathfrak{w}_j \) such that \( \ad(\sigma X_0)|_{\mathfrak{w}_k} = 0 \).
This implies that
\[
\sigma[\sigma Y, X_0] = \sigma[\sigma X_0, Y] = 0
\qquad\forall Y \in \mathfrak{w}_k.
\]
Now there are two possibilities: either \( [\sigma Y, X_0] = 0 \) for all \( Y \in \mathfrak{w}_k \), or there exists \( Y \in \mathfrak{w}_k \) such that \( [\sigma Y, X_0] \neq 0 \).
In the former case, the skew-symmetry of \( \ad(\sigma Y)|_{\mathfrak{w}_j} \) implies that \( \ad(\sigma Y)|_{\mathfrak{w}_j} = \{0\} \), for all \( Y \in \mathfrak{w}_k \), and \eqref{eq:claim-1} holds.
In the latter case, there exists \( Y \in \mathfrak{w}_k \) such that \( [\sigma Y, X_0] = X_1 \) and hence \( \sigma X_1 = 0 \); coupled with the fact that \( \ad(\sigma X_0)|_{\mathfrak{w}_k} = 0 \),  this shows that \( \ad(\sigma X)|_{\mathfrak{w}_k} = 0 \) for all \( X \in \mathfrak{w}_j \) and \eqref{eq:claim-1} holds in this case too from \eqref{eq:claim-1a}.

{
\subsubsection*{Step 2: the case where \( \mathfrak{n} \) is abelian}
We recall that \( \mathfrak{t} \cap Z(\mathfrak{h}) = \{ 0\} \).

Since
\[
[\mathfrak{h}, \mathfrak{g}]
= [\mathfrak{t}, \mathfrak{g}] + [\mathfrak{g}, \mathfrak{g}]
\subseteq [\mathfrak{t}, \mathfrak{g}] + \mathfrak{g} ,
\]
\( \mathfrak{g} \) is an ideal if and only if \( [\mathfrak{t},\mathfrak{g}] \subseteq \mathfrak{g} \).

We consider the decomposition of \( \mathfrak{n} \) into \( \ad(\mathfrak{t}) \)-invariant subspaces \( \mathfrak{w}_j \), as described in the second paragraph of this proof.
Since \( \mathfrak{n} \) is abelian, \( [ \mathfrak{w}_j, \mathfrak{w}_k ] = \{0\} \) for all \( j \) and \( k \).
If \( \sigma X = 0 \) for all \( X \in \mathfrak{w}_j \) and for all \( j \), then \( \mathfrak{g} = \mathfrak{n} \) and we are done.
Otherwise, we fix a summand \( \mathfrak{w}_j \) and \( X \in \mathfrak{w}_j \) such that \( \sigma X \neq 0 \), and then our assumption that \( \mathfrak{t} \cap Z(\mathfrak{h}) = \{ 0\} \) implies that there exists \( k \) such that \( [\sigma X ,\mathfrak{w}_k] \neq \{0\} \).
Now \( \sigma[\sigma X,Y] = 0 \) for all \( Y \in \mathfrak{w}_k \) and since \( \ad(\sigma X)|_{\mathfrak{w}_k} \) is surjective, \( \sigma Y = 0 \) for all \( Y \in \mathfrak{w}_k \).
Then \( \mathfrak{w}_k \) is a nontrivial ideal in \( \mathfrak{h} \) that is contained in \( \mathfrak{n} \) and in \( \mathfrak{g} \).
We may now write
\[
  \mathfrak{h}'
= \mathfrak{n}' \oplus \mathfrak{t}'
= \mathfrak{g}' \oplus \mathfrak{t}',
\]
where
\[
\mathfrak{h}' = \mathfrak{h} / \mathfrak{w}_k, \qquad
\mathfrak{n}' = \mathfrak{n} /\mathfrak{w}_k, \qquad
\mathfrak{g}' = \mathfrak{g} /\mathfrak{w}_k, \quad\text{and}\quad
\mathfrak{t}' = (\mathfrak{t} + \mathfrak{w}_k)/\mathfrak{w}_k \simeq \mathfrak{t},
\]
and it is easy to show that \( \mathfrak{h}' \), \( \mathfrak{n}' \), \( \mathfrak{t}' \) and \( \mathfrak{g}' \) satisfy the hypotheses of the lemma and \( \dim(\mathfrak{h}') < \dim(\mathfrak{h}) \), and so \( \mathfrak{g}' \) is an ideal in \( \mathfrak{h}' \) by the inductive assumption and hence \( \mathfrak{g} \) is an ideal in \( \mathfrak{h} \).

For the rest of the proof, we may and shall assume that \( \mathfrak{n} \) is not abelian.
}

\subsubsection*{Step 3: the induction on dimension argument}
Suppose that \( \mathfrak{n}_0 \) is a (nontrivial) ideal in \( \mathfrak{h} \), that \( \mathfrak{n}_0\subseteq \mathfrak{n}_{[1]} \), and that \( \sigma \mathfrak{n}_0 = \{0\} \),
that is, \( \mathfrak{n}_0 \subseteq \mathfrak{n}_{[1]} \cap \mathfrak{g} \).
In this case, we may show that \( \mathfrak{g} \) is an ideal by induction on dimension.
Indeed, we may write
\[
  \mathfrak{h}'
= \mathfrak{n}' \oplus \mathfrak{t}'
= \mathfrak{g}' \oplus \mathfrak{t}',
\]
where
\[
\mathfrak{h}' = \mathfrak{h} / \mathfrak{n}_0, \qquad
\mathfrak{n}' = \mathfrak{n} /\mathfrak{n}_0, \qquad
\mathfrak{g}' = \mathfrak{g} /\mathfrak{n}_0, \quad\text{and}\quad
\mathfrak{t}' = (\mathfrak{t} + \mathfrak{n}_0)/\mathfrak{n}_0 \simeq \mathfrak{t}.
\]
By our inductive assumption, \( \mathfrak{g}' \) is an ideal in \( \mathfrak{h}' \), and hence \( \mathfrak{g} \) is an ideal in \( \mathfrak{h} \), as required.

\subsubsection*{Step 4: minimal \( \ad(\mathfrak{t}) \)-invariant subspaces}
Suppose that there exists a subspace \( \mathfrak{w}_j \) such that
\( \sigma(\mathfrak{w}_j) = \{0\} \).
Then for all \( X \in \mathfrak{w}_j \), all \( Y \in \mathfrak{n} \) and all \( U \in \mathfrak{t} \),
\[
\sigma [X,Y] = \sigma[ \sigma X, Y] + \sigma[ \sigma Y, X] = 0
\]
since \( \sigma X = 0 \) by hypothesis and \( [\sigma Y, X ] \in \mathfrak{w}_j \), and
\[
\sigma [U, [X,Y]] = \sigma[[U, X], Y] + \sigma[ X , [U, Y]] = 0
\]
similarly.
Define
\[
\mathfrak{n}_0
\coloneqq  \mathfrak{w}_j + [\mathfrak{h}, \mathfrak{w}_j] + [\mathfrak{h}, [\mathfrak{h}, \mathfrak{w}_j]] + \dots ;
\]
then \( \mathfrak{n}_0 \) is an ideal in \( \mathfrak{h} \) and \( \sigma \mathfrak{n}_0 = \{0\} \), that is, \( \mathfrak{n}_0 \subseteq \mathfrak{g} \).

There are now two possibilities: \( \mathfrak{n}_0 \not\subseteq \mathfrak{v}_{[1]} \) or \( \mathfrak{n}_0 \subseteq \mathfrak{v}_{[1]} \).
In the first case, define \( \mathfrak{n}_1 \coloneqq  \mathfrak{n}_0 \cap \mathfrak{n}_{[1]} \).
Then \( \mathfrak{n}_1  \) is also an ideal which may be factored out much as in Step 3 to show that \( \mathfrak{g} \) is an ideal by induction on dimension.
Otherwise, \( \mathfrak{n}_0 \) is central in \( \mathfrak{n} \) and an ideal in \( \mathfrak{h} \), and may be factored out so that induction on dimension again shows
that \( \mathfrak{g} \) is an ideal.

\subsubsection*{Step 5: Denouement}
Take \( \mathfrak{w}_j \subseteq \mathfrak{n}_{[\ell-1]} \); then \( \mathfrak{w}_j \) is an ideal in \( \mathfrak{h} \), where \( \ell \) is the nilpotent length of \( \mathfrak{n} \).
If \( \dim(\mathfrak{w}_j) = 2 \), then there exists \( X \in \mathfrak{w}_k \subseteq \mathfrak{n} \) such that \( \ad(\sigma X)|_{\mathfrak{w}_j} \neq 0 \).
Now \( \sigma[\sigma X, Y] = 0 \) for all \( Y \in \mathfrak{w}_j \) by \eqref{eq:claim-1} and hence \( \sigma\mathfrak{w}_j = \{0\} \).
We may factor out \( \mathfrak{w}_j \) and show that \( \mathfrak{g} \) is an ideal by induction on dimension, as in Step 3.
Otherwise, if \( \dim(\mathfrak{w}_j) = 1 \) and \( \sigma \mathfrak{w}_j = \{0\} \), then \( \mathfrak{w}_j \) is an ideal which we may factor out to apply the induction on dimension argument and show that \( \mathfrak{g} \) is an ideal.
Finally, if \( \dim(\mathfrak{w}_j) = 1 \) and \( \sigma \mathfrak{w}_j \neq \{0\} \), there exists \( \mathfrak{w}_k \subseteq \mathfrak{n} \) such that \( [ \sigma \mathfrak{w}_j, \mathfrak{w}_k] = \mathfrak{w}_k \),
and now \( \sigma \mathfrak{w}_k = \sigma[ \sigma \mathfrak{w}_j, \mathfrak{w}_k] = \{0\} \) by \eqref{eq:claim-1}, so again we may apply the result of Step 4 to conclude that \( \mathfrak{g} \) is an ideal.
\end{proof}

\subsection{Split-solvability and the real-radical}\label{ssec:split-solv}

Recall that a solvable Lie algebra \( \mathfrak{g} \) is said to be \introd{split-solvable} (or \emph{completely solvable}) if the eigenvalues of each \( \ad(X) \), where \( X \in \mathfrak{g} \), are real.
A connected Lie group \( G \) is said to be split-solvable if its Lie algebra is split-solvable.
If \( G \) is split-solvable and of polynomial growth, then the eigenvalues of each \( \ad(X) \) are also purely imaginary, and so they are all zero, that is, \( G \) is nilpotent.
Similarly, every toral subalgebra \( \mathfrak{t} \) of a split-solvable Lie algebra \( \mathfrak{g} \) is central; indeed, if \( U \in \mathfrak{t} \), then the eigenvalues of \( \ad(U) \) are \( 0 \) because they are simultaneously real and purely imaginary, and since \( \ad(U) \) is semisimple, then \( \ad(U) = 0 \).

\begin{lemma}\label{lem:split-solvable}
Suppose that \( G \) is a connected Lie group with Lie algebra \( \mathfrak{g} \).
The following conditions are equivalent:
\begin{enumerate}
  \item the eigenvalues of each \( \Ad(g) \), where \( g \in G \), are positive;
  \item the eigenvalues of each \( \ad(X) \), where \( X \in \mathfrak{g} \), are real;
  \item there are ideals \( \mathfrak{g}_j \) in \( \mathfrak{g} \) that form a complete flag, that is,
\[
\{0\} = \mathfrak{g}_0 \subset \mathfrak{g}_1 \subset \dots \subset \mathfrak{g}_{n-1} \subset \mathfrak{g}_{n} = \mathfrak{g}
\]
and \( \dim(\mathfrak{g}_j) = j \) for all \( j \); and
  \item there are closed connected normal subgroups \( G_j \) of \( G \) such that
\[
\{0\} = G_0 \subset G_1 \subset \dots \subset G_{n-1} \subset G_{n} = G
\]
and \( \dim(G_j) = j \) for all \( j \).
\end{enumerate}
Further, if the Lie algebra \( \mathfrak{g} \) satisfies condition (ii) and \( \mathfrak{f}_1 \subset \mathfrak{f}_2 \subset \dots \subset \mathfrak{f}_{l-1} \subset \mathfrak{f}_{l} \) is any partial flag of ideals of \( \mathfrak{g} \), then it is possible to choose the ideals \( \mathfrak{g}_j \) in condition (iii) such that all the \( \mathfrak{f}_k \) appear in the complete flag.
\end{lemma}

\begin{proof}
The exponential mapping is both surjective and injective on simply connected split-solvable Lie groups (see \cite{Dixmier}); it follows that the exponential mapping is surjective on all split-solvable Lie groups.
Since \( \exp(\ad(X)) = \Ad(\exp(X)) \) for all \( X \in \mathfrak{g} \), (i) and (ii) are equivalent.

See \cite[p.~45]{Knapp-Lie-Groups-Beyond} for the equivalence of (ii) and (iii).
By passing to Lie algebras, it is clear that (iv) implies (iii).

If (iii) holds and \( \mathfrak{f}' \) and \( \mathfrak{f}'' \) are ideals in \( \mathfrak{g} \) such that \( \mathfrak{f}' \subseteq \mathfrak{f}'' \), then \( \mathfrak{f}' + (\mathfrak{g}_j \cap \mathfrak{f}'') \) is also an ideal in \( \mathfrak{g} \), and as \( j \) increases by \( 1 \), the dimension of \( \mathfrak{f}' + (\mathfrak{g}_j \cap \mathfrak{f}'') \) increases by at most \( 1 \).
Thus by omitting repeated terms and relabelling, we obtain a flag of ideals that starts at \( \mathfrak{f}' \), increases in dimension by \( 1 \) at each step, and arrives at \( \mathfrak{f}'' \).
By applying this observation with \( \mathfrak{f}' \) and \( \mathfrak{f}'' \) taken to be adjacent
terms \( \mathfrak{f}_{j-1} \) and \( \mathfrak{f}_{j} \) of the partial flag of ideals of \( \mathfrak{g} \), and then concatenating the flags between \( \mathfrak{f}_{j-1} \) and \( \mathfrak{f}_{j} \) for different \( j \), we obtain a complete flag that passes
through all the ideals of the partial flag.

Let \( \mathfrak{t} \) be the Lie algebra of a maximal compact subgroup of \( G \), which is necessarily a torus since \( G \) is solvable, and central (as argued before the statement of this lemma).
Choose a basis \( \{Z_1, \dots, Z_m\} \) of \( \mathfrak{t} \) such that each \( \exp(\R Z_j) \) is a \( 1 \)-dimensional torus in \( G \), and let \( \mathfrak{f}_k \) be \( \Span\{Z_1, \dots, Z_k\} \) when \( 1 \leq k \leq m \).

Extend the partial flag of ideals \( \mathfrak{f}_1 \subset \dots \subset \mathfrak{f}_{m} \) to a complete flag of ideals \( \mathfrak{g}_j \), and let \( G_j \) be the normal analytic subgroup of \( G \) that corresponds to \( \mathfrak{g}_j \).
If \( j \leq m \), then \( G_j \) is closed by construction.
If \( j \geq m \), then \( \mathfrak{g}_j \supseteq \mathfrak{t} \), and \( \mathfrak{g}_j / \mathfrak{t} \) is a subalgebra of \( \mathfrak{g} / \mathfrak{t} \).
Since \( G/T \) is connected and simply connected, the analytic subgroup \( G_j/T \) of \( G/T \) corresponding to \( \mathfrak{g}_j/\mathfrak{t} \) is closed by \cite[end of Section II]{Chevalley}, whence \( G_j \) is closed in \( G \).
\end{proof}

\begin{proposition}\label{prop:split-solv}Let \( G \) be a split-solvable connected Lie group and \( H \) be a connected Lie group.
\begin{enumerate}
  \item If \( F \) is an analytic subgroup of \( G \), then \( F \) is split-solvable.
  \item If \( M \) is a closed normal subgroup of \( G \), then \( G/M \) is split-solvable.
  \item If \( T \) is a central torus in \( H \) then \( H/T \) is split-solvable if and only if \( H \) is split-solvable.
  \item If \( G \) is a dense analytic subgroup of \( H \), then \( H \) is split-solvable.
\end{enumerate}
\end{proposition}

\begin{proof}
We will work with the Lie algebras.

Suppose that there are ideals \( \mathfrak{g}_j \) in \( \mathfrak{g} \) such that
\[
\{0\} = \mathfrak{g}_0 \subset \mathfrak{g}_1 \subset \dots \subset \mathfrak{g}_{n-1} \subset \mathfrak{g}_{n} = \mathfrak{g} ,
\]
and \( \dim(\mathfrak{g}_j / \mathfrak{g}_{j-1}) =1 \) for all \( j \) (or equivalently, \( \dim(\mathfrak{g}_j) = j \)).
If \( \mathfrak{f} \) is a subalgebra of \( \mathfrak{g} \), then \( [ \mathfrak{f} , \mathfrak{g}_j  \cap \mathfrak{f}] \subseteq \mathfrak{g}_j \cap \mathfrak{f} \), so \( \mathfrak{g}_j  \cap \mathfrak{f} \) is an ideal in \( \mathfrak{f} \), and
\[
\{0\} =  \mathfrak{g}_0  \cap \mathfrak{f} \subseteq \mathfrak{g}_1  \cap \mathfrak{f} \subseteq \dots \subseteq \mathfrak{g}_{n-1}  \cap \mathfrak{f} \subseteq \mathfrak{g}_{n}  \cap \mathfrak{f} = \mathfrak{f},
\]
and \( \dim((\mathfrak{g}_j \cap \mathfrak{f})/( \mathfrak{g}_{j-1} \cap \mathfrak{f})) \leq 1 \) for all \( j \).
After eliminating all ideals \(  \mathfrak{f} \cap \mathfrak{g}_{j} \) that agree with the preceding ideal \( \mathfrak{f} \cap \mathfrak{g}_{j-1} \), we obtain a complete flag of ideals that shows that \( \mathfrak{f} \) is split-solvable.
Similarly, if \( \mathfrak{m} \) is an ideal in \( \mathfrak{g} \), then so is \( \mathfrak{g}_j + \mathfrak{m} \), and
\[
\{0\} = (\mathfrak{g}_0  + \mathfrak{m} ) / \mathfrak{m} \subseteq (\mathfrak{g}_1  + \mathfrak{m} ) / \mathfrak{m} \subseteq \dots \subseteq (\mathfrak{g}_{n-1}  + \mathfrak{m} ) / \mathfrak{m} \subseteq (\mathfrak{g}_{m}  + \mathfrak{m} ) / \mathfrak{m} = \mathfrak{g} / \mathfrak{m};
\]
again, after elimination of redundant terms, we produce a flag of ideals that show that \( \mathfrak{g}/ \mathfrak{m} \) is split-solvable.

For part (iii), note that \( H \) is solvable if and only if \( H/T \) is solvable.
If \( H \) is split-solvable, then so is \( H/T \) by part (ii).
Conversely, suppose that \( H/T \) is split solvable.
If \( X \in \mathfrak{h} \) and \( \lambda \) is an eigenvalue of \( \ad(X) \) acting on \( \mathfrak{h} \), then there exists \( Y \) in the complexification \( \mathfrak{h}_{\C} \) of \( \mathfrak{h} \) such that \( \ad(X) Y =  \lambda Y \).
For these \( X \), \( \lambda \) and \( Y \),
\[
\ad(X + \mathfrak{t})(Y + \mathfrak{t}) \subseteq \lambda (Y + \mathfrak{t}) ,
\]
and \( \lambda \) is an eigenvalue of the adjoint action of \( \ad(X + \mathfrak{t}) \) acting on \( \mathfrak{h}/\mathfrak{t} \), whence \( \lambda \in \R \)
Hence \( H \) is split-solvable, as required.

To prove (iv), suppose that \( G \) is a split-solvable analytic subgroup of the connected Lie group \( H \).
We need to show that \( H \) is split-solvable.
Now if \( H \) is also simply connected, then \( G \) is closed in \( H \) by \cite[end of Section II]{Chevalley}, and so \( G = H \) and \( H \) is split-solvable.

Let \( T \) be a maximal torus in \( H \).
We propose to show that \( T \) is central in \( H \); it then follows that \( GT/T \) is a dense split-solvable analytic subgroup of \( H/T \), and \( H/T \) is simply connected, whence \( GT = H \) and \( H/T \) is split-solvable, and so \( H \) is split-solvable.

Let \( \mathfrak{m} \) be the nilradical of \( \mathfrak{g} \) and \( M \) be the corresponding subgroup of \( G \), hence of \( H \).
Let \( N \) be the closure of \( M \) in \( H \), and \( \mathfrak{n} \) be its Lie algebra.
Since \( M \) is nilpotent, so is \( N \), and since \( xMx^{-1} = M \) for all \( x \in G \),  we see first that \( xNx^{-1} = N \) for all \( x \in G \) by closing \( M \), and then that \( xNx^{-1} = N \) for all \( x \in H \) since \( G \) is dense in \( H \).
If follows that \( N \) is normal in \( H \) and \( \mathfrak{n} \) is a nilpotent ideal in \( \mathfrak{h} \).

For \( x \in H \), we denote by \( \Ad_{\mathfrak{m}}(x) \) the restriction of \( \Ad(x) \) to \( \mathfrak{m} \), and by \( p_{\mathfrak{m}}(x, \cdot)  \) its characteristic polynomial:
\[
p_{\mathfrak{m}}(x, \lambda) \coloneqq  \det\bigl( (\Ad(x) - \lambda I)|_{\mathfrak{m}} \bigr)
\qquad\forall \lambda \in \C.
\]
If \( x \in G \), then \( \Ad_{\mathfrak{m}}(x) \) has positive real eigenvalues so all roots of \( p_{\mathfrak{m}}(x, \cdot) \) are positive real numbers.
By density, if \( x \in H \), then \( p_{\mathfrak{m}}(x, \cdot) \) is a limit of polynomials with positive real roots, so the roots of \( p_{\mathfrak{m}}(x, \cdot) \) are nonnegative real numbers.
However, \( \Ad_{\mathfrak{m}}(x) \) is invertible, and so all the roots of \( p_{\mathfrak{m}}(x, \cdot) \) are positive.

If \( x \in T \), then all the roots of \( p_{\mathfrak{m}}(x, \cdot) \) are of modulus \( 1 \), by Weyl's unitarian trick, and so all the roots of \( p_{\mathfrak{m}}(x, \cdot) \) are \( 1 \).
Since \( \Ad_{\mathfrak{m}}(x) \) is semisimple, it is the identity mapping.
It follows that conjugation by \( x \) fixes every point of \( M \); again by density, conjugation by \( x \) fixes every point of \( N \).
Passing to the Lie algebras, if \( U \in \mathfrak{t} \), then \( \ad(U) | _{\mathfrak{n}} = 0 \).

Now the commutator subgroup \( [G,G] \) of \( G \) is contained in \( M \) (since \( [\mathfrak{g}, \mathfrak{g}] \subseteq \mathfrak{m} \)), and by density, the commutator subgroup \( [H,H] \) of \( H \) is contained in \( N \).
Hence \( [\mathfrak{h}, \mathfrak{h}] \subseteq \mathfrak{n} \), and so if \( Y \in \mathfrak{h}_\C \) is an eigenvector for the adjoint action of \( \ad(U) \) (where \( U \in \mathfrak{t} \)) on \( \mathfrak{h}_{\C} \) corresponding to a nonzero eigenvalue, then \( Y \in \mathfrak{n}_{\C} \).
It follows that all eigenvalues of \( \ad(U) \) are \( 0 \), and as \( \ad(U) \) is semisimple, \( \ad(U) = 0 \) and \( U \) is central.
We conclude that \( \mathfrak{t} \) is central, and we are done.
\end{proof}

The next two results are largely due to Jablonski \cite{Jablonski}.

\begin{lemma}\label{lem:Jablonski-sum-lemma}
Let \( X \) and \( Y \) be elements of a solvable Lie algebra \( \mathfrak{h} \).
Then the spectrum of \( \ad(X+Y) \) is contained in the sum of the spectra of \( \ad(X) \) and \( \ad(Y) \).
\end{lemma}

\begin{proof}
This proof is due to Jablonski \cite{Jablonski}, and included for the convenience of the reader.

We may pass to the complexification \( \mathfrak{h}_{\C} \) without changing the eigenvalues.
Now by Lie's theorem, \( \ad(\mathfrak{h}) \) may be represented as an algebra of upper triangular matrices, and the eigenvalues of \( \ad(X) \) correspond to the diagonal entries of the associated matrix.
It follows that the eigenvalues of \( \ad(X+Y) \) are sums of eigenvalues of \( \ad(X) \) and \( \ad(Y) \).
\end{proof}

\begin{theorem}[{After Jablonski \cite{Jablonski}}]\label{thm:rradical}
Let \( G \) be a Lie group with Lie algebra \( \mathfrak{g} \).
Then \( \mathfrak{g} \) contains a unique maximal split-solvable ideal \( \mathfrak{s} \), which is characteristic in the sense that \( \phi(\mathfrak{s}) = \mathfrak{s} \) for any isomorphism \( \phi: \mathfrak{g} \to \mathfrak{g} \).
The analytic subgroup \( S \) of \( G \) corresponding to \( \mathfrak{s} \) is closed, connected, and normal in \( G \).
If \( G \) is connected and simply connected, then so is \( S \).
If \( T \) is a torus contained in \( S \), then \( T \) is normal in \( G \) and central in the connected component of the identity in \( G \).
\end{theorem}

\begin{proof}
Jablonski \cite[Proposition 2.1]{Jablonski} showed that \( \mathfrak{g} \) contains a unique maximal split-solvable ideal \( \mathfrak{s} \);
again we summarise the argument for the reader's convenience.
Let \( \mathfrak{r} \) and \( \mathfrak{n} \) denote the radical and nilradical of \( \mathfrak{g} \).

If \( \mathfrak{s}' \) and \( \mathfrak{s}'' \) are split-solvable ideals of \( \mathfrak{g} \), then both are contained in \( \mathfrak{r} \), and further, \( \mathfrak{s}' + \mathfrak{s}'' \) is an ideal.
If we can show that \( \mathfrak{s}'+\mathfrak{s}'' \) is split-solvable, then we may define \( \mathfrak{s} \) to be the sum of all split-solvable ideals of \( \mathfrak{g} \).

Take \( X' \in \mathfrak{s}' \) and \( X'' \in \mathfrak{s}'' \).
Since \( \mathfrak{s}' \) and \( \mathfrak{s}'' \) are ideals and \( [\mathfrak{g},\mathfrak{r}] \subseteq \mathfrak{n} \), all the nonzero eigenvalues of \( \ad(X') \) and \( \ad(X'') \) acting on \( \mathfrak{r} \) are realised on eigenvectors that lie in \( \mathfrak{s}'\cap \mathfrak{n} \) and \( \mathfrak{s}''\cap \mathfrak{n} \) respectively.
Hence \( \ad(X') \) and \( \ad(X'') \) have real eigenvalues when acting on \( \mathfrak{r} \), so \( \ad(X'+X'') \) has real eigenvalues when acting on \( \mathfrak{r} \) by Lemma \ref{lem:Jablonski-sum-lemma}, and \emph{a fortiori} when acting on \( \mathfrak{s}'+\mathfrak{s}'' \).

If \( \phi: \mathfrak{g} \to \mathfrak{g} \) is a Lie algebra isomorphism, then \( \phi(\mathfrak{s}) \) is a solvable ideal in \( \mathfrak{g} \).
Further, if \( [X, Y] = \lambda Y \), then \( [\phi(X), \phi(Y)] = \lambda \phi(Y) \), so that the eigenvalues of \( \ad(\phi (X)) \) coincide with those of \( \ad(X) \), whence \( \phi(\mathfrak{s}) \) is split-solvable.

We take \( S \) to be the Lie subgroup of \( G \) with Lie algebra \( \mathfrak{s} \).
Then \( S \) is normal in \( G \), even if \( G \) is not connected, since \( \mathfrak{s} \) is a characteristic ideal (in the sense above).
The closure \( \bar S \) is a connected normal solvable subgroup of \( G \), contained in the radical \( R \) of \( G \), and is still split-solvable by Proposition \ref{prop:split-solv}.
If \( \bar S \) were larger than \( S \), its Lie algebra would be a larger split-solvable ideal than \( \mathfrak{s} \), which is absurd.
Hence \( S \) is closed.

By \cite[end of Section II]{Chevalley}, all analytic subgroups of a connected simply connected solvable Lie group \( G \) are closed and simply connected.
In particular,  \( S \) is simply connected if \( G \) is simply connected.

Finally, if \( T \) is a torus contained in \( S \), with Lie algebra \( \mathfrak{t} \), and \( U \in \mathfrak{t} \), then \( \ad(U) \), acting on \( \mathfrak{s} \), has eigenvalues that are real as \( U \in \mathfrak{s} \) and purely imaginary as \( T \) is a torus; hence all eigenvalues of \( \ad(U) \) acting on \( \mathfrak{s} \) are \( 0 \).
Moreover \( [U,X] \in \mathfrak{n} \subseteq \mathfrak{s} \) for all \( X \in \mathfrak{g} \) because \( [\mathfrak{r}, \mathfrak{g}] \subseteq \mathfrak{n} \), and so all eigenvalues of \( \ad(U) \) acting on \( \mathfrak{g} \) are \( 0 \).
Since \( \ad(U) \) is also semisimple because \( T \) is a torus, then \( \ad(U) = 0 \) on \( \mathfrak{g} \).
Hence \( \mathfrak{t} \subseteq \mathfrak{n} \) and \( T \subseteq N \).

Lemma \ref{lem:central torus} now implies that \( T \) is normal in \( G \) and central in the connected component of the identity in \( G \).
\end{proof}

We call the Lie algebra \( \mathfrak{s} \) and the group \( S \) of the theorem above the \introd{real-radical} of \( \mathfrak{g} \) and \( G \), and denote them by \( \rrad(\mathfrak{g}) \) and \( \rrad(G) \).
The real-radical coincides with the nilradical in the special case where \( G \) is of polynomial growth.

The role of the real-radical is highlighted by the following simple result.

\begin{lemma}\label{lem:torus-and-real-radical}
Suppose that \( H \) is a connected solvable Lie group with real-radical \( S \), and \( T \) is a torus in \( H \).
Then \( S \cap T \subseteq Z(H) \) and \( \mathfrak{s} \cap \mathfrak{t} \subseteq Z(\mathfrak{h}) \).
\end{lemma}

\begin{proof}
If \( x \in S \cap T \), then every eigenvalue of \( \Ad(x) \) is of modulus \( 1 \) since \( x \in T \) by Weyl's unitarian trick, and is a positive real number since \( x \in S \).
Hence all eigenvalues are \( 1 \); since \( \Ad(x) \) is semisimple because \( x \in T \), \( \Ad(x) \) is the identity operator, whence \( x \in Z(H) \), and so \( S \cap T \subseteq Z(H) \).
\emph{A fortiori} \( \mathfrak{s} \cap \mathfrak{t} \subseteq Z(\mathfrak{h}) \).
\end{proof}

Split-solvable Lie subalgebras of solvable Lie algebras and the corresponding split-solvable subgroups of solvable Lie groups have nice properties.

\begin{theorem}[{After Jablonski \cite{Jablonski}}] \label{thm:split-solvable-is-good-algebra}
Suppose that \( \mathfrak{g} \) is a split-solvable subalgebra of a solvable Lie algebra \( \mathfrak{h} \) and \( \mathfrak{t} \) is a toral subalgebra of \( \mathfrak{h} \) such that \( \mathfrak{h} = \mathfrak{g}\oplus \mathfrak{t} \) and \( Z(\mathfrak{h}) \cap \mathfrak{t} = \{0\} \).
Then \( \mathfrak{g} \) is the real-radical of \( \mathfrak{h} \).
If \( \mathfrak{g}_1 \) is a subalgebra of \( \mathfrak{h} \) such that \( \mathfrak{h} = \mathfrak{g}_1 \oplus \mathfrak{t} \),  then \( \mathfrak{g}_1 \) is also an ideal in \( \mathfrak{h} \).
\end{theorem}

\begin{proof}
We write \( \mathfrak{n} \) and \( \mathfrak{s} \) for the nilradical and real-radical of \( \mathfrak{h} \).

First we are going to show that \( \mathfrak{n} \subseteq \mathfrak{g} \).
This implies immediately that \( \mathfrak{g} \) is an ideal by Remark \ref{rem:derivations-radicals}.
Then, since \( \mathfrak{g} \) is split-solvable by hypothesis, \( \mathfrak{g} \) is contained in \( \mathfrak{s} \).
The hypotheses and Lemma \ref{lem:torus-and-real-radical} imply that
\[
\dim(\mathfrak{h}) - \dim(\mathfrak{t})
= \dim(\mathfrak{g}) \leq \dim(\mathfrak{s})
\leq \dim(\mathfrak{h}) - \dim(\mathfrak{t}),
\]
so \( \mathfrak{g} = \mathfrak{s} \).

Since \( \mathfrak{h} = \mathfrak{g} \oplus \mathfrak{t} \supseteq \mathfrak{n} \oplus \mathfrak{t} \), there is a unique linear mapping \( \sigma : \mathfrak{n} \to  \mathfrak{t} \) defined by the condition that
\[
X + \sigma X  \in \mathfrak{g}
\qquad\forall X \in \mathfrak{n}.
\]
Define
\[
\tilde{\mathfrak{t}} \coloneqq  \sigma(\mathfrak{n}) ,
\qquad
\tilde{\mathfrak{n}} \coloneqq  \mathfrak{n} ,
\qquad
\tilde{\mathfrak{h}} \coloneqq  \mathfrak{n} \oplus \tilde{\mathfrak{t}},
\quad\text{and}\quad
\tilde{\mathfrak{g}} \coloneqq  \{ X + \sigma X : X \in \mathfrak{n} \}.
\]
Since \( \mathfrak{t} \) is abelian, \( \tilde{\mathfrak{t}} \) is a subalgebra of \( \mathfrak{t} \); by Remark \ref{rem:derivations-radicals}, \( \tilde{\mathfrak{h}} \) is an ideal in \( \mathfrak{h} \); and by linear algebra, \( \tilde{\mathfrak{g}} = \tilde{\mathfrak{h}} \cap \mathfrak{g} \); hence \( \tilde{\mathfrak{g}} \) is a subalgebra of \( \tilde{\mathfrak{h}} \) and  \( \tilde{\mathfrak{h}} = \tilde{\mathfrak{g}} \oplus \tilde{\mathfrak{t}} \).

Clearly, \( \tilde{\mathfrak{n}} \) is a nilpotent ideal in \( \tilde{\mathfrak{h}} \); if it were not the nilradical \( \nil(\tilde{\mathfrak{h}}) \) of \( \tilde{\mathfrak{h}} \), then it would be a subalgebra thereof, and there would be some nonzero element \( U \) of \( \tilde{\mathfrak{t}} \) in \( \nil(\tilde{\mathfrak{h}}) \).
Consider \( \ad(U) \) acting on \( \tilde{\mathfrak{h}} \); this is semisimple since \( \tilde{\mathfrak{t}} \) is toral, and nilpotent since \( U \in \nil(\tilde{\mathfrak{h}}) \), and hence \( \ad(U) \) annihilates \( \tilde{\mathfrak{h}} \).
Since \( U \in \tilde{\mathfrak{t}} \subseteq \mathfrak{t} \), \( \ad(U) \) annihilates \( \mathfrak{t} \) as \( \mathfrak{t} \) is abelian and \( \mathfrak{a} \) by the definition of \( \mathfrak{a} \).
Hence \( \ad(U) \) annihilates \( \mathfrak{h} \), that is, \( U \in Z(\mathfrak{h}) \cap \mathfrak{t} \). We conclude that \( U = 0 \) and hence \( \tilde{\mathfrak{n}} \) is also the nilradical of \( \tilde{\mathfrak{h}} \).

We fix \( X \in \tilde{\mathfrak{n}} \) and consider \( \ad(X+\sigma X ) \), acting on the complexified algebra \( (\tilde{\mathfrak{g}})_\C \); suppose that
\[
[ X + \sigma X , Y + \sigma Y ] = \lambda (Y + \sigma Y ),
\]
where \( Y \in \tilde{\mathfrak{n}}_{\C} \setminus \{0\} \) and \( \lambda \in \C \setminus \{0\} \).
On the one hand, since \( \mathfrak{g} \) is split-solvable, \( \lambda \) is real.
On the other hand, \( Y + \sigma Y  \in \tilde{\mathfrak{n}}_{\C} \) since \( \lambda \neq 0 \), and so \( \sigma Y  = 0 \).
Now \( Y \in (\tilde{\mathfrak{n}}_{\C})_{[j]} \setminus (\tilde{\mathfrak{n}}_{\C})_{[j+1]} \) (see \eqref{eq:def-lcc} for the definition of the lower central series) for some \( j \),
whence
\[
[\sigma X , Y  ] + (\tilde{\mathfrak{n}}_{\C})_{[j+1]}
= [X + \sigma X , Y  ] + (\tilde{\mathfrak{n}}_{\C})_{[j+1]}
= \lambda Y + (\tilde{\mathfrak{n}}_{\C})_{[j+1]},
\]
that is, \( \lambda \) is an eigenvalue of \( \ad(\sigma X ) \) acting on the quotient space \( \tilde{\mathfrak{n}}_{\C} /(\tilde{\mathfrak{n}}_{\C})_{[j+1]} \).
Since \( \ad(\sigma X ) \) has purely imaginary eigenvalues, \( \lambda \) is purely imaginary.

These conclusions are almost contradictory, and imply that all eigenvalues of \( \ad(X + \sigma X ) \), acting on \( (\tilde{\mathfrak{g}})_\C \), are \( 0 \), and \( \tilde{\mathfrak{g}} \) is nilpotent.

By Lemma \ref{lem:modifications-are-normal}, \( \tilde{\mathfrak{g}} \) is an ideal in \( \tilde{\mathfrak{h}} \); then \( \tilde{\mathfrak{g}}  \subseteq \tilde{\mathfrak{n}} \) as \( \tilde{\mathfrak{n}} \) is the largest nilpotent ideal in \( \tilde{\mathfrak{h}} \); for dimensional reasons, \( \tilde{\mathfrak{n}} = \tilde{\mathfrak{g}} \).
This completes the proof that \( \mathfrak{n} \subseteq \mathfrak{g} \).

Now suppose that \( \mathfrak{h} \) is a solvable Lie algebra with subalgebras \( \mathfrak{g} \), \( \mathfrak{g}_1 \), and \( \mathfrak{t} \) such that \( \mathfrak{g} \) is a split-solvable ideal, \( \mathfrak{t} \) is toral, and \( \mathfrak{h} = \mathfrak{g} \oplus \mathfrak{t} = \mathfrak{g}_1 \oplus \mathfrak{t} \); we shall prove that \( \mathfrak{g}_1 \)  is an ideal.

By the first part of this theorem, \( \mathfrak{n} \subseteq \mathfrak{g} \) and \( \mathfrak{g} \) is an ideal; now by Weyl's unitarian trick, we may write \( \mathfrak{g} = \mathfrak{n} \oplus \mathfrak{a} \), where \( [\mathfrak{t},\mathfrak{a}] = \{0\} \).
Much as before, there is a unique linear mapping \( \sigma: \mathfrak{g} \to \mathfrak{t} \) such that
\[
\mathfrak{g}_1 = \{ X + \sigma X : X \in \mathfrak{g} \}.
\]
As \( \mathfrak{g}_1 \) is a subalgebra of \( \mathfrak{h} \), \( \mathfrak{g}_1 \) is an ideal if and only if \( [\mathfrak{t}, \mathfrak{g}_1] \subseteq \mathfrak{g}_1 \).
Now \( [\mathfrak{t}, \sigma \mathfrak{g}] = \{0\} \), and \( \mathfrak{g} = \mathfrak{n} \oplus \mathfrak{a} \), where \( [\mathfrak{t},\mathfrak{a}] = \{0\} \), so
\begin{equation}\label{eq:same-commutators}
\begin{aligned}
[\mathfrak{t}, \mathfrak{g}_1]
&= \Span\{ [U, X + \sigma X ] : U \in \mathfrak{t}, X \in \mathfrak{g}\} \\
&= \Span\{ [U, X ] : U \in \mathfrak{t}, X \in \mathfrak{g}\}
= [\mathfrak{t}, \mathfrak{g}]
= [\mathfrak{t}, \mathfrak{n}],
\end{aligned}
\end{equation}
and hence \( \mathfrak{g}_1 \) is an ideal if and only if \( [\mathfrak{t}, \mathfrak{n}] \subseteq \mathfrak{g}_1 \).

Much as before, we
define
\[
\tilde{\mathfrak{t}} \coloneqq  \mathfrak{t},
\qquad
\tilde{\mathfrak{n}} \coloneqq  \mathfrak{n},
\qquad
\tilde{\mathfrak{h}} \coloneqq  \mathfrak{n} \oplus \mathfrak{t},
\quad\text{and}\quad
\tilde{\mathfrak{g}} \coloneqq  \{ X + \sigma X : X \in \mathfrak{n} \}.
\]
Clearly \( \tilde{\mathfrak{h}} \) is a subalgebra of \( \mathfrak{h} \) and \( \tilde{\mathfrak{h}} = \tilde{\mathfrak{g}} \oplus \tilde{\mathfrak{t}} \).
Further, \( [\mathfrak{t},\mathfrak{a}]=\{0\} \) and \( \mathfrak{h} = \mathfrak{a} \oplus \tilde{\mathfrak{h}} \), whence \( Z(\tilde{\mathfrak{h}}) \cap \tilde{\mathfrak{t}} = Z(\mathfrak{h}) \cap \mathfrak{t} = \{0\} \); moreover, \( [\tilde{\mathfrak{t}}, \tilde{\mathfrak{g}}] = [\tilde{\mathfrak{t}}, \tilde{\mathfrak{n}}] \) by the argument used to prove that \( [\mathfrak{t}, \mathfrak{g}_1] = [\mathfrak{t}, \mathfrak{g}] \) in \eqref{eq:same-commutators}.
From Lemma \ref{lem:modifications-are-normal}, \( \tilde{\mathfrak{g}} \) is an ideal in \( \tilde{\mathfrak{h}} \), and
\[
[\mathfrak{t}, \mathfrak{n}]
= [\tilde{\mathfrak{t}}, \tilde{\mathfrak{n}}]
= [\tilde{\mathfrak{t}}, \tilde{\mathfrak{g}}]
\subseteq \tilde{\mathfrak{g}}
\subseteq \mathfrak{g}_1 ,
\]
and hence \( \mathfrak{g}_1 \) is an ideal, as required.
\end{proof}

\begin{corollary} \label{cor:split-solvable-is-good-group}
Suppose that \( G \) is a split-solvable subgroup of a connected solvable Lie group \( H \) and \( T \) is a toral subgroup of \( H \) such that \( H = G \cdot T \) and \( Z(H) \cap T = \{e\} \).
Then \( G \) is normal in \( H \) and hence is the real-radical of \( H \).
If \( G_1 \) is a subgroup of \( H \) such that \( H = G_1 \cdot T \), then \( G_1 \) is also normal in \( H \).
\end{corollary}

\begin{proof}
We reduce this proof to the previous result by considering the Lie algebras of the various groups and subgroups.
The fact that the Lie algebra \( \mathfrak{g} \) of \( G \) is an ideal and coincides with \( \mathfrak{s} \) establishes that \( G \) is normal and the real-radical of \( H \).

Next, if \( G_1 \) satisfies the hypotheses of the theorem, then \( \mathfrak{h} \) is a solvable Lie algebra with subalgebras \( \mathfrak{g}_0 \), \( \tilde{\mathfrak{g}} \), and \( \mathfrak{t} \) such that \( \mathfrak{g}_0 \) is a split-solvable ideal, \( \mathfrak{t} \) is toral, \( \mathfrak{h} = \mathfrak{g}_0 + \mathfrak{g}_1 \),
and \( \mathfrak{h} = \mathfrak{g}_0 \oplus \mathfrak{t} = \mathfrak{g}_1 \oplus \mathfrak{t} \).
By the preceding theorem \( \mathfrak{g}_1 \)  is an ideal, and hence \( G_1 \) is normal in \( H \).
\end{proof}

We conclude with a result that shows that much of what we have done with solvable subgroups of the isometry group of a solvable metric Lie group may be extended to amenable subgroups of the isometry group.

\begin{theorem}
Let \( (G,d) \) be a split-solvable metric Lie group.
Suppose that  \( G \subseteq H \subseteq \Iso(G,d) \), and \( H \) is an amenable group.
Then \( G \) is the real-radical of \( H \), and \( H \) is of the form \( G \rtimes K \), where \( K \) is the stabiliser in \( H \) of the identity \( e \) of \( G \).
\end{theorem}

\begin{proof}
First, \( H = G \cdot K \) by Remark \ref{rem:more-on-G-dot-K}.
Next, since \( G \) is connected and the real-radical of \( H \) coincides with the real-radical of the connected component of the identity in \( H \), there is no loss of generality in supposing that \( H \) is connected, and we assume this throughout this proof.
We use Lie algebra; let \( \mathfrak{r} \), \( \mathfrak{s} \) and \( \mathfrak{n} \) be the radical, the real-radical and the nilradical of \( \mathfrak{h} \).
It will suffice to show that \( \mathfrak{g} = \mathfrak{s} \).

Let \( \pi \) be the canonical projection of \( \mathfrak{h} \) onto the Levi factor \( \dot{\mathfrak{h}} \coloneqq  \mathfrak{h}/\mathfrak{r} \), which is a compact Lie algebra since \( H \) is amenable.
Since \( \mathfrak{g} \) is solvable, so is \( \dot{\mathfrak{g}} \coloneqq  \pi(\mathfrak{g}) \), and since \( \dot{\mathfrak{h}} \) is compact, \( \dot{\mathfrak{g}} \) lies in a maximal torus of \( \dot{\mathfrak{h}} \).
Similarly \( \dot{\mathfrak{k}} \coloneqq  \pi(\mathfrak{k}) \) is a compact subalgebra of \( \dot{\mathfrak{h}} \).
Evidently, \( \dot{\mathfrak{h}} = \dot{\mathfrak{g}} + \dot{\mathfrak{k}} \), and so  \( \dot{\mathfrak{k}} = \dot{\mathfrak{h}} \) by Lemma \ref{lem:compact-semisimple-algebras}.

Thus we may take a compact Levi subgroup \( L \) of \( H \) such that \( \mathfrak{l} \subseteq \mathfrak{k} \), and a maximal torus \( \mathfrak{t} \) of \( \mathfrak{l} \) such that \( \mathfrak{g} \subseteq \mathfrak{r} \oplus \mathfrak{t} \).
Write \( \mathfrak{k}_R \coloneqq  \mathfrak{k} \cap \mathfrak{r} \); then by Lemma \ref{lem:Levi-and-max-cpct}, we see that \( \mathfrak{k} = \mathfrak{k}_R \oplus \mathfrak{l} \) and \( \mathfrak{k}_R \) and \( \mathfrak{l} \) commute; further, \( \mathfrak{k}_R \oplus \mathfrak{t} \) is a maximal torus of \( \mathfrak{k} \).

The first step is to show that \( \mathfrak{s} \subseteq \mathfrak{g} \).
To do this, observe that \( \mathfrak{s} \subseteq \mathfrak{r} \), and so \( \mathfrak{g} + \mathfrak{s} \) is a subalgebra of the solvable subalgebra \( \mathfrak{r} \oplus \mathfrak{t} \) of \( \mathfrak{h} \).
As \( \mathfrak{h} = \mathfrak{g} \oplus \mathfrak{k} \), there is a subalgebra \( \mathfrak{k}_0 \) of \( \mathfrak{k} \), necessarily a torus, such that \( \mathfrak{g} + \mathfrak{s} = \mathfrak{g} \oplus \mathfrak{k}_0 \).
By Theorem \ref{thm:split-solvable-is-good-algebra}, \( \mathfrak{g} \) is an ideal in \( \mathfrak{g} \oplus \mathfrak{k}_0 \).
Now both \( \mathfrak{g} \) and \( \mathfrak{s} \) are split-solvable ideals in \( \mathfrak{r} \oplus \mathfrak{t} \), and so by Theorem \ref{thm:rradical}, \( \mathfrak{g} + \mathfrak{s} \) is also split-solvable.
For \( X \in \mathfrak{g} + \mathfrak{s} \), the eigenvalues of \( \ad(X) \), acting on \( \mathfrak{g} + \mathfrak{s} \), are all real; \emph{a fortiori}, the eigenvalues of \( \ad(X) \), acting on \( \mathfrak{n} \), are all real.
If \( \mathfrak{g} \neq \mathfrak{g} + \mathfrak{s} \), there would be a nonzero element \( Z \) of \( \mathfrak{k} \cap (\mathfrak{g} + \mathfrak{s}) \), and the eigenvalues of \( \ad(Z) \), acting on \( \mathfrak{n} \), would be both real and purely imaginary, and hence \( 0 \).
Since \( \ad(Z) \mathfrak{r} \subseteq \mathfrak{n} \), the eigenvalues of \( \ad(Z) \), acting on \( \mathfrak{r} \), are also \( 0 \).
As \( Z \in \mathfrak{k} \), \( \ad(Z) \) is semisimple, and so \( \ad(Z) \mathfrak{r} = \{0\} \).
Moreover, since \( Z \in \mathfrak{r} \oplus \mathfrak{t} \) and \( \mathfrak{k}_R \) commutes with \( \mathfrak{l} \), it follows that \( \ad(Z) \), acting on \( \mathfrak{r} \oplus \mathfrak{t} \), is trivial.
This implies that \( \exp(tZ) \) commutes with \( G \) for all \( t \in \R \), which is absurd, since \( K \) is a group of nontrivial isometries of \( G \), by Remark
\ref{rem:no-compact-normal-subgroups}.
The impossibility of this shows that \( \mathfrak{s} \subseteq \mathfrak{g} \).

The second step is to show that \( \mathfrak{s} = \mathfrak{g} \).
To do this, we use Weyl's unitarian trick to decompose \( \mathfrak{r} \) into two \( \ad(\mathfrak{k}) \)-invariant subspaces, that is, write \( \mathfrak{r} = \mathfrak{s} \oplus \mathfrak{a} \), and hence
\[
\mathfrak{h} = \mathfrak{s} \oplus \mathfrak{a} \oplus \mathfrak{l} ,
\]
where each subspace is \( \ad(\mathfrak{k}) \)-invariant.
If \( X \in \mathfrak{g} \cap (\mathfrak{a} \oplus \mathfrak{l}) \), then \( X \in \mathfrak{g} \cap (\mathfrak{a} \oplus \mathfrak{t}) \), and we show that \( X = 0 \) as follows: we write \( X = U + V \), where \( U \in \mathfrak{a} \) and \( V \in \mathfrak{t} \), and take \( k_1 \), \dots, \( k_J \) in \( L \) as described in Lemma \ref{lem:Weyl-group}.
Note that
\[
\Ad(k_j)(U+V) = U + \Ad(k_j)V \in \mathfrak{r} \oplus \mathfrak{t},
\]
and if \( [U + V, W] = \lambda W \) for some \( \lambda \in \R \), then
\[
[ \Ad(k_j)(U + V) , \Ad(k_j)W]  = \lambda \Ad(k_j)W,
\]
and it follows from Lemmas \ref{lem:Weyl-group} and \ref{lem:Jablonski-sum-lemma} that
\( J\ad(U) = \ad\left( \sum_{j} \Ad(k_j)(U + V) \right) \), acting in the solvable algebra \( \mathfrak{r} \oplus \mathfrak{t} \), has real eigenvalues.
Since \( U \in \mathfrak{r} \), \( [\mathfrak{r}, U] \subseteq \mathfrak{n} \), and by construction, \( [\mathfrak{k}, U] = \{0\} \), whence \( \mathfrak{s} \oplus \R U \) is a split-solvable ideal in \( \mathfrak{h} \).
By the maximality of \( \mathfrak{s} \), \( U = 0 \).

At this point, \( V \in \mathfrak{l} \cap \mathfrak{g} \), and so all the eigenvalues of \( \ad(V) \), acting on \( \mathfrak{n} \), are both purely real and pure imaginary.
Consequently, much as argued for \( Z \) above, \( V = 0 \).
\end{proof}

\begin{corollary}[{After Wolf \cite{Wolf-Growth}}]\label{cor:After-Wolf}
Suppose that \( (G,d) \) is a nilpotent metric Lie group.
Then \( G \) is normal in \( \Iso(G,d) \) and is the nilradical of \( \Iso(G,d) \); further, \( \Iso(G,d) \) is of the form \( G \rtimes K \), where \( K \) is the stabiliser in \( \Iso(G,d) \) of the identity \( e \) of \( G \).
\end{corollary}

\begin{proof}
If \( G \) is nilpotent, then \( G \) is of polynomial growth, and so is \( \Iso(G,d) \).
It follows that \( \Iso(G,d) \) is amenable, and the previous result applies.
\end{proof}

\subsection{Twisted versions of groups and isometry of solvable groups}\label{ssec:twists}
We begin by recalling some results from Chapter \ref{sec:prel} and an observation that arises from the work of Alekseevski\u\i\ \cite{Alekseevskii}.

If a connected Lie group \( G_0 \) acts simply transitively and isometrically on a metric manifold \( (M,d) \) and \( H \) is an isometry group of \( (M,d) \) containing \( G_0 \), then we may write \( H = G_0 \cdot K \), where \( K \) is the stabiliser in \( H \) of a base point in \( M \), and the condition \( \bigcap_{h \in H} (hKh^{-1}) = \{e_H\} \) holds.
We suppose that \( G_0 \) is normal in \( H \).
If \( G_1 \) is also contained in \( H \) and acts simply transitively and isometrically on \( M \), then \( H = G_1 \cdot K \).
Hence there is a continuous bijection \( \Tau: G_1 \to G_0 \) and a continuous map \( \Phi: G_1 \to K \) such that \( g = \Tau(g) \Phi(g) \) for all \( g \in G_1 \).
We check that
\[
\Tau(gh) \Phi(gh) = gh = \Tau(g) \Phi(g) \Tau(h) \Phi(h) = \Tau(g)  \Tau(h)^{\Phi(g)} \Phi(g) \Phi(h)
\]
for all \( g, h \in G_1 \), where \( \Tau(h)^{\Phi(g)} \coloneqq  \Phi(g)\Tau(h)\Phi(g)^{-1} \); thus \( \Phi \) is a continuous homomorphism and \( \Tau(g) = g \Phi(g)^{-1} \), so both maps are smooth; further,
\[
\Tau(gh) = \Tau(g)  \Tau(h)^{\Phi(g)},
\]
and \( \Tau \) is a twisted homomorphism or cocycle.
Further, \( G_0 = \{ g \Phi(g)^{-1} : g \in G_1\} \), as \( \Tau \) is a bijection.
We summarise this discussion in the following definition and lemma.

\begin{definition}\label{def:twist}
We write \( G_1 \) is a \introd{twisted version} of \( G_0 \) to mean that there exists a Lie group \( H \), containing \( G_0 \) and \( G_1 \) as closed subgroups, with \( G_0 \) normal, and a Lie group homomorphism \( \Phi: G_1 \to K \), where \( K \) is a compact subgroup of \( H \), such that \( H = G_1 \cdot K \) and \( G_0 = \{ g \Phi(g)^{-1} : g \in G_1 \} \).
In this case, we also say that \( \Phi \) is the \introd{twisting homomorphism}.
\end{definition}

\begin{example}
Let \( H \) denote the Lie group \( (\R^2 \rtimes \mathrm{SO}(2) ) \times \R \), and define closed subgroups \( G_0 \coloneqq  (\R^2 \rtimes \{0\} ) \times \R \) and \( G_1 \coloneqq  \{ (x,y,[\alpha],\alpha) : x,y,\alpha \in \R \} \), where \(  [\alpha] \) denotes the equivalence class of \( \alpha \) in \( \mathrm{SO}(2) \), which we may identify with \( \R/ 2\pi\Z \).
Now both \( G_0 \) and \( G_1 \) are normal subgroups of \( H \).
Define the subgroup \( K \) to be \( \{(0,0)\} \times \mathrm{SO}(2) \times \{0\} \) and the homomorphism \( \Phi \colon G_1 \to K \) by \( (x,y,[\alpha],\alpha) \mapsto [\alpha] \).
Then \( \{ g \Phi(g)^{-1} : g \in G_1 \} = G_0 \), and hence \( G_1  \) is a twisted version of \( G_0 \).
In this case, \( G_0 \) is also a twisted version of \( G_1 \), via the twisting homomorphism \( \Phi' : G_0 \to K \) given by \( (x,y,0,\alpha) \mapsto -[\alpha] \).
Thus the semi-direct product \( \R^2 \rtimes \R \), where \( \R \) acts on \( \R^2 \) by rotations (embedded as \( G_1 \)), and the direct product \( \R^3 \) (embedded as \( G_0 \)), are twisted versions of each other.
\end{example}

Note that if \( G_1 \) is connected and solvable, then the closure \( (\Phi(G_1))\bar{\phantom{x}} \) is connected, solvable and compact, and so is a torus; we often write \( T \) instead of \( K \) in this case.
This remark leads to our next result.

\begin{lemma}\label{lem:twist}
Let \( (G_0,d) \) be a solvable metric Lie group, \( H \) be the connected component of the identity in \( \Iso(G_0,d) \), \( K \) be the stabiliser in \( H \) of the point \( e \) in \( G \), and \( T \) be a maximal torus of \( K \).
Suppose that \( G_0 \) is normal in \( H \).

Then, for a connected solvable Lie group \( G \), the following are equivalent:
\begin{enumerate}
  \item \( G \) may be made isometric to \( (G_0,d) \);
  \item \( G \) may be embedded in \( G_0 \rtimes T \) in such a way that \( G \cdot T = G_0 \rtimes T \); and
  \item \( G \) is a twisted version of \( G_0 \) with a twisting homomorphism \( \Phi: G \to T \).
\end{enumerate}
\end{lemma}

\begin{proof}
We recall that maximal tori of compact Lie groups are conjugate; hence the group \( G_0 \rtimes T \) does not depend on the choice of \( T \), up to isomorphism.

Suppose that \( G \) may be made isometric to \( (G_0,d) \).
From Corollary \ref{cor:isometric-groups-complementary-subgroups}, there is an embedding of \( G \) into \( H \) such that \( H = G \cdot K = G_0 \cdot K \), and \( G_0 \cdot K = G_0 \rtimes K \) by assumption.
The closure of the image of \( G \) in the quotient \( (G_0 \rtimes K)/G_0 \), which is isomorphic to \( K \), is solvable, connected, and compact, hence a torus, and so contained in a maximal torus.
This implies that \( G \cdot T = G_0 \rtimes T \).

Conversely, if we may embed \( G \) into \( G_0 \rtimes T \) in such a way that \( G_0 \rtimes T = G \cdot T \), then we may embed \( G \) into \( G_0 \rtimes K \), and it may be checked that \( G_0 \rtimes K = G \cdot K \); again from Corollary \ref{cor:isometric-groups-complementary-subgroups}, \( G \) may be made isometric to \( (G_0,d_0) \).

The equivalence of (ii) and (iii) follows from the discussion preceding Definition \ref{def:twist}.
\end{proof}

In our situation, where we have solvable subgroups \( G_1 \) and \( G_2 \) of an isometry group \( H \) that we want to show are algebraically similar, it would seem to be desirable to have \( G_1 \) and \( G_2 \)  normal in \( H \), and a way to try to do this is to make \( H \) as small as possible.
Our first two lemmas show that \( H \) may be taken to be solvable.

\begin{proposition}\label{prop:reduction-to-solvable}
Suppose that \( H \) is a connected Lie group with a connected compact subgroup \( K \) and  a connected solvable subgroup \( G \) of \( H \) such that \( H = G \cdot K \).
Let \( H_0 \) be a maximal connected solvable subgroup of \( H \) that contains \( G \).
Then
\begin{enumerate}
  \item \( H_0 \) is unique up to conjugation in \( H \);
  \item \( T \coloneqq  H_0 \cap K \) is a torus, and \( H_0 = G \cdot T \); and
  \item if \( G_1 \) is a connected solvable subgroup of \( H \) such that \( H = G_1 \cdot K \), then there is a conjugate \( G_1^h \) of \( G_1 \) in \( H \) contained in \( H_0 \) such that \( H_0 = G_1^h \cdot T \).
\end{enumerate}
If moreover \( H \) acts effectively on \( H/K \), then \( H_0 \) acts effectively on \( H_0 /T \).
\end{proposition}

\begin{proof}
As usual, denote by \( \mathfrak{h} \) the Lie algebra of \( H \), by \( \mathfrak{k} \) the Lie algebra of \( K \), and so on.

By hypothesis, \( \mathfrak{h} = \mathfrak{g} \oplus \mathfrak{k} \), and \emph{a fortiori} \( \mathfrak{h} = \mathfrak{h}_0 + \mathfrak{k} \).
If \( H_1 \) is a maximal connected solvable subgroup of \( H \) that contains \( G \), then \( \mathfrak{h} = \mathfrak{h}_1 + \mathfrak{k} \), and by Lemma \ref{lem:maximal-solvable}, \( \mathfrak{h}_1 \) is conjugate to \( \mathfrak{h}_0 \) under the action of the adjoint group of \( \mathfrak{h} \), whence \( H_1 \) is conjugate to \( H_0 \) in \( H \), and (i) holds.

Consider the action of \( H_0 \) on the quotient space \( H/K \).
Since \( G \) acts transitively, \( H_0 \) does so, and the stabiliser in \( H_0 \) of the point \( K \) in the quotient space \( H/K \) is \( H_0 \cap K \).
Now \( H = G\cdot K \), so that \( H_0 = G\cdot (H_0\cap K) = G \cdot T \) (by definition).
Further, \( T \) is connected because \( H_0 \) is connected and \( H_0 = G \cdot T \), solvable because \( H_0 \) solvable, and compact because it is a closed subgroup of \( K \).
Hence \( T \) is a torus.

If \( G_1 \) is a connected solvable subgroup of \( H \) such that \( H = G_1 \cdot K \), then \( \mathfrak{h} = \mathfrak{g}_1 + \mathfrak{k}  \).
If \( \mathfrak{h}' \) is a maximal solvable subalgebra of \( \mathfrak{h} \) that contains \( \mathfrak{g}_1 \), then \( \mathfrak{h} = \mathfrak{h}' + \mathfrak{k} \), and there exists \( h \in H \) such that \( \mathfrak{h}_0 = \Ad(h) \mathfrak{h}' \).
It follows that \( \Ad(h) \mathfrak{g}_1 \subseteq \mathfrak{h}_0 \), and it follows that \( hG_1h^{-1} \subseteq H_0 \) and \( H_0 = hG_1h^{-1} \cdot T \).

Finally if \( H \) acts effectively on \( H/K \), then so does the subgroup \( H_0 \), and \( H/K \) may be identified with \( H_0/T \).
\end{proof}

Let \( G_1 \) and \( G_2 \) be connected solvable Lie groups, and suppose that \( H \) is a solvable Lie group with a toral subgroup \( T \) such that \( H = G_1 \cdot T = G_2 \cdot T \) and \( Z(H) \cap T = \{e\} \).
Ideally, we would like to deduce that \( G_1 \) is a twisted version of \( G_2 \), or vice versa, but unfortunately this is not quite true; however, from Lemma \ref{lem:existence-of-G0-1}, there is a subgroup \( G_0 \) of \( H \) such that \( H = G_0 \rtimes T = G_1 \cdot T = G_2 \cdot T \), and hence both \( G_1 \) and \( G_2 \) are twisted versions of \( G_0 \).
In the proof of Lemma \ref{lem:existence-of-G0-1}, there were many possible choices for \( G_0 \), and it might be hoped that there is a choice with some additional properties that are useful and make it unique.

For instance, suppose that \( H \) is of polynomial growth.
One might hope that \( G_0 \) is nilpotent, but this is not always so.
However, one may define an abelian extension \( H^* \) of \( H \), in which \( H \) is a normal subgroup, with a toral subgroup \( T^* \) containing \( T \), such that \( H^* = G_1 \cdot T^* = G_2 \cdot T^* \), whose nilradical \( N \) satisfies \( H^* = N \rtimes T^* \).
Then \( G_1 \) and \( G_2 \) are both twisted versions of \( N \), which is known as the nilshadow of both \( G_1 \) and \( G_2 \).
We shall describe a construction of the group \( H^* \) like that of Alexopoulos \cite{Alexopoulos}, Dungey, ter Elst and Robinson \cite{Dungey-Elst-Robinson}, and Breuillard \cite{Breuillard-Geometry}, and show that one choice of \( G_0 \) is the real-radical of \( H^* \).

\subsection{Hulls and real-shadows}\label{ssec:hull-shadow}
In this section, we sketch the proof of the following theorem, whose roots are in results of Cornulier \cite[Section 2]{Cornulier-dimension} and of Jablonski \cite[Proposition 4.2]{Jablonski}, as well as earlier results of Gordon and Wilson \cite{Gordon-Wilson-fine, Gordon-Wilson-solvmanifolds} and even earlier work of Auslander and Green \cite{Auslander-Green}.

\begin{theorem}\label{thm:hull}
Let \( G \) be a connected solvable Lie group.
Let \( T \) be a maximal torus in a maximal compact subgroup of the automorphism group of \( G \), let \( H \) be the semidirect product \( G \rtimes T \), and let \( G_0 \) be the real-radical of \( H \).
Then \( H = G_0 \rtimes T \); further, there is a smallest subtorus \( J \) of \( T \), unique up to isomorphism, such that
\begin{equation}\label{eq:hull}
G \rtimes J = G_0 \rtimes J;
\end{equation}
\( G \) is a twisted version of \( G_0 \), with twisting homomorphism into \( J \), and vice versa.
\end{theorem}

\begin{proof}
Maximal compact subgroups of \( \Aut(G) \) are connected and conjugate in \( \Aut(G) \), and maximal tori of a given maximal compact subgroup \( K \) are conjugate in \( K \).
Hence \( H \) is unique up to isomorphism, and so \( G_0 \) is too.

We now show that \( H = G_0 \rtimes T \), using Lie algebra.
We choose a maximal torus with some convenient properties.
Let \( \mathfrak{g} \) and \( \mathfrak{n} \) be the Lie algebra of \( G \) and its nilradical.
Take a Cartan subalgebra \( \mathfrak{c} \) (see \cite[pp.~13--15]{Bourbaki7-9}) of \( \mathfrak{g} \).
The quotient \( (\mathfrak{n} + \mathfrak{c}) / \mathfrak{n} \) is a Cartan subalgebra of the abelian Lie algebra \( \mathfrak{g} / \mathfrak{n} \), by \cite[Corollary 2, page 14]{Bourbaki7-9}; hence \( \mathfrak{n} + \mathfrak{c} = \mathfrak{g} \).
Hence we may take a subspace \( \mathfrak{a} \) of \( \mathfrak{c} \) such that
\begin{equation}\label{eq:excellent-decomposition}
\mathfrak{g} = \mathfrak{n} \oplus \mathfrak{a}.
\end{equation}
Denote by \( \pi_{\mathfrak{a}} \) and \( \pi_{\mathfrak{n}} \) the corresponding projections of \( \mathfrak{g} \) onto \( \mathfrak{a} \) and \( \mathfrak{n} \).
Then
\begin{enumerate}\renewcommand{\labelenumi}{(\alph{enumi})}
  \item \( \Sad(X) Y = 0 \) and \( [\Sad(X) ,\Sad(Y)] = 0 \) for all \( X, Y \in \mathfrak{c} \); and
  \item the map \( X \mapsto \Sad( \pi_{\mathfrak{a}} X) \) is a Lie algebra homomorphism from \( \mathfrak{g} \) onto an abelian subalgebra of \( \Der(\mathfrak{g}) \), the Lie algebra of derivations of \( \mathfrak{g} \).
\end{enumerate}
Item (a) is proved as part of the proof of Proposition III.1.1 of \cite{Dungey-Elst-Robinson}; item (b) is Lemma 3.1 of \cite{Breuillard-Geometry}.
(To be precise, these authors have a type (R) assumption, but, as they state, this is not needed.)

We define the homomorphism \( \phi: \mathfrak{g} \to \Der(\mathfrak{g}) \) by
\[
\phi(X) \coloneqq  \SIad(\pi_{\mathfrak{a}} X),
\]
that is, the ``imaginary part'' (as in Lemma \ref{lem:shadow-step-1}) of the semisimple derivation \( \Sad(\pi_{\mathfrak{a}} X) \) constructed above.
This homomorphism annihilates \( \mathfrak{n} \), and also \( \mathfrak{s} \), and its image is a toral subalgebra of \( \Der(\mathfrak{g}) \).
Consider the closure \( J \) in \( \Aut(\mathfrak{g}) \) of the analytic subgroup corres\-ponding to \( \phi(\mathfrak{a}) \).
Then \( J \) is a torus.
(It is an abuse of notation to call this torus \( J \), but we shall later check that it does satisfy \eqref{eq:hull}, and so the abuse is justified.)

Let \( T \) be a maximal torus of \( \Aut(\mathfrak{g}) \) that contains \( J \), and define the Lie algebra \( \mathfrak{h} \) to be the semidirect sum algebra \(  \mathfrak{g} \oplus \mathfrak{t} \),
with Lie product given by
\begin{equation}\label{eq:hull-product}
  [(X, D), (Y, E)] \coloneqq  ([X,Y] + D(Y) - E(X), 0 )
\end{equation}
for all \( X, Y \in \mathfrak{g} \)  and all \( D, E \in \mathfrak{t} \).
In this proof, we write elements of \( \mathfrak{h} \) as ordered pairs rather than as sums as we feel that this helps understanding.
The subspace \( \mathfrak{g}_0 \) of \( \mathfrak{g} \) is defined by
\begin{align*}
  \mathfrak{g}_0  \coloneqq  \{(X, -\phi(X)) : X \in \mathfrak{g} \} .
\end{align*}
(Again, we are abusing notation here, but proving the next claim justifies the abuse.)
We claim that
\begin{enumerate}\renewcommand{\labelenumi}{\textrm{(\alph{enumi})}}
  \item \( \mathfrak{h} = \mathfrak{g}_0 \oplus \mathfrak{t} \);
  \item the map \( \tau: X \mapsto (X, -\phi(X)) \) is a bijection from
  \( \mathfrak{g} \) to \( \mathfrak{g}_0 \), and further,
\[
[\tau(X), \tau(Y)] =  \tau([X,Y]_{\rrad}),
\]
where
\begin{equation}\label{eq:shadow-product}
  [X,Y]_{\rrad} = [X,Y] - \phi(X)Y + \phi(Y)X
  \qquad\forall X, Y \in \mathfrak{g};
\end{equation}
  \item \( \mathfrak{g}_0 \) is an ideal and is the real-radical of \( \mathfrak{h} \).
\end{enumerate}

Parts (a) and (b) follow immediately from the definitions.

Third, \( \mathfrak{g}_0 \) is an ideal since
\( [\mathfrak{h}, \mathfrak{h}] \subseteq \mathfrak{n} \oplus \{0\} \subseteq \mathfrak{g}_0 \),
by \eqref{eq:hull-product} and Remark \ref{rem:derivations-radicals}.
To see that \( \mathfrak{g}_0 \) is split-solvable, we suppose that
\( X \in \mathfrak{g} \), \( Y \in \mathfrak{g}_{\C} \setminus \{0\} \),
and
\[
([X,Y] - \phi(X)Y + \phi(Y) X, 0)
=  [(X, -\phi(X)), (Y, -\phi(Y))]
=  \lambda (Y, -\phi(Y));
\]
it will suffice to show that \( \lambda \) is real.
If \( \lambda \neq 0 \), then \( \phi(Y) = 0 \), so we may suppose that \( Y \in \mathfrak{n}_{\C} \),
whence, from \eqref{eq:hull-product},
\( \ad(X) Y - \phi(X) Y = \lambda Y \), which implies that
\begin{align*}
(\SRad(\pi_{\mathfrak{a}} X) + \Nad(\pi_{\mathfrak{a}} X) + \ad(\pi_{\mathfrak{n}} X))Y
& = (\ad(\pi_{\mathfrak{a}} X) + \ad(\pi_{\mathfrak{n}} X) - \SIad(\pi_{\mathfrak{a}} X))Y \\
& =  \lambda Y .
\end{align*}
Consider the lower central series of \( \mathfrak{n}_{\C} \), as in  \eqref{eq:def-lcc}.
Since \( Y \neq 0 \), there exists \( j\in \N \) such that
\( Y \in (\mathfrak{n}_{\C})_{[j]} \setminus  (\mathfrak{n}_{\C})_{[j+1]} \).
Now all \( (\mathfrak{n}_{\C})_{[j]} \) are invariant under all derivations of \( \mathfrak{n}_{\C} \), and in particular under \( \SRad(\pi_{\mathfrak{a}} X) \), \( \Nad(\pi_{\mathfrak{a}} X) \) and \( \ad(\pi_{\mathfrak{n}} X)) \).
Thus these operators have quotient actions on the quotient algebra
\( (\mathfrak{n}_{\C})_{[j]} / (\mathfrak{n}_{\C})_{[j+1]} \), which we write as \( \SRqad(\pi_{\mathfrak{a}} X) \),
\( \Nqad(\pi_{\mathfrak{a}} X) \) and \( \qad(\pi_{\mathfrak{n}} X) \).
Evidently, \( \qad(\pi_{\mathfrak{n}} X))= 0 \),  \( \SRqad(\pi_{\mathfrak{a}} X) \) is semisimple with real eigenvalues, \( \Nqad(\pi_{\mathfrak{a}} X) \) is nilpotent; further, the last two quotient operators commute.
The eigenvalues of \( \SRqad(\pi_{\mathfrak{a}} X) \) and of
\( \SRqad(\pi_{\mathfrak{a}} X) + \Nqad(\pi_{\mathfrak{a}} X) \) coincide by \cite[Theorem 1, p.~A.VII.43]{Bourbaki-Alg2}.
So all the eigenvalues of \( \ad(X, -\phi(X)) \) are real, and \( \mathfrak{g}_0 \) is indeed split-solvable.

From Theorem \ref{thm:split-solvable-is-good-algebra}, \( \mathfrak{g}_0 \) is the real-radical of \( \mathfrak{h} \).

Next, we consider the groups that correspond to these Lie algebras.
We have already seen that \( T \) is a torus.
By Theorem \ref{thm:rradical}, the connected analytic subgroup \( G_0 \) of \( H \) whose Lie algebra is \( \mathfrak{g}_0 \) is closed and normal, and by Lemma \ref{lem:Lie-algebra-criterion}, \( H = G_0T \).
If \( g \in G_0 \cap T \), then the eigenvalues of \( \Ad(g) \), acting on \( \mathfrak{g}_0 \), are both real since \( G_0 \) is split-solvable, and of modulus \( 1 \), since \( T \) is a torus, and hence are all \( 1 \).
Since \( T \) is a torus, \( \Ad(g) \) is semisimple, and so \( \Ad(g) \), acting on \( \mathfrak{g}_0 \), is the identity mapping.
As \( T \) is abelian and \( \mathfrak{h} = \mathfrak{g}_0 \oplus \mathfrak{t} \), \( \Ad(g) \) acts trivially on \( \mathfrak{h} \).
It follows that \( g \in Z(H) \).
However, \( H = G \rtimes T \) and \( T \) is a torus of automorphisms of \( G \), so if \( g \in T \cap Z(H) \), then \( g = e \).
Thus \( G_0 \cap T \) is trivial and \( H = G \rtimes T = G_0 \rtimes T \).

We organised matters so that
\( \mathfrak{h} = \mathfrak{g} \oplus \mathfrak{t} = \mathfrak{g}_0 \oplus \mathfrak{t} \);
however, by construction,
\( \mathfrak{g} \oplus \mathfrak{j} = \mathfrak{g}_0 \oplus \mathfrak{j} \),
where \( \mathfrak{j} \) is the Lie algebra of \( J \), and there is no proper subtorus of \( J \) whose Lie algebra has this property.
At group level, \( H/G \) may be identified with \( T \) in \( \Aut(G) \) and \( J \) is the smallest subtorus of \( T \) that may be identified with the closure of \( G_0 G /G \) therein.
Thus
\[
G \rtimes J = G_0 \rtimes J,
\]
as required.
We have seen that \( H \) and hence \( G_0 \) are unique up to isomorphism: it follows that \( J \) is too.
\end{proof}

\begin{definition}\label{def:hull-and-shadow}
The \introd{real-shadow} of \( G \) is the group \( G_0 \) of Theorem \ref{thm:hull}, which is the real-radical of both \( G \rtimes J \) and of \( G \rtimes T \) (by Corollary \ref{cor:split-solvable-is-good-group} and by definition).
The \introd{hull} of \( G \) is defined to be the group \( G \rtimes J = G_0 \rtimes J \), which is the smallest solvable group containing both \( G \) and \( G_0 \).
The corresponding Lie algebras are also called the real-shadow and hull of \( \mathfrak{g} \).
\end{definition}

\begin{remark}\label{rem:shadows}
Let \( G_0 \) and \( G \) be connected solvable Lie groups, with \( G_0 \) split-solvable.

Let \( K \) be a torus of automorphisms of \( G \) and define \( H \coloneqq G \rtimes K \).
Suppose that we may write \( H = G_0 \cdot K \).
Then \( G_0 \) is normal in \( H \), by Corollary \ref{cor:split-solvable-is-good-group}; thus \( G \rtimes K = G_0 \rtimes K \).
By assumption, \( Z(H) \cap K = \{e\} \), so there is no nontrivial element of \( K \) that acts trivially on \( G_0 \).
We may suppose that \( K \) is a subtorus of the maximal torus \( T \) of automorphisms of \( G \) that appears in Theorem \ref{thm:hull}, and then \( J \subseteq K \subseteq T \).
Therefore \( G_0 \) is the real-radical of \( G \rtimes J \), and \( G_0 \) is the real-shadow of \( G \).

Let \( K \) be a torus of automorphisms of \( G_0 \) and define \( H \coloneqq  G_0 \rtimes K \).
Suppose that we may write \( H = G \cdot K \).
Let \( J \) be the smallest subtorus of \( K \) such that \( G \subseteq G_0 \rtimes J \); then \( G \) is normal in \( G_0 \rtimes J \) by Corollary \ref{cor:split-solvable-is-good-group}, and \( G_0 \rtimes J = G\rtimes J \).
Therefore \( G_0 \) is the real-shadow of \( G \).
\end{remark}

\begin{remark}\label{rem:shadows-2}
If \( G \) is split-solvable, then it is isomorphic to its real-shadow.
If \( G \) is of polynomial growth, then its real-shadow coincides with its nilshadow, since in this case the real-radical and the nilradical of the hull are the same.
\end{remark}

\subsection{Applications to metric Lie groups}
\label{ssec:consequences-thm-3}

Now we look at some of the consequences of the theory that we have developed, not only in Chapter \ref{sec:solvable}, but also earlier.

We recall that a connected solvable Lie group is simply connected if and only if it is contractible.
Thus a Lie group that may be made isometric to a connected simply connected solvable Lie group  is contractible.
By Remark \ref{rem:contractible-Lie-groups}, a contractible Lie group \( G \) may be written as \( R \rtimes L \), where the radical \( R \) is simply connected and the Levi subgroup \( L \) is a direct product of finitely many (possibly zero) copies of the universal covering group of \( \mathrm{SL}(2, \R) \).
Conversely, Theorem \ref{thm:main-2} shows that a Lie group \( G \) with the structure just described may be made isometric to a solvable Lie group.
By contrast, if a Lie group \( G \) may be made isometric to a connected simply connected nilpotent Lie group, then \( G \) is contractible and of polynomial growth, and by Lemma \ref{lem-poly-growth}, \( G \) is solvable.

These observations are the underlying reason for the inclusion of a solvability hypothesis in many but not all of the upcoming results.

\begin{corollary}\label{cor:may-be-made-isometric}
Let \( (G,d) \) be a connected solvable metric Lie group, \( H \) be a maximal connected solvable subgroup of \( \Iso(G,d) \) containing \( G \), and \( T \) be the stabiliser in \( H \) of the point \( e_G \) in \( G \).
Let \( G_0 \) be a normal subgroup of \( H \) such that \( H = G_0 \rtimes T \), as in Lemma \ref{lem:existence-of-G0-1}, and let \( G_1 \) be a connected solvable Lie group.
Then the following are equivalent:
\begin{enumerate}
  \item \( G_1 \) may be made isometric to \( (G,d) \);
  \item \( G_1 \) may be embedded in \( H \) in such a way that \( H = G_1 \cdot T \); and
  \item \( G \) and \( G_1 \) are both twisted versions of \( G_0 \) with twisting homomorphisms into \( T \).
\end{enumerate}
\end{corollary}

\begin{proof}
If \( G_1 \) may be made isometric to \( (G,d) \), then there is an embedding of \( G_1 \) in \( \Iso(G,d) \), by Theorem \ref{thm:isometry-isomorphism}, hence an embedding of \( G_1 \) in \( H \) by Proposition \ref{prop:reduction-to-solvable}, and so \( H \) contains closed subgroups \( G_1 \) and \( T \) such that \( H = G_1 \cdot T \).
Conversely, if \( G_1 \) may be embedded in \( H \) in such a way that \( H = G_1 \cdot T \), then \( G_1 \) may be made isometric to \( (G,d) \) by Corollary \ref{cor:isometric-groups-complementary-subgroups}.

The equivalence of (ii) and (iii) follows from Theorem \ref{thm:hull}.
\end{proof}

\begin{corollary}\label{cor:hull-and-isometry}
Let \( (G,d) \) be a connected solvable metric Lie group.
Let \( G^* \) and \( G_0 \) be the hull and the real-shadow of \( G \), and write \( G^* = G \rtimes J \), as in Theorem \ref{thm:hull}.
Then the following are equivalent:
\begin{enumerate}
  \item \( G_0 \) may be made isometric to \( (G,d) \); and
  \item \( d \) is invariant under conjugation by elements of \( J \).
\end{enumerate}
\end{corollary}

\begin{proof}
If \( d(jgj^{-1}, jhj^{-1}) = d(g,h) \) for all \( g, h \in G \) and all \( j \in J \), then we may view \( d \) as a \( G^* \)-invariant metric on \( G^*/J \), hence as a \( G_0 \)-invariant metric on \( G^*/J \), and hence as a metric on \( G_0 \).

Conversely, if \( G_0 \) may be made isometric to \( (G,d) \), then we may embed \( G \) and \( G_0 \) into a maximal connected solvable subgroup \( H \) of \( \Iso(G,d) \), by Proposition \ref{prop:reduction-to-solvable}, and write \( H = G_0 \cdot T = G \cdot T \)
for a suitable torus \( T \).
By Corollary \ref{cor:split-solvable-is-good-group}, \( H = G_0 \rtimes T \).
We may take a smaller subgroup \( H_0 \) of \( H \) of the form \( G_0 \rtimes J \), where \( J \) is a subtorus of \( T \), that is minimal subject to the requirement that \( G \subseteq H_0 \), and then, by Remark \ref{rem:shadows}, \( G \) is normal in \( H_0 \), and \( H_0 \) and \( G_0 \) are the hull and real-shadow of \( G \).
We may identify \( G \) with \( H_0/J \) and the metric \( d \) is \( H_0 \)-invariant, and \emph{a fortiori} is \( J \)-invariant.
\end{proof}

{
We now restate (and expand slightly) Theorem C.
}

\begin{theorem}\label{thm:main-3}
Let \( G_0 \) be a connected split-solvable Lie group, \( T \) be a maximal torus in \( \Aut(G_0) \), and \( d_0 \) be a \( T \)-invariant metric on \( G_0 \).
Let \( G_1 \) be a connected solvable Lie group.
Then the following are equivalent:
\begin{enumerate}
  \item \( G_1 \) may be made isometric to \( G_0 \);
  \item \( G_1 \) may be made isometric to \( (G_0,d_0) \);
  \item \( G_0 \) is the real-shadow of \( G_1 \);
  \item \( G_1 \) may be embedded in \( H \coloneqq  G_0 \rtimes T \) in such a way that \( H = G_1 \cdot T \); and
  \item \( G_1 \) is a twisted version of \( G_0 \) with twisting homomorphism into \( T \).
\end{enumerate}
\end{theorem}

\begin{proof}
Before we start our proof, we note that the existence of a \( T \)-invariant metric \( d_0 \) on \( G_0 \) is shown in Corollary \ref{cor:H/K-is-metrisable}.

It follows from Theorem \ref{thm:hull} and Remark \ref{rem:shadows} that (iii) and (iv) are equivalent, and from the discussion preceding Definition \ref{def:twist} that (iv) and (v) are equivalent.
Further, it is immediate that (ii) implies (i).

Suppose that (i) holds; then there exists a metric \( d \) on \( G_0 \) such that
\( G_1 \) may be viewed as a closed subgroup of \( \Iso(G_0,d) \) that acts simply transitively on \( G_0 \), as in Theorem \ref{thm:isometry-isomorphism}.
By Proposition \ref{prop:reduction-to-solvable}, there exists a connected solvable subgroup \( H \) of \( \Iso(G_0,d) \) such that \( H = G_0 \rtimes K = G_1 \cdot K \), where \( K \) is the stabiliser in \( H \) of \( e \) in \( G_0 \); since \( H \) is connected and solvable, so is \( K \) and since \( K \) is also compact, \( K \) is a torus.
Now (iii) holds, by Remark \ref{rem:shadows}.

To complete the proof, suppose that (iv) holds.
Then it is immediate that \( G_1 \) may be viewed as a closed subgroup of \( \Iso(G_0,d_0) \) that acts simply transitively on \( G_0 \), as in Theorem \ref{thm:isometry-isomorphism}, and so (ii) holds.
\end{proof}

The following are corollaries of Theorem \ref{thm:main-3} and the theory that we have developed.
This first follows immediately from the riemannian version of Corollary \ref{cor:H/K-is-metrisable} (which is well known) and  Theorem \ref{thm:main-3}

\begin{corollary}\label{cor:preferred-metric-on-split-solvable}
Let \( G_0 \) be a connected split-solvable Lie group.
Then there exists a riemannian metric \( d_0 \) on \( G_0 \) such that every connected solvable Lie group that may be made isometric to \( G_0 \) may be made isometric to \( (G_0, d_0) \).
\end{corollary}

Part of the next corollary also follows immediately from Theorem \ref{thm:main-3}.

\begin{corollary}
Let \( G_1 \) and \( G_2 \) be connected solvable Lie groups.
Then \( G_1 \) and \( G_2 \) may be made isometric if and only if their real-shadows are isomorphic.
\end{corollary}

\begin{proof}
First, suppose that \( G_0 \) is the real-shadow of both \( G_1 \) and \( G_2 \), and take a metric \( d_0 \) on the real-shadow \( G_0 \) that is invariant under a maximal torus \( T \) of \( \Aut(G_0) \).
Then both \( G_1 \) and \( G_2 \) may be made isometric to \( (G_0, d_0) \).

Conversely, suppose that \( G_1 \) and \( G_2 \) are connected solvable Lie groups that admit admissible left-invariant metrics \( d_1 \) and \( d_2 \) such that \( (G_1,d_1) \) and \( (G_2, d_2) \) are isometric.
Let \( H \) be a maximal connected solvable subgroup of the Lie group \( \Iso(G_1, d_1) \), and \( T \) be the stabiliser of the identity \( e \) of \( G_1 \) in \( H \).
By Corollary \ref{cor:may-be-made-isometric}, there is a normal subgroup \( G \) of \( H \) such that
\[
H = G \rtimes T = G_1 \cdot T = G_2 \cdot T.
\]

Let \( T^* \) be a maximal torus of \( \Aut(G) \) that contains \( T \), and let \( G_0 \) be the real-radical of \( H^* \coloneqq  G \rtimes T^* \), so that \( H^* = G_0 \rtimes T^* \) by Theorem \ref{thm:hull}.
Now \( G_1 \subseteq H_0 \subseteq H^*  \) and \( G_2 \subseteq H^* \) similarly.
We may check that \( G_1 \) and \( G_2 \) act simply transitively on \( H^*/T^* \), whence \( H^* \coloneqq  G_1 \cdot T^* = G_2 \cdot T^* \).
By Theorem \ref{thm:main-3}, \( G_0 \) is the real-shadow of both \( G_1 \) and \( G_2 \).
\end{proof}

Of course, if \( G_1 \) and \( G_2 \) have the same real-shadow \( G_0 \), then not only may they be made isometric to \( G_0 \), but to \( (G_0,d_0) \), where \( d_0 \) is the metric of Corollary \ref{cor:preferred-metric-on-split-solvable}.

We have already observed that in the nilpotent case, stronger results are possible.

\begin{corollary}
Let \( G_1 \) and \( G_2 \) be connected Lie groups and assume that \( G_1 \) is nilpotent.
The following are equivalent:
\begin{enumerate}
\item  \( G_2 \) and \( G_1 \) may be made isometric;
\item  \( G_2 \) is solvable and of polynomial growth and \( G_1 \) is its nilshadow.
\end{enumerate}
\end{corollary}

\begin{proof}
From Theorem \ref{thm:main-3},  (i) and (ii) are equivalent if \( G_2 \) is solvable; hence it suffices to assume that (i) holds and show that \( G_2 \) is solvable.

Since \( G_1 \) is nilpotent, it is of polynomial growth; hence \( G_2 \) is also of polynomial growth by Lemma \ref{lem03171107}.
Further, \( G_1 \) and \( G_2 \)  are homeomorphic, so their universal covering groups \( \tilde G_1 \) and \( \tilde G_2 \) are also homeomorphic.
As \( \tilde G_1 \) is contractible, \( \tilde G_2 \) is too.
We know that \( \tilde G_2 \) is of the form \( R \rtimes L \), where \( R \) is solvable and \( L \) is compact semisimple; it follows that \( L \) is trivial and so \( \tilde G_2 \) and \( G_2 \) are solvable.
%
%
%Further, \( G_1 \) and \( G_2 \) are homeomorphic, so their maximal compact subgroups are homotopic by Lemma \ref{lem:Iwasawa-max-cpct}.
%A maximal compact subgroup of \( G_1 \) is a torus, by Lemma \ref{lem-poly-growth}, so a maximal compact subgroup of \( G_2 \) is homotopic to a torus.
%Since the homotopy groups of a compact Lie group determine the Lie group up to local isomorphism (see \cite{Boekholt}), a maximal compact subgroup of \( G_2 \) is abelian, and so is also a torus.
%Hence \( G_2 \) is solvable by Lemma \ref{lem-poly-growth}.
\end{proof}

This leads to the following, which should be compared to a result of Kivioja and Le Donne \cite{Kivioja-LeDonne}.

\begin{corollary}\label{cor:Wilson-Enrico-Ville}
If \( G_1 \) and \( G_2 \) are connected nilpotent Lie groups, and both may be made isometric to the same connected Lie group \( G \) (not \emph{a priori} solvable, and possibly with different metrics), then \( G \) is solvable and \( G_1 \) and \( G_2 \) are isomorphic.
\end{corollary}

\begin{proof}
By the previous corollary, \( G \) is solvable and of polynomial growth, and both \( G_1 \) and \( G_2 \) are isomorphic to the nilshadow of \( G \).
\end{proof}

In the preceding corollary, if ``nilpotent'' is replaced with ``split-solvable'', we cannot deduce that \( G \) must be solvable.
However, if we replace ``nilpotent'' with ``split-solvable'' and we assume \emph{a priori} that \( G \) is solvable, then the conclusion that \( G_1 \) and \( G_2 \) are isomorphic still holds, as they are both isomorphic to the real-shadow of \( G \).

There are examples in the work of Gordon and Wilson \cite{Gordon-Wilson-solvmanifolds, Gordon-Wilson-fine} and of Jablonski \cite{Jablonski} where stronger results hold for split-solvable groups if an \emph{a priori} assumption of unimodularity is included.

Our final corollaries are concerned with quasi-isometry rather than isometry.
A general observation is that if two Lie groups may be made isometric using arbitrary admissible left-invariant metrics, then they may be made isometric for the derived semi-intrinsic metrics of \eqref{eq:derived-metric}, or for suitable riemannian metrics, as in Corollary \ref{cor:iso-of-two-manifolds-is-Lie-Riemannian}, and hence they are quasi-isometric when equipped with admissible left-invariant proper quasigeodesic metrics, as all such metrics on a given group are quasi-isometric.
We recall from Theorem \ref{thm:main-2} that a contractible homogeneous metric manifold \( (M,d) \) is homeomorphically roughly isometric to a connected, simply connected solvable metric Lie group.
With an additional hypothesis of polynomial growth, more may be said.

\begin{corollary}\label{cor04112351}
Let \( (M, d) \) be a contractible homogeneous metric manifold.
Suppose further that \( d \) is proper quasigeodesic and that \( M \) is of polynomial growth, as in \eqref{eq:poly-growth-space}.
Then \( (M, d) \) is quasi-isometrically homeomorphic to a simply connected riemannian nilpotent Lie group.
\end{corollary}

\begin{proof}
Theorem \ref{thm:main-2} shows that \( (M, d) \) is roughly isometrically homeomorphic to a simply connected solvable metric Lie group \( (H, d_H) \); by construction, \( (H, d_H) \) is proper quasigeodesic.

Let \( N \) be the nilshadow of \( H \).
By Theorem~\ref{thm:main-3}, there are metrics \( d_H' \) and \( d_N' \) on \( H \) and \( N \) such that \( (H, d_H') \) and \( (N, d_N') \) are isometric.
We may assume that \( d_H' \) and \( d_N' \) are riemannian, by Corollary \ref{cor:preferred-metric-on-split-solvable}.

Finally, \( d \) is {proper quasigeodesic} and all admissible left-invariant {proper quasigeodesic} metrics on a Lie group are quasi-isometric, so the identity map on \( H \)  is a quasi-isometry from \( d_H \) to \( d_H' \).
\end{proof}

With a slightly weaker hypothesis, we obtain a slightly weaker conclusion.

\begin{corollary}\label{cor04112356}
Let \( (M, d) \) be a homogeneous metric space of polynomial growth, and suppose that the metric \( d \) is {proper quasigeodesic}.
Then \( (M, d) \) is quasi-isometric to a connected simply connected nilpotent riemannian Lie group.
\end{corollary}

{
\begin{proof}
Theorem \ref{thm:main-2} shows that \( (M, d) \) is roughly isometric to a simply connected solvable metric Lie group \( (H, d_H) \), which is a metric quotient of \( (N,d) \) with compact fibre, and hence also of polynomial growth.

We now repeat the argument of the previous corollary.
\end{proof}

If \( (M,d) \) is a homogeneous metric space of polynomial growth, then the argument above shows that there is an admissible metric \( d' \) on \( M \), such that \( (M,d') \) is of polynomial growth and quasi-isometric to a connected simply connected nilpotent riemannian Lie group.
For example, we may take \( d' \) to be a derived semi-intrinsic metric, as defined just before Lemma \ref{lem:Busemann-gauge}.
}

\subsection{Notes and remarks}\label{sec:history}

\subsubsection*{4.3.\enspace Modifications}
In the language of Gordon and Wilson \cite{Gordon-Wilson-solvmanifolds}, our Lemma \ref{lem:modifications-are-normal} states that modifications of nilpotent Lie ideals are normal.
Gordon and Wilson \cite{Gordon-Wilson-solvmanifolds} proved the stronger result that modifications of nilpotent subalgebras are normal.
However, our Theorem \ref{thm:split-solvable-is-good-algebra} shows that nilpotent subalgebras of solvable Lie algebras with a toral complement are ideals, and so our two results combined include their theorem.

\subsubsection*{4.4.\enspace Split-solvability and the real-radical}

The real-radical, at the Lie algebra level, appears in the work of Jablonski \cite{Jablonski}.
In particular, the Lie algebra part of Theorem \ref{thm:rradical} and Theorem \ref{thm:split-solvable-is-good-algebra} are due to him.
In the language of Gordon and Wilson \cite{Gordon-Wilson-solvmanifolds}, \cite{Gordon-Wilson-fine}, the second part of Theorem \ref{thm:split-solvable-is-good-algebra} states that modifications of split-solvable groups are normal.

It was shown by Wolf \cite{Wolf-Growth} that a connected riemannian nilpotent group is normal in its isometry group.
On the other hand, the examples of symmetric spaces of the noncompact type show that a riemannian split-solvable connected Lie group \( G \) need not be normal in its isometry group \( H \); we may write \( H = G \cdot K \), where \( K \) is the stabiliser of a base-point, but it is certainly false that \( H = G \rtimes K \).
So Theorem \ref{thm:split-solvable-is-good-algebra} and Corollary \ref{cor:split-solvable-is-good-group} are perhaps a little surprising.

One important way in which our approach differs from that of Gordon and Wilson is that we use Mostow's theorem \cite{Mostow-subgroups} on maximal solvable subgroups to reduce questions of possible isometry of solvable groups to questions of possible isometry of solvable groups in a solvable supergroup.
This enables us to avoid some of the complications that arise in dealing with general Lie groups and algebras.

\subsubsection*{4.5.\enspace Twisted versions of groups and isometry of solvable groups}
Definition \ref{def:twist} is close to a proposal of Alekseevski\u\i\ \cite{Alekseevskii}, who used the expression \emph{twisting} rather than twisted version (or rather his translator did).
Actually, he considered the related question whether \( \{g \Phi(g)^{-1}: g \in G_1\} \) is a subgroup if \( G_1 \) is normal and \( \Phi: G_1 \to K \) is a homomorphism.
His answer is not definitive, but the situation is now clearer due to the contributions of Gordon and Wilson \cite{Gordon-Wilson-fine, Gordon-Wilson-solvmanifolds}, who looked at the corresponding question at the Lie algebra level, namely, when \( \{ X + \phi(X) : X \in \mathfrak{g}_1 \} \) is a subalgebra.

\subsubsection*{4.6.\enspace Hulls and real-shadows}

The idea of using a Cartan subalgebra of \( \mathfrak{g} \) to find a good complement of \( \nil(\mathfrak{g}) \), as in Theorem \ref{thm:hull}, or to construct the nilshadow, seems to be due to Alexseevski\u\i.
However, his class of solvable groups is restricted to those which arise in the study of riemannian homogeneous spaces of nonpositive curvature, and for these groups, the Cartan subalgebra \( \mathfrak{a} \) is abelian; extra ideas are needed to deal with general solvable Lie groups.
These are due to Alexopoulos (in the polynomial growth case).

The following example shows that not all the Cartan subalgebras that appear in the ``shadow construction'' are abelian.
We take the Lie algebra \( \mathfrak{g} \) with basis \( \{ U, V, X, Y, Z\} \) and commutation relations determined by linearity, antisymmetry and the nonzero basis commutation relations
\[
[X,Y] = Z, \qquad [X,U] = U, \qquad [Y,V] = V.
\]
This is a solvable extension of the abelian algebra \( \Span\{U,V\} \) by the nilpotent algebra \( \Span\{X, Y, Z\} \).
The Cartan subalgebra \( \Span\{X, Y, Z\} \) is nilpotent and not abelian.

The nilshadow appears in work of Auslander and Green \cite{Auslander-Green}, where the group \( G^* \) is called the \emph{hull} of \( G \); it seems that the term nilshadow was first used in \cite{Auslander-Tolimieri}.
Interestingly, it seems that type (R) also appeared for the first time in \cite{Auslander-Green}.
Their construction of the nilshadow used ideas from the theory of algebraic groups.
An alternative construction of the nilshadow, based on Lie algebras, appears in the work of Gordon and Wilson \cite{Gordon-Wilson-fine, Gordon-Wilson-solvmanifolds}, phrased in the language of modifications; their work was not restricted to groups of polynomial growth, and perhaps for this reason they did not make explicit the connection with the construction of Auslander and Green.
The Lie algebraic construction of the nilshadow was found later by Alexopoulos \cite{Alexopoulos}, and developed by Dungey, ter Elst, and Robinson \cite{Dungey-Elst-Robinson} and by Breuillard \cite{Breuillard-Geometry}.
The nilshadow has been used quite extensively in the area of harmonic analysis on Lie groups, and in applications to nonriemannian metric geometry of Lie groups.

What we call the real-shadow is more recent.
For groups that need not be of polynomial growth, the detailed investigation of Gordon and Wilson \cite{Gordon-Wilson-fine, Gordon-Wilson-solvmanifolds} identified a special subgroup \( G_0 \), said to be in standard position, that is sometimes split-solvable.
Cornulier \cite{Cornulier-dimension} developed an object that he called the trigshadow, using techniques closer to those of Auslander and Green, and in particular working at group level rather than algebra level.
In the recent work of Jablonski \cite{Jablonski}, which has roots in the work of Gordon and Wilson, the idea of a maximal split-solvable ideal appears and the real-shadow as viewed as a maximal split-solvable ideal of a larger Lie algebra.

Our construction of the hull \( G^* \) is like that of Alexopoulos, Dungey, ter Elst and Robinson, and of Breuillard.

Recall from Lemma \ref{lem:existence-of-G0-1} that if \( H \) is a solvable Lie group with a toral subgroup \( T \) such that \( Z(H) \cap T = \{e\} \), then we may find a normal subgroup \( G_0 \) of \( H \) such that \( H = G_0 \rtimes T \).
Gordon and Wilson \cite{Gordon-Wilson-fine, Gordon-Wilson-solvmanifolds} spend some effort on finding a choice of \( G_0 \) ``in standard position''.
Essentially this is a group which is ``as real as possible''.
From our point of view, the construction of \( G_0 \) proceeds, using Lie algebras, as follows: first, take a Cartan subalgebra \( \mathfrak{c} \) of \( \mathfrak{h} \) containing \( \mathfrak{t} \) (this is possible), and then a subspace \( \mathfrak{a} \) of \( \mathfrak{c} \) such that \( \mathfrak{h} = \mathfrak{n} \oplus \mathfrak{a} \oplus \mathfrak{t} \).
Replace any \( X \in \mathfrak{a} \) such that \( \SIad(X) \), the imaginary part of the semisimple part of \( \ad(X) \), as in Corollary \ref{cor:super-Jordan}, coincides with \( \ad(U) \) for some \( U \) in \( \mathfrak{t} \) by \( X-U \).
This produces a new subspace \( \tilde{\mathfrak{a}} \) such that
\( \mathfrak{h} = \mathfrak{n} \oplus \tilde{\mathfrak{a}} \oplus \mathfrak{t} \).
Let \( \mathfrak{g}_0 \) be \( \mathfrak{n} \oplus \tilde{\mathfrak{a}} \).

Gordon and Wilson often use the Killing form to construct nice subalgebras, such as the nilshadow of a solvable Lie algebra, and orthogonal complements of compact subalgebras appear in their development, much as in Corollary \ref{lem:amenable-contractible}.

\subsubsection*{4.7.\enspace Consequences and applications}
Gordon and Wilson \cite[Example 2.8]{Gordon-Wilson-solvmanifolds} give examples of nonisomorphic connected simply connected solvable Lie groups \( G_1 \) and \( G_2 \) that are isometric, but they are not isometric to their real-shadow.

The universal covering group \( H \) of the group \( \R^2 \rtimes \textrm{SO}(2) \) of rigid motions of \( \R^2 \) that preserve orientation is a simply connected three-dimensional solvable Lie group that admits a left-invariant subriemannian metric \( d \) such that \( (H, d) \) is not bi-Lipschitz equivalent to any nilpotent group.
Indeed, the two simply connected three-dimensional nilpotent Lie groups are the abelian group \( \R^3 \), which is the nilshadow of \( H \), and the nonabelian Heisenberg group \( \mathbb H \).
However, if \( d \) is a suitable left-invariant subriemannian metric on \( H \), then \( (H, d) \) is not even quasiconformally equivalent to either \( \R^3 \) or \( \mathbb H \); see \cite{Fassler-Koskela-LeDonne}.
Nevertheless, \( (H, d) \) is locally bi-Lipschitz to \( \mathbb H \) with the standard subriemannian metric.

Apropos of Theorem \ref{thm:split-solvable-is-good-algebra} and Corollary \ref{cor:split-solvable-is-good-group}, in the riemannian case, the normality of a nilpotent Lie group \( N \) in its isometry group was proved by Wolf \cite{Wolf-Growth} and rediscovered by Wilson \cite{Wilson-homogeneous}.

In the special case where \( (M,d) \) is of polynomial growth, so is every group \( G \) that acts simply transitively and isometrically on \( (M,d) \).
If any such group \( G \) is nilpotent, then \( G \) is normal in \( \Iso(M,d) \) by Corollary \ref{cor:After-Wolf}.
This was extended to unimodular split-solvable groups by Gordon and Wilson \cite{Gordon-Wilson-fine, Gordon-Wilson-solvmanifolds}.

Corollary \ref{cor:Wilson-Enrico-Ville} was known for nilpotent \( G \) and arbitrary metrics, and for solvable \( G \) with riemannian metrics; see \cite{Alekseevskii, Gordon-Wilson-fine, Gordon-Wilson-solvmanifolds, Kivioja-LeDonne, Wilson-homogeneous, Wolf-Growth}.
Kivioja and Le Donne also showed that isometries of nilpotent metric Lie groups are affine, that is, are composed of translations and group automorphisms.

\section{Characterisation of metrically self-similar Lie groups}\label{sec:dil}

In this chapter we present a study of homogeneous metric spaces that admit a (non-trivial) metric dilation.
Recall that a \introd{metric dilation} of \emph{factor} \( \lambda \) in a metric space \( (M, d) \) is a bijection \( \delta \) of \( M \)
such that \( d(\delta x, \delta y) = \lambda d(x,y) \) for all \( x,y \in M \).
Note that dilations of factor 1 are isometries; since we are interested in the case when \( \delta \) is not an isometry, we will always assume that \( \lambda\neq1 \), unless otherwise stated.

Theorem~D, which appears here as Theorem \ref{thm:main-4}, characterises all homogeneous metric spaces that admit a metric dilation as metrically self-similar Lie groups.

\begin{definition}
A \introd{metrically self-similar Lie group} is a triple \( (G,d,\delta) \), where \( G \) is a connected Lie group, \( d \) is an admissible left-invariant metric on \( G \), and \( \delta \) is an automorphism of \( G \) such that \( d(\delta x,\delta y)=\lambda d(x,y) \) for some \( \lambda\neq1 \).
\end{definition}

In the rest of this chapter, we will give a more precise description of metrically self-similar Lie groups in Section \ref{ssec:self-sim} and some motivation for their study.
In particular, we will show how metrically self-similar Lie groups appear
as tangents to doubling spaces in Section \ref{ssec:tangents} and as parabolic boundary of Heintze groups \ref{ssec:Heintze}.
Finally, we prove Theorem~D in Section~\ref{ssec:proof-thm-main-4}.

%%%%%%%%%%%%%%%%%%%%%%%%%%%%%%%%%%%%%%%%%%%%%%%%%%%%%%%%%%%%%%%
\subsection{Properties of self-similar Lie groups}\label{ssec:self-sim}
The basic examples of metrically self-similar Lie groups are finite dimensional normed vector spaces, where the dilation is scalar multiplication.
Several other examples are already available when \( G=\R^2 \).

If \( \alpha,\beta > 1 \), then the automorphisms \( \delta_\lambda \) corresponding to the matrix
\begin{equation}\label{eq60893540}
\begin{pmatrix} \lambda^\alpha & 0 \\ 0 & \lambda^\beta \end{pmatrix}
\end{equation}
are all dilations of factor \( \lambda \) for metrics including
\[
d((x,y),(x',y')) = \max\{|x-x'|^{1/\alpha},|y-y'|^{1/\beta}\}
\]
or, when \( \alpha=\beta \),
\[
d(x,y) = \|x-y \|^{1/\alpha}
\]
where \( \|\cdot\| \) is any norm on \( \R^2 \).
In \cite[Proposition 5.1]{LeDonne-Nicolussi-TAMS}, it is shown that when \( \alpha = \beta = 2 \),  there exists a homogeneous metric \( d \) whose spheres are fractals in \( \R^2 \).

When \( \alpha > 1 \), the automorphisms \( \delta_\lambda \) corresponding to the matrix
\begin{equation}\label{eq6089354f}
\lambda^\alpha
\begin{pmatrix}
\cos(\log\lambda) & -\sin(\log\lambda)\\
\sin(\log\lambda) & \cos(\log\lambda)
\end{pmatrix}
\end{equation}
are dilations of factor \( \lambda \) for the metric \(  d(x,y)=\|x-y\|^{1/\alpha} \), where \( \|\cdot\| \) is the euclidean norm.

If \( \alpha>1 \), then there is a left-invariant metric \( d \) on \( \R^2 \) for which the automorphisms \( \delta_\lambda \), corresponding to the matrices
\begin{equation}\label{eq6089355f}
\begin{pmatrix}
\lambda^\alpha & \lambda^\alpha\log(\lambda^\alpha) \\
0 &\lambda^\alpha
\end{pmatrix},
\end{equation}
are dilations of factor \( \lambda \).
These dilations appear in \cite{Bonk-Kleiner} in the study of visual boundaries of Gromov hyperbolic spaces.
See also \cite{Xie} for further results and examples in \( \R^n \).

\begin{definition}
A (positive) \introd{grading} of a Lie algebra \( \mathfrak{g} \) is a decomposition \( \mathfrak{g} = \bigoplus_{t \in \R^+}\mathfrak{v}_t \) such that \( [\mathfrak{v}_s,\mathfrak{v}_t]\subseteq \mathfrak{v}_{s+t} \) for all \( s,t \in \R^+ \).
A Lie group \( G \) is \emph{gradable} if it is simply connected and its Lie algebra admits a grading.
\end{definition}

Note that finitely many \( \mathfrak{v}_t \) have positive dimension, because \( \mathfrak{g} \) has finite dimension; further, a gradable group is nilpotent.
When \( G \) is a gradable Lie group with Lie algebra grading \( \mathfrak{g} = \bigoplus_{t \in \R^+}\mathfrak{v}_t \), we may define the \emph{standard dilations} \(  \delta_\lambda: G\to G \) by requiring that the differential \( (\delta_\lambda)_* \) acts as the scalar \( \lambda^t \) on \( \mathfrak{v}_t \).
It is known that, for standard dilations, a metric \( d \) exists on \( G \) so that \( (G, d, \delta_\lambda) \) is a metrically self-similar Lie group if and only if \( \mathfrak{v}_t=\{0\} \) for all \(  t \in (0,1) \), see \cite{Folland-Stein-Hardy}.
For much more on gradable groups, see \cite{LeDonne-Rigot-Besicovitch} and the references cited there.

Gradable groups are the only Lie groups that admit an automorphic dilation, by the following theorem.

\begin{theorem}[{Siebert, \cite{Siebert}}]\label{thm:Siebert}
Let \( G \) be a connected Lie group and suppose that there exists a Lie group automorphism  \( \delta : G\to G \) such that
\begin{equation}\label{eq60893e07}
\lim_{n\to\pinfty} \delta^ng=e_G
\qquad\forall g \in G.
\end{equation}
Then \( G \) is gradable, nilpotent and simply connected.
\end{theorem}

\begin{corollary}\label{cor:Siebert-plus}
If \( (G,d,\delta) \) is a metrically self-similar Lie group, then \( G \) is gradable, nilpotent and simply connected.
Moreover, \emph{all} metric dilations on \( (G,d) \) are Lie group automorphisms of \( G \).
\end{corollary}
\begin{proof}
Since a metrically self-similar Lie group admits a contractive automorphism by definition, the first statement follows from Theorem~\ref{thm:Siebert}.

If \( f:G\to G \) is a metric dilation of factor \( \mu>0 \), then it is also an isometry from \( (G,\mu d) \) to \( (G,d) \), and by \cite[Proposition 2.4]{Kivioja-LeDonne}, isometries between connected nilpotent Lie groups are group isomorphisms composed with translations.
\end{proof}

The proof of Theorem~\ref{thm:Siebert} constructs a grading of \( \mathfrak g \) in terms of the generalised eigenspaces of the automorphism \( \delta \).
If \( \delta \) is a metric dilation with factor \( \lambda<1 \), then the contraction property~\eqref{eq60893e07} is clearly satisfied.
\emph{Vice versa}, if we fix \( \delta \) and \( \lambda \), the following proposition gives a necessary and sufficient condition for the existence of a distance such that \( \delta \) is a dilation of factor \( \lambda \).
It follows from Proposition~\ref{prop12061651} that if \( \delta \) satisfies~\eqref{eq60893e07}, then there is a distance for which \( \delta \) is a dilation of some factor \( \lambda \).

\begin{proposition}[{\cite[Theorem D]{Golo-LeDonne}}]\label{prop12061651}
Suppose that \( G \) is a Lie group, \( \delta \) is a Lie group automorphism of \( G \) and \( 1\neq \lambda\in\R^+ \).
The following statements are equivalent:
\begin{enumerate}
\item
there is an admissible distance on \( G \) for which \( \delta \) is a dilation of factor \( \lambda \), and
\item
the Lie group \( G \) is connected and simply connected, the eigenvalues of \( \delta_* \) have modulus no greater than \( \lambda \) if \( \lambda<1 \), no less than \( \lambda \) if \( \lambda>1 \), and the complexification of \( \delta_* \) is diagonalisable on the generalised eigenspaces corresponding to the eigenvalues of modulus \( \lambda \).
\end{enumerate}	
\end{proposition}

For instance, Proposition~\ref{prop12061651} implies that, if \( \alpha=1 \), then for no \( \lambda\neq1 \) is the map in~\eqref{eq6089355f} a dilation of factor \( \lambda \) for any left-invariant distance on \( \R^2 \).

One observes that in the examples we gave there is not only one dilation but a one-parameter family of dilations \( \delta_\lambda \), one for each factor \( \lambda \in \R^+ \).
In fact, this is the general scenario, up to bi-Lipschitz changes of the distance, as we will explain.

A multiplicative one-parameter subgroup of \( \Aut(G) \), by which we mean a mapping \( \lambda\mapsto\delta_\lambda \) from \( \R^+ \) to \( \Aut(G) \) such that \( \delta_{\lambda\lambda'} = \delta_\lambda\delta_{\lambda'} \), is determined by its infinitesimal generator, which is a derivation \( A \) of \( \mathfrak g \) such that
\begin{equation}\label{eq12061455}
(\delta_\lambda)_* = \expe^{(\log\lambda)A} .
\end{equation}
We say that \( d \) is \emph{\( A \)-homogeneous} if \( \delta_\lambda \) is a metric dilation of factor \( \lambda \) for \( d \) for all \( \lambda \in \R^+ \).

For example, the dilations in~\eqref{eq60893540}, \eqref{eq6089354f}, and \eqref{eq6089355f} are of the form \( \expe^{(\log\lambda)A} \), where \( A \) is the matrix
\[
\begin{pmatrix}\alpha&0\\0&\beta\end{pmatrix},
\qquad
\begin{pmatrix}0&-\alpha\\\alpha&0\end{pmatrix},
\qquad\text{and}\qquad \begin{pmatrix}\alpha&\alpha\\0&\alpha\end{pmatrix}.
\]

In terms of the derivation \( A \),
the characterisation in Proposition~\ref{prop12061651} is equivalent to requiring that the real parts of the eigenvalues of \( A \) are all at least \( 1 \), and that \( A \) is diagonalisable over \( \C \) on the generalised eigenspaces corresponding to the eigenvalues with real part equal to 1.
See \cite[Theorem B]{Golo-LeDonne}.
The standard dilations have this property.

\begin{proposition}[{\cite[Theorem C]{Golo-LeDonne}}]\label{prop12061648}
If \( (G,d,\delta) \) is a self-similar metric Lie group, where \( \delta \) is a metric dilation of factor \( \lambda \), then there is \( A\in\Der(\mathfrak g) \) with eigenvalues in \( [1,\pinfty) \) and an \( A \)-homogeneous distance \( d' \) on \( G \) such that \( \delta \) is also a dilation of factor \( \lambda \) for \( d' \).
Further, for any such \( A \) and \( d' \), the identity mapping from \( (G,d) \) to \( (G,d') \) is bi-Lipschitz.
\end{proposition}

Note that such a bi-Lipschitz change of the distance may be necessary.
For instance, consider the piecewise linear function \( D: [0, \pinfty) \to [0, \pinfty) \) with nodes at \( (0, 0) \), and \( (4^n, 2^n) \), where \( n \in \Z \).
The nodes all lie on the graph \( y = x^{1/2} \), so \( D \) is increasing and concave, and \( d(x, y) \coloneqq  D(|x-y|) \) is a translation-invariant metric on \( \R \).
Then \( \delta(x) = 4x \) is a metric dilation of factor \( 2 \), but \( d \) does not admit dilations whose factors are not integer powers of \( 2 \).
However, \( d \) is bi-Lipschitz equivalent to \( d' \), where \( d'(x,y) = \sqrt{|x-y|} \), which has dilations of every factor.

One interesting consequence of Proposition~\ref{prop12061648} is that the ``imaginary part'' of the derivation \( A \) may be dropped.
For instance, distances with dilations of the form~\eqref{eq6089354f} are bi-Lipschitz equivalent to distances with dilations of the form~\eqref{eq60893540} and \( \alpha=\beta \).
On the contrary, any distance \( d \) with dilations of the form~\eqref{eq6089355f} cannot be bi-Lipschitz equivalent to any distance \( d' \) with dilations of the form~\eqref{eq60893540}, because \( d \) does not attain its conformal dimension, as explained in \cite[Section 6]{Bonk-Kleiner}, while \( d' \) does.

%%%%%%%%%%%%%%%%%%%%%%%%%%%%%%%%%%%%%%%%%%%%%%%%%%%%%%%%%%%%%%%
\subsection{Self-similar Lie groups as tangent spaces}\label{ssec:tangents}

Theorem~D, together with the above description of metrically self-similar Lie groups, leads to a metric characterisation of homogeneous metric spaces that are
\emph{homothetic}, that is, that admit dilations for every positive factor,
see \cite[Theorem E]{Golo-LeDonne}.
In fact, homogeneous homothetic metric spaces are isometric to metrically self-similar Lie groups endowed with a \( A \)-homogeneous distance.

It follows that metrically self-similar Lie groups are the typical tangents
to doubling metric spaces with unique tangents.

\begin{proposition}[{\cite[Theorem F]{Golo-LeDonne}}]\label{prop05011958}
Let \( X \) be a metric space with a doubling measure \( \mu \).
Assume that \( X \) has a unique (\( p \)-dependent) tangent at \( \mu \)-almost every \( p \) in \( X \).
Then for \( \mu \)-almost every \( p \) in \( X \), the tangent to \( X \) at \( p \) is a metrically self-similar Lie group \( G_p \) endowed with an \( A \)-homogeneous distance, for some \( A \) in \( \Der(\mathfrak{g}_p) \).
\end{proposition}

%%%%%%%%%%%%%%%%%%%%%%%%%%%%%%%%%%%%%%%%%%%%%%%%%%%%%%%%%%%%%%%
\subsection{Self-similar Lie groups as parabolic visual boundaries} \label{ssec:Heintze}
A well known motivation for the study of metrically self-similar Lie groups is their appearance as parabolic visual boundaries of negatively curved homogeneous riemannian manifolds.
More precisely, Heintze \cite{Heintze} showed that every simply connected negatively curved homogeneous riemannian manifold is isometric to a riemannian Lie group \( (G, g) \) that is a semidirect product \( N \rtimes_A \R \), where \( N \) is a simply connected nilpotent Lie group and at the Lie algebra level, \( \R \) acts on \( \mathfrak{n} \) by a derivation \( A \) whose eigenvalues have strictly positive real parts.
The parabolic visual boundary of \( (G,g) \)  may be  identified with the Lie group \( N \) endowed with a \( A \)-homogeneous distance, as we will explain below.
It is important to remark that the quasi-isometric classification of Heintze groups is equivalent to the quasi-symmetric classification of their parabolic boundaries,
which in turn reduces to a bi-Lipschitz classification of metrically self-similar Lie groups (see \cite{Kivioja-LeDonne-Nicolussi} and references therein).

We now explain how the parabolic boundary is identified with a metrically self-similar Lie group.
The construction is well known, but we include it here for completeness.
A Heintze group \( G = N \rtimes_A \mathbb{R} \) may always be equipped with a left-invariant infinitesimal riemannian metric \( g \) such that the maximum sectional curvature is \( -1 \), and \( N \times \{0\} \) is orthogonal to \( \{1_N\} \times \mathbb{R}\).
Denote by \( d_g \) the distance function induced by \( g \).

The vertical line with support \(\{1_N\} \times \mathbb{R} \) is length minimising between any two of its points,
as one may check directly by comparing any path with its own projection to the line \(\{1_N\} \times \mathbb{R} \).
By using the left-invariance of the distance, we deduce from this that there exists \( \mu \in \R^+ \) such that all the curves \( s \mapsto (n,\mu s) \), where \( n \in N\), are isometric embeddings.
Let \( \xi \colon [0,\pinfty) \to G \) be the curve \( s \mapsto (1_N,\mu s) \).
Following~\cite[p.~383]{Hersonsky-Paulin}, we define the \introd{parabolic visual boundary} \( \partial_\infty(G,d_g) \) of \( (G,d_g) \) to be the set of isometric embeddings \( \gamma \colon \R \to (G,d_g) \) that satisfy
\begin{equation} \label{eq:asymptotic}
\lim_{s \to \pinfty} d_g( \gamma(s),\xi(s) ) =0 .
\end{equation}
The parabolic visual boundary is then equipped with the so-called Hamenst\"adt distance:
\begin{equation}\label{eq608fcc87}
d(  \sigma,  \gamma) \coloneqq  \exp(- \tfrac{1}{2} \lim_{s\to \pinfty}(2s - d_g(\sigma(-s),\gamma(-s))))
.
\end{equation}
On a Heintze group, the expression in~\eqref{eq608fcc87} need not satisfy the triangle inequality; here it does, because we constructed a \( \mathrm{CAT}(-1) \) metric on \( G \) and chose the unit speed parametrisation of the ray \( \xi \).
See both~\cite[p.~383]{Hersonsky-Paulin} and \cite{Bourdon95} or \cite[Proposition 2.17]{Bourdon18}.

Next we explain how the Lie group \( N \) may be identified with \( \partial_\infty(G,d_g) \).
In one direction, to each element \( n \in N \) we associate the infinite geodesic \( \gamma(s)= (n,\mu s) \).
To show that this is indeed well defined, one needs to verify that \eqref{eq:asymptotic} holds.
We denote by \( \phi_s \) the automorphism of \( N \) with differential \( \expe^{sA} \).
By using the group law of the semidirect product and the left-invariance of \( g \), we calculate that
\begin{equation}\label{eq:calculation}
\begin{aligned}
\lim_{s \to \pinfty} d_g( (n,\mu s+s_0),(1_N,\mu s) )
&=
\lim_{s \to \pinfty} d_g( (1_N,\mu s)(\phi_{-\mu s}n, s_0),(1_N,\mu s) )
\\
&=
\lim_{s \to \pinfty} d_g( ((\phi_{-\mu s}n, s_0),(1_N,0) )
\\
&=
d_g( (1_N,s_0),(1_N,0) ) = s_0/\mu
\end{aligned}
\end{equation}
for any \( s_0 \in \mathbb{R} \).
In particular, putting \( s_0=0 \), the curve \( \gamma \) satisfies \eqref{eq:asymptotic}.

In the other direction, we consider \( \gamma \in \partial_\infty(G,d_g) \) and write \( (n,s_0) \in N \rtimes_A \mathbb{R} \) for the point \( \gamma(0) \).
If we can show that \( s_0 =0 \), that is, \( \gamma(0) \in N \times \{0\} \),  then we will have found a natural map from \( \partial_\infty(G,d_g) \) to \( N \).
Let \( \sigma \) be the infinite geodesic \( s \mapsto (n,\mu s+s_0 ) \). First, from \eqref{eq:calculation},
\[ \lim_{s \to \pinfty} d( \sigma(s),\xi(s) ) = s_0/\mu . \]
We see that
\[
0 \le \lim_{s \to \pinfty} d( \sigma(s),\gamma(s) ) \le \lim_{s \to \pinfty} d( \sigma(s),\xi(s) )  + \lim_{s \to \pinfty} d( \xi(s),\gamma(s) ) \le s_0/\mu .
\]
Thus the triangle formed by the points \( \gamma(0) = \sigma(0) \), \( \sigma(s) \) and \( \gamma(s) \) has two sides of length \( s \) and one side of length at most \( s_0/\mu \).
Because the space \( (G,d_g) \) is \( \mathrm{CAT}(-1) \) and  \( s \) may be taken arbitrarily large, the triangle has to be degenerate, that is, \( \gamma = \sigma \).

Using the correspondence above, we view \( d \) as an admissible left-invariant distance function on \( N\) (to show admissibility, \cite[Theorem A]{Golo-LeDonne} may be used).
A computation analogous to \eqref{eq:calculation} proves that
\begin{align*}
d(\phi_t(n), &\phi_t (n')) \\
&=
\exp(- \tfrac{1}{2} \lim_{s\to \pinfty}(2s - d_g((\phi_t (n),- \mu s),(\phi_t (n'),- \mu s))))
\\
&=\exp(- \tfrac{1}{2} \lim_{s\to \pinfty}(2s - d_g((1_N,t)(n,- \mu s-t),(1_N,t)(n',- \mu s-t))))
\\
&=\exp(- \tfrac{1}{2} \lim_{s\to \pinfty}(2s - d_g((n,- \mu s-t),(n',- \mu s-t))))
\\
&=\exp(- \tfrac{1}{2} \lim_{h\to \pinfty}(2(h - t/\mu) - d_g((n,- \mu h),(n',- \mu h))))
\\
&=\mathrm{e}^{t/\mu}
\exp(- \tfrac{1}{2} \lim_{h\to \pinfty}(2h  - d_g((n,- \mu h),(n',- \mu h))))
\\
&=\mathrm{e}^{t/\mu}  d(n,n') .
\end{align*}
The differential of the map \( \delta_t \coloneqq \phi_{\mu \log(t)} \) is \( (\delta_t)_* =   (\phi_{\mu \log(t)} )_* = \expe^{\mu log(t) A} \), and thus \( t \mapsto \delta_t \) is the one-parameter subgroup of automorphisms associated to the derivation \( \mu A \) and \( d(\delta_t n, \delta_t n') = t  d(n,n')  \)  for all \( n,n' \in N \).
We have thus proved that, after the identification of \( N \) with the parabolic visual boundary, the Hamenst\"adt distance is \( (\mu A) \)-homo\-geneous on \( N \).

%%%%%%%%%%%%%%%%%%%%%%%%%%%%%%%%%%%%%%%%%%%%%%%%%%%%%%%%%%%%%%%
\subsection{Proof of Theorem~D}\label{ssec:proof-thm-main-4}

We restate Theorem~ for the reader's convenience.

\begin{theorem}\label{thm:main-4}
If a homogeneous metric space admits a metric dilation, then it is isometric to a metrically self-similar Lie group.
Moreover, all metric dilations of a metrically self-similar Lie group are automorphisms.
\end{theorem}

The last sentence in Theorem~\ref{thm:main-4} was proved in Corollary~\ref{cor:Siebert-plus}.
Throughout this section, we assume that \( (M, d) \) is a homogeneous metric space, \( \lambda \in (1, \pinfty) \), and \( \delta \) is a bijection of \( M \) such that \( d(\delta x, \delta y) = \lambda d(x,y) \) for all \( x,y \in M \).
Since \( M \) is locally compact and isometrically homogeneous, it is complete, and the Banach fixed point theorem shows that \( \delta \) has a unique fixed point, \( o \) say.
We prove a few preliminary results.

\begin{lemma}\label{lem-dilations}
The metric space \( (M, d) \) is proper and doubling.
\end{lemma}
\begin{proof}
The ball \( B(o, r) \) is relatively compact for all sufficiently small \( r \); using the dilation we see that this holds for all \( r \in \R \), and \( (M,d) \) is proper.

We now show that \( (M, d) \) is a doubling metric space.
Since the closed ball \( \barB(o, \lambda) \) is compact, there are points \( x_1, \ldots, x_k\in \barB(o, \lambda) \) such that
\[
\barB(o, \lambda)\subseteq\bigcup_{i=1}^k B(x_i, 1/2).
\]
Take \( R\in\R^+ \), and define \( n \coloneqq  \lfloor \log_\lambda R \rfloor \), so that \( 1 \le \lambda^{-n} R < \lambda \).
Then
\[
\delta^{n} B(x_i, 1/2)\subseteq \delta^{n} B(x_i, \lambda^{-n} R/2) = B(\delta^{n}x_i, R/2) ,
\]
and so
\[
B(o, R)
= \delta^{n}(B(o, \lambda^{-n}R))
\subseteq \delta^{n}(B(o, \lambda))\subseteq \bigcup_{i=1}^k B(\delta^{n}x_i, R/2).
\]
Since \( (M, d) \) is isometrically homogeneous, \( (M, d) \) is doubling.
\end{proof}

Let \( H \) denote the connected component of the identity in \( \Iso(M,d) \).

\begin{lemma}\label{lem-MZ}
The space \( M \) is contractible, and \( H \) and \( M \) may be given analytic structures, compatible with their topologies, such that the Lie group \( H \) acts on \( M \) analytically and transitively.
Moreover \( H \) is of polynomial growth.
\end{lemma}

\begin{proof}
Define \( \pi: H \to M \) by \( \pi h \coloneqq  ho \) and \( T:H\to H \) by \( Th\coloneqq \delta\circ h\circ\delta^{-1} \); then \( \pi\circ T = \delta\circ\pi \).
Let \( K \) be the maximal compact normal subgroup of \( H \).
Note that \( T(K)=K \), since \( T \) is an automorphism of \( H \).
Then \( \pi(K) \) is a compact subset of \( M \): let \( r\coloneqq\max\{d(o, p):p\in \pi(K)\} \).
Then
\[
\pi(K) = \pi T^{-1}(K) = \delta^{-1}\pi(K)\subseteq B(o, \lambda^{-1} r),
\]
which implies that \( r=0 \).
Therefore \( \pi(K)=\{o\} \), and \( K \) is contained in the stabiliser in \( H \) of the point \( o \) in \( M \);  by Remark \ref{rem:no-compact-normal-subgroups}, \( K = \{ e_H \} \).
By the Montgomery--Zippin structure theory (as in Theorem \ref{thm:GIMYZ} and Corollary  \ref{cor:isometry-group-is-Lie-group}), \( H \) and \( M \) may be given analytic structures, compatible with their topologies, such that \( M \) is a manifold and the action of \( H \) on \( M \) is analytic.

Since \( M \) is a manifold and admits a metric dilation, it is compactly contractible, and hence contractible by Lemma~\ref{lem:contractible}.
Since moreover \( M \) is doubling and proper by Lemma~\ref{lem-dilations}, it is of polynomial growth by Remark~\ref{rem01041721}.
By Lemma~\ref{lem03171107}, \( H \) is a group of polynomial growth.
\end{proof}

\begin{proof}[Proof of Theorem~\ref{thm:main-4}]
Let \( (M, d) \) be a homogeneous metric space.
Let \( \delta \) be a metric dilation of factor \( \lambda\in(1, \pinfty) \) and with fixed point \( o \).
Let \( H \) denote the connected component of the identity in \( \Iso(M ,d) \).
By Lemma~\ref{lem-MZ}, \( H \) is a Lie group of polynomial growth and hence is amenable, and \( M \) may be identified with \( H/K \), where \( K \) is the stabiliser of \( o \) in \( H \); further, \( M \) is contractible, so \( K \) is a maximal compact subgroup of \( H \) by Lemma \ref{lem:contractible}.

We may now apply Lemma \ref{lem:amenable-contractible}, and deduce that there exists a connected Lie subgroup \( G \) of \( H \) such that the restricted quotient map \( h\mapsto h(o) \) from \( G \) to \( M \) is a homeomorphism.
We use this homeomorphism to make \( G \) into a metrically self-similar Lie group isometric to \( (M, d) \).

Define the metric \( d_G \) on \( G \) by \( d_G(h, h') \coloneqq  d(h(o), h'(o)) \).
It is clear that this is an admissible metric, and it is left-invariant because
\[
d_G(hh', hh'') = d(h (h'(o)), h( h''(o))) = d(h'o, h''o) = d_G(h', h'')
\]
for all \( h, h', h''\in G \).
Further, define the map \( T \) on \( H \) by
\[
Tg \coloneqq  \delta\circ g\circ\delta^{-1} .
\]
Then \( T \) is a Lie group automorphism of \( H \).
Since \( TK=K \), Lemma \ref{lem:amenable-contractible} implies that \( TG=G \).
Thus \( T|_G \) is a Lie group automorphism of \( G \).

We note that after the identification of \( G \) with \( M \), the map \( T|_G \) coincides with \( \delta \).
Indeed,
\[
(Th)(o)
= (\delta h \delta^{-1})(o)
= \delta(ho) ,
\]
and the proof is complete.
\end{proof}

\raggedbottom
\newpage

% ----------------------------------------------------------------
 
\end{document}